\documentclass[11pt,letterpaper]{amsart}
\usepackage[foot]{amsaddr}

\usepackage{mathtools}
\usepackage{amsmath}
\usepackage[T1]{fontenc}
\usepackage[latin9]{inputenc}
\usepackage{geometry}
\usepackage{wrapfig}
\usepackage{multicol}
\usepackage{graphicx}
\usepackage{soul}
\usepackage{xcolor}
\usepackage{amssymb}
\usepackage{placeins}
\usepackage{bbm}
\usepackage{cite}
\usepackage{caption}
\usepackage{enumerate}
\usepackage{afterpage}
\usepackage{enumitem}
\usepackage{bmpsize}
\usepackage{hyperref}
\usepackage{tabu}
\usepackage{enumitem}
\numberwithin{equation}{section}
\usepackage{stmaryrd}
\usepackage{tikz}
\usetikzlibrary{matrix,graphs,arrows,positioning,calc,decorations.markings,shapes.symbols}

\makeatletter
\renewcommand{\email}[2][]{%
  \ifx\emails\@empty\relax\else{\g@addto@macro\emails{,\space}}\fi%
  \@ifnotempty{#1}{\g@addto@macro\emails{\textrm{(#1)}\space}}%
  \g@addto@macro\emails{#2}%
}
\makeatother

\newtheorem{theorem}{Theorem}[section]
\newtheorem{lemma}[theorem]{Lemma}
\newtheorem{proposition}[theorem]{Proposition}

{ \theoremstyle{definition}
\newtheorem{definition}[theorem]{Definition}}
{ \theoremstyle{remark}
\newtheorem{remark}[theorem]{Remark}}

\newcommand{\ice}{\mathsf{Inter}}

\title{Tightness for interlacing geometric random walk bridges}
\date{\today}
\author{Evgeni Dimitrov}

\begin{document}

\begin{abstract}
We investigate a class of line ensembles whose local structure is described by independent geometric random walk bridges, which have been conditioned to interlace with each other. The latter arise naturally in the context Schur processes, including their versions in a half-space and a finite interval with free or periodic boundary conditions. We show that under one-point tightness of the curves, these line ensembles are tight and any subsequential limit satisfies the Brownian Gibbs property. As an application of our tightness results, we show that sequences of spiked Schur processes, that were recently considered in \cite{ED24a}, converge uniformly over compact sets to the Airy wanderer line ensembles.
\end{abstract}

\maketitle

\tableofcontents

%------------------------------------------------------------------------------------------------------
%
% Section 1
%
%------------------------------------------------------------------------------------------------------
\section{Introduction and main results}\label{Section1}

%------------------------------------------------------------------------------------------------------
%
% Section 1.1
%
%------------------------------------------------------------------------------------------------------
\subsection{Introduction}\label{Section1.1} A {\em Gibbsian line ensemble} is a collection of independent labeled random walks or Brownian motions whose joint law is reweighed by a Radon-Nikodym derivative proportional to the exponential of the sum of local interaction energies between consecutively labeled curves. A simple example of a Gibbsian line ensemble is a collection of independent random walks, conditioned not to cross each other (e.g. level lines of random lozenge or domino tilings). In this case the local energy is infinity or zero depending on whether the crossing occurs or does not. A considerably more involved example (which arises under edge scaling of certain avoiding random walks) is given by the (parabolic) {\em Airy line ensemble}, which is a collection of random continuous functions that are globally parabolic and whose local distribution is described by independent Brownian bridges, conditioned on avoiding each other (this is called the {\em Brownian Gibbs property}). The Airy line ensemble was introduced in \cite{CorHamA} and it forms the foundation of the entire Kardar-Parisi-Zhang (KPZ) universality class through its role in the construction of the {\em Airy sheet} in \cite{DOV22}.

Over the last several years, there has been a substantial interest in finding general conditions that ensure that a sequence of Gibbsian line ensembles is tight, see e.g. \cite{CorHamA, CorHamK, DEA21, DW21, DNV23, S23, W23a, W23b}. In the earliest of these works, \cite{CorHamA}, establishing tightness was an important ingredient in the construction of the Airy line ensemble, namely in showing that this processes has a continuous version. In some of the subsequent works establishing tightness can be seen as an important step towards showing that a sequence of Gibbsian line ensembles converge uniformly over compact sets to the Airy line ensemble, which should be thought of as a stronger notion of convergence than in the finite-dimensional sense. As mentioned in the introduction of \cite{DNV23}, establishing this stronger notion of convergence is not merely cosmetic, but a crucial ingredient in showing that various models converge to the Airy sheet. The earliest result in this direction is \cite{DOV22}, see also \cite{ACH24} and the references therein. 

Showing that a sequence of line ensembles is tight ensures the existence of subsequential limits, and to prove uniform over compacts convergence one has to show that all of these limits are the same, i.e. establish uniqueness. One way to accomplish this is to prove finite-dimensional convergence, which is hard for many models due to the lack of exact formulas. Because of this reason, a general scheme for showing uniqueness of subsequential limits has been to establish that each subsequential limit has to satisfy enough properties that narrows down the list of possible candidates to exactly one. For example, in many of the aforementioned works the authors were able to show that any subsequential limit satisfies the Brownian Gibbs property, enjoyed by the Airy line ensemble. As shown in \cite{DM21}, any line ensemble possessing this property is uniquely characterized by the law of its lowest-indexed curve, and this characterization was instrumental, together with \cite{QS22}, in establishing the convergence of the {\em KPZ line ensemble} to the Airy line ensemble in \cite{W23a}. More recently, \cite{AH23} showed that the Airy line ensemble is essentially the unique line ensemble that possesses the Brownian Gibbs property and whose top curve is globally parabolic. The latter {\em strong characterization} of the Airy line ensemble is what enabled the results in \cite{ACH24}. 

The recent characterization results for Brownian Gibbsian line ensembles have placed even higher emphasis on finding general conditions that ensure that a sequence of line ensembles is tight and that any subsequential limit satisfies the Brownian Gibbs property, and the goal of this paper is to find such conditions for interlacing geometric random walk bridges.\\

In the present paper we deal with discrete line ensembles $\mathfrak{L} = \{L_i\}_{i = 1}^{\infty}$, where $L_i: \mathbb{Z} \rightarrow \mathbb{Z}$ satisfy
\begin{equation}\label{S1E1}
L_i(s) \leq L_{i} (s+1) \mbox{ for all } s \in \mathbb{Z}.
\end{equation}
In words, equation (\ref{S1E1}) says that one can think of $L_i$ as the trajectories of geometric random walkers. In addition, our ensemble $\mathfrak{L}$ satisfies the property that for each $s \in \mathbb{Z}$ the vectors $\mathfrak{L}(s) = (L_1(s), L_2(s), \dots)$ and $\mathfrak{L}(s+1) = (L_1(s+1), L_2(s+1), \dots)$ {\em interlace} in the sense that 
\begin{equation}\label{S1E2}
L_1(s+1) \geq L_1(s) \geq L_2(s+1) \geq L_2(s) \geq L_3(s+1) \geq L_3(s) \geq \cdots,
\end{equation}
and locally the distribution of $\mathfrak{L}$ is uniform subject to the interlacing inequalities in (\ref{S1E2}). The last statement will be made precise in Definition \ref{DefSGP}, and we refer to it as the {\em interlacing Gibbs property}. In words, $\mathfrak{L}$ satisfying the interlacing Gibbs property means that its law is locally described by independent geometric random walkers, which have been conditioned to interlace. By linearly interpolating the points $(s, L_i(s))$ we can view $\mathfrak{L}$ as a continuous line ensemble, see Figure \ref{S1_1}.

\begin{figure}[ht]
	\begin{center}
		\begin{tikzpicture}[scale=0.7]
		\begin{scope}
        \def\r{0.1}
		\draw[dotted, gray] (0,0) grid (8,6);

        % Curve L_1
        \draw[fill = black] (0,1) circle [radius=\r];
        \draw[fill = black] (1,4) circle [radius=\r];
        \draw[fill = black] (2,5) circle [radius=\r];
        \draw[fill = black] (3,5) circle [radius=\r];
        \draw[fill = black] (4,5) circle [radius=\r];
        \draw[fill = black] (5,5) circle [radius=\r];
        \draw[fill = black] (6,6) circle [radius=\r];
        \draw[fill = black] (7,6) circle [radius=\r];
        \draw[fill = black] (8,6) circle [radius=\r];
        \draw[-][black] (0,1) -- (1,4);
        \draw[-][black] (1,4) -- (2,5);
        \draw[-][black] (2,5) -- (3,5);
        \draw[-][black] (3,5) -- (4,5);
        \draw[-][black] (4,5) -- (5,5);
        \draw[-][black] (5,5) -- (6,6);
        \draw[-][black] (6,6) -- (7,6);
        \draw[-][black] (7,6) -- (8,6);

        % Curve L_2
        \draw[fill = black] (0,1) circle [radius=\r];
        \draw[fill = black] (1,1) circle [radius=\r];
        \draw[fill = black] (2,1) circle [radius=\r];
        \draw[fill = black] (3,3) circle [radius=\r];
        \draw[fill = black] (4,3) circle [radius=\r];
        \draw[fill = black] (5,3) circle [radius=\r];
        \draw[fill = black] (6,3) circle [radius=\r];
        \draw[fill = black] (7,4) circle [radius=\r];
        \draw[fill = black] (8,4) circle [radius=\r];
        \draw[-][black] (0,1) -- (1,1);
        \draw[-][black] (1,1) -- (2,1);
        \draw[-][black] (2,1) -- (3,3);
        \draw[-][black] (3,3) -- (4,3);
        \draw[-][black] (4,3) -- (5,3);
        \draw[-][black] (5,3) -- (6,3);
        \draw[-][black] (6,3) -- (7,4);
        \draw[-][black] (7,4) -- (8,4);

        % Curve L_3
        \draw[fill = black] (0,0) circle [radius=\r];
        \draw[fill = black] (1,0) circle [radius=\r];
        \draw[fill = black] (2,1) circle [radius=\r];
        \draw[fill = black] (3,1) circle [radius=\r];
        \draw[fill = black] (4,3) circle [radius=\r];
        \draw[fill = black] (5,3) circle [radius=\r];
        \draw[fill = black] (6,3) circle [radius=\r];
        \draw[fill = black] (7,3) circle [radius=\r];
        \draw[fill = black] (8,3) circle [radius=\r];
        \draw[-][black] (0,0) -- (1,0);
        \draw[-][black] (1,0) -- (2,1);
        \draw[-][black] (2,1) -- (3,1);
        \draw[-][black] (3,1) -- (4,3);
        \draw[-][black] (4,3) -- (5,3);
        \draw[-][black] (5,3) -- (6,3);
        \draw[-][black] (6,3) -- (7,3);
        \draw[-][black] (7,3) -- (8,3);

        \draw (8.5, 6) node{$L_1$};
        \draw (8.5, 4) node{$L_2$};
        \draw (8.5, 3) node{$L_3$};

		\end{scope}

		\end{tikzpicture}
	\end{center}
	\caption{The figure depicts the top three curves in $\mathfrak{L}$, satisfying (\ref{S1E1}) and (\ref{S1E2}).}
	\label{S1_1}
\end{figure}

The main result of the paper, Theorem \ref{S1T1} (which is a special case of the more general Theorem \ref{S5T} in the main text), gives general conditions under which a sequence of line ensembles that satisfy the interlacing Gibbs property is tight (under an appropriate diffuse scaling). We furthermore show that any subsequential limit satisfies the Brownian Gibbs property. I.e. in a diffuse scaling the local geometric random walk structure of the paths becomes that of Brownian motions (or rather bridges), and the local interlacing condition becomes a locally avoiding one.  

Part of the interest in ensembles that satisfy the interlacing Gibbs property comes from the fact that they arise in the Schur processes, introduced by Okounkov and Reshetikhin in \cite{OR03}. In particular, in Section \ref{Section6} we apply Theorem \ref{S1T1} to a sequence of spiked Schur processes, that were recently considered in \cite{ED24a}, and show that the latter converge uniformly over compact sets to the {\em Airy wanderer line ensembles} -- see Proposition \ref{S62P}. We mention that the Airy wanderer line ensembles were constructed for general parameters in \cite{ED24a} as weak limits of such ensembles for special parameters from \cite{AFM10, CorHamA}. By combining \cite[Theorem 1.8]{ED24a} and Proposition \ref{S62P} one obtains an alternative construction of these ensembles, directly as weak limits of Schur processes, which is independent of the results in \cite{AFM10, CorHamA}. In addition, the more general Theorem \ref{S5T} in the main text allows one to deal with line ensembles defined on a half-infinite or even a finite interval, which makes it applicable to a large class of models including the half-space and free boundary Schur processes from \cite{BBCS18, BBNV18}, and the periodic Schur processes from \cite{BB19, Bor07}. For example, the results of the present paper have already found applications to the analysis of half-space Schur processes away from the boundary in \cite{Z25} and in constructing half-space analogues of the Airy line ensemble in \cite{DY25}.

%------------------------------------------------------------------------------------------------------
%
% Section 1.2
%
%------------------------------------------------------------------------------------------------------
\subsection{Main result}\label{Section1.2} As mentioned in Section \ref{Section1.1} the main technical result of the paper is Theorem \ref{S5T}, which we forego stating until the main body of text after we have introduced some relevant notation in Section \ref{Section2}. In this section we state a special case of that result as Theorem \ref{S1T1}, which provides a good idea of the type of assumptions we make and consequences we can deduce, while being easier to formulate.

A {\em partition} is a sequence $\lambda = (\lambda_1, \lambda_2 ,\dots)$ of non-negative integers such that $\lambda_1 \geq \lambda_2 \geq \cdots$ and all but finitely many terms are zero. We denote the set of all partitions by $\mathbb{Y}$. We say that $\lambda, \mu \in \mathbb{Y}$ {\em interlace}, denoted by $\lambda \succeq \mu$ or $\mu \preceq \lambda$, if $\lambda_1 \geq \mu_1 \geq \lambda_2 \geq \mu_2 \geq \cdots$.  

We assume that $\mathfrak{L}^N$ is a sequence of random elements in $\mathbb{Y}^{\mathbb{Z}}$, where we write $\mathfrak{L}^N = \{\mathfrak{L}^N(s): s \in \mathbb{Z}\}$ and $\mathfrak{L}^N(s) = (L^N_1(s), L^N_2(s), \dots) \in \mathbb{Y}$. We make the following assumptions about this sequence.\\

{\bf \raggedleft Assumption 1. (One-point tightness)} There exists $p \in (0,\infty)$, $\sigma = \sqrt{p (1+p)}$ and sequences $d_N \in (0, \infty), C_N \in \mathbb{R}$ with $\lim_N d_N = \infty$ such that for each $t \in \mathbb{R}$ and $i \in \mathbb{N}$ the sequence of random variables $\sigma^{-1} d_N^{-1/2} (L_i(\lfloor t d_N \rfloor) - p td_N - C_N)$ is tight.\\

{\bf \raggedleft Assumption 2. (Interlacing Gibbs property)} There are sequences of integers $\hat{A}_N, \hat{B}_N \in \mathbb{Z}$ with $\hat{A}_N \leq \hat{B}_N$, such that $\lim_N d_N^{-1} \hat{A}_N = -\infty$, $\lim_N d_N^{-1} \hat{B}_N = \infty$, and for all $\lambda^{\hat{A}_N}, \lambda^{\hat{A}_N+1}, \dots, \lambda^{\hat{B}_N} \in \mathbb{Y}$ 
\begin{equation*}
\begin{split}
&\mathbb{P}\left(\mathfrak{L}^N(\hat{A}_N)= \lambda^{\hat{A}_N}, \mathfrak{L}^N(\hat{A}_N + 1) = \lambda^{\hat{A}_N+1}, \dots,  \mathfrak{L}^N(\hat{B}_N) = \lambda^{\hat{B}_N} \right)\\
& = \mathbb{P} \left( \mathfrak{L}^N(\hat{A}_N) = \lambda^{\hat{A}_N}, \mathfrak{L}^N(\hat{B}_N) = \lambda^{\hat{B}_N} \right) \cdot {\bf 1} \{ \lambda^{\hat{A}_N} \preceq \lambda^{\hat{A}_N + 1} \preceq \cdots \preceq \lambda^{\hat{B}_N} \}.
\end{split}
\end{equation*}
In plain words, the last equation states that the conditional law of $\{ \mathfrak{L}^N(s): s = \hat{A}_N+1, \dots, \hat{B}_N-1\}$, given that $\mathfrak{L}^N(\hat{A}_N) = \lambda^{\hat{A}_N}, \mathfrak{L}^N(\hat{B}_N) = \lambda^{\hat{B}_N}$, is uniform on $\mathbb{Y}^{\hat{B}_N - \hat{A}_N - 1}$ subject to interlacing.\\

We can now formulate our main result about $\mathfrak{L}^N$. 
\begin{theorem}\label{S1T1} Suppose that $\mathfrak{L}^N$ is a sequence of random elements in $\mathbb{Y}^{\mathbb{Z}}$ that satisfies Assumptions 1 and 2 above. Define the sequence $\mathcal{L}^N = \{\mathcal{L}_i^N\}_{i = 1}^{\infty} \in C(\mathbb{N} \times \mathbb{R})$ by
$$\mathcal{L}^N_i(t) = \sigma^{-1} d_N^{-1/2} \cdot \left( L_i^N(t d_N) - p td_N - C_N\right), $$
where $L_i^N(s)$ is defined for non-integer $s$ by linear interpolation. Then, $\mathcal{L}^N$ is a tight sequence of random elements in $C(\mathbb{N} \times \mathbb{R})$. Moreover, any subsequential limit satisfies the Brownian Gibbs property from \cite{CorHamA}, see also \cite[Definition 2.5]{DM21}.
\end{theorem}
\begin{remark} We mention that $C(\mathbb{N} \times \mathbb{R})$ is the space of continuous functions $f: \mathbb{N} \times \mathbb{R} \rightarrow \mathbb{R}$ with the topology of uniform convergence over compact sets, see Section \ref{Section2.1}.
\end{remark}
\begin{remark} In plain words Theorem \ref{S1T1} says that in the presence of the interlacing Gibbs property (Assumption 2), one point tightness for a sequence of line ensembles (Assumption 1) implies tightness of the sequence as random elements in $C(\mathbb{N} \times \mathbb{R})$. Moreover, the interlacing Gibbs property in the limit becomes the Brownian Gibbs property.
\end{remark}

The proof of Theorem \ref{S1T1} is given right after Theorem \ref{S5T} in Section \ref{Section5}, and is an easy corollary of the latter statement. Specifically, Theorem \ref{S5T} generalizes Theorem \ref{S1T1} by allowing $\mathfrak{L}^N$ to possibly have finitely many curves, and by allowing $d_N^{-1}[ \hat{A}_N, \hat{B}_N]$ to approximate an arbitrary interval $\Lambda \subseteq \mathbb{R}$. The ability to deal with general $\Lambda$ (and not just $\Lambda = \mathbb{R}$) is important for applications to the half-space Schur processes in \cite{BBCS18, BBNV18}, where $\Lambda = (0, \infty)$, and the periodic Schur processes from \cite{BB19, Bor07}, where $\Lambda = (0,L)$ for some $L > 0$.\\

Theorem \ref{S1T1}, and more generally Theorem \ref{S5T}, is similar to \cite[Theorem 3.8]{CorHamA}, \cite[Theorem 1.1]{DEA21}, and \cite[Theorem 1.1]{S23}. In \cite{DEA21} the authors dealt with avoiding Bernoulli random walk bridges, and \cite{S23} extended the argument to a class of avoiding random walkers with general jump distributions, while \cite{CorHamA} worked directly with avoiding Brownian line ensembles. One of the differences between our setup and \cite{DEA21, S23} is that our walkers are interlacing rather than avoiding -- this is quite mild and does not affect the problem substantially. A more serious difference is that in \cite{DEA21, S23} the authors only worked on $\mathbb{R}$, assumed that the top curves $\mathcal{L}_1^N$ have one-point tight marginals, which are globally parabolic. This global parabolicity is what allowed them to conclude that all the other curves also need to have tight one-point marginals. In Theorem \ref{S5T} we want to deal with possibly finite intervals $\Lambda$, for which one needs more than just information about the top curve. Specifically, it is possible for all curves except the first ones in our sequence of ensembles to dip to $-\infty$, and the top curves to converge to a Brownian bridge on $\Lambda$. Consequently, one-point tightness of the top curves does not prevent other curves from escaping to $-\infty$. This is why we require one-point tightness of {\em all} the curves and not just the top one. In this regard our assumptions are closer to those in \cite{CorHamA}.

One thing that we do differently from \cite{CorHamA} is that we do not assume that the curves $\mathcal{L}^N_i$ separate from each other at a fixed time $t$ with high probability, see Hypothesis H3 in \cite[Definition 3.3]{CorHamA}. Instead, we show that this likely separation of the curves occurs by itself under only the one-point tightness assumption. The separation of the curves is an important ingredient in establishing tightness for $\mathcal{L}^N$, since it allows one to make a good comparison between interlacing random walk bridges and free (i.e. not conditioned on interlacing) random walk bridges. In particular, one can show that any event that is unlikely for the free bridges is also unlikely for the interlacing ones. The way this is used in our proofs is as follows. To show tightness, we need to show that for each $i$, the random curves $\mathcal{L}^N_i$ likely have well-behaved moduli of continuity. Being able to compare interlacing with free random walk bridges, allows us to reduce this statement to showing that a sequence of free geometric random walk bridges have well-behaved moduli of continuity. The latter statement can then be established by comparing a geometric random walk bridge with a Brownian bridge using the strong coupling from \cite{DW19}, which we use in establishing Proposition \ref{KMT}. We mention that our arguments do not rely on the full strength of the results in \cite{DW19}, and even a mild comparison between our random walk bridges and Brownian bridges suffices for us.

The likely curve separation was established in \cite{DEA21} through detailed (and rather involved) exact computations for avoiding Bernoulli random walkers, coming from Schur symmetric functions. In \cite{S23} this separation was established by a careful comparison with the Bernoulli case, see \cite[Lemma 3.12]{S23}. Our approach is considerably simpler and free of exact computations. Specifically, it only relies on a mild comparison of our random walk bridges with Brownian bridges, and the fact that we can monotonically couple interlacing random walk bridges in their boundary. Between the two, the monotone coupling, see Lemma \ref{MCL} for a precise statement, is the more restrictive but it is known to occur in many contexts including the aforementioned works \cite{CorHamA, DEA21, S23}. The crux of the separation argument is in the proof of Lemma \ref{S41L}, and the well-behaved nature of the modulus of continuity under separation is Lemma \ref{S42L}. We believe that the latter two statements and their proofs can be readily generalized to a large class of line ensembles, which satisfy the monotone coupling. This would remove the hypothesis of likely curve separation (which can be quite challenging to verify for many models), and also eliminate the need for exact computations.

%-----------------------------------------------------------------------------------------------------
%
% Section 1.3
%
%-----------------------------------------------------------------------------------------------------
\subsection{Outline}\label{Section1.3} In Section \ref{Section2} we introduce the general class of models we study, called {\em geometric line ensembles}, and their local interlacing structure, called the {\em interlacing Gibbs property}. The section contains most of the core notation for the rest of the paper, as well as the monotone coupling, see Lemma \ref{MCL}; the comparison to Brownian bridges, see Proposition \ref{KMT}; and the statement that under diffuse scaling interlacing geometric random walk bridges converge to avoiding Brownian bridges, see Lemma \ref{lem:RW}. The proofs of Lemmas \ref{MCL} and \ref{lem:RW} are given in Appendix \ref{AppendixA}. Section \ref{Section3} provides several technical estimates for a single geometric random walk bridge, which is either free, or conditioned to stay above a well-behaved lower boundary. The proofs in Section \ref{Section3} rely mostly on the monotone coupling in Lemma \ref{MCL} and the bridge comparison in Proposition \ref{KMT}. Section \ref{Section4} contains the critical Lemma \ref{S41L}, which ensures curve separation, and Lemma \ref{S42L}, which provides modulus of continuity estimates under the assumption of curve separation. Section \ref{Section5} contains the most general version of our results, see Theorem \ref{S5T}. In the same section we deduce Theorem \ref{S1T1} from Theorem \ref{S5T} and provide a proof of the latter. In Section \ref{Section6} we apply Theorem \ref{S1T1} to sequences of line ensembles that arises in spiked Schur processes, and show that the latter converge to the Airy wanderer line ensembles, see Proposition \ref{S62P}.

%-----------------------------------------------------------------------------------------------------
%
% Section 1.4
%
%-----------------------------------------------------------------------------------------------------
\subsection*{Acknowledgments}
The author would like to thank Alexei Borodin and Ivan Corwin for many inspiring conversations. This work has been partially supported by the NSF grant DMS:2230262.

%-----------------------------------------------------------------------------------------------------
%
% Section 2
%
%-----------------------------------------------------------------------------------------------------
\section{Basic definitions and results}\label{Section2} In this section we introduce the notions of a {\em geometric line ensemble} and the {\em interlacing Gibbs property}. Some of our notation is based on \cite[Section 2]{DEA21}, which in turn goes back to \cite{CorHamA}. We also state several basic properties for these ensembles, including a monotone coupling lemma and a strong coupling of such ensembles with a single curve to Brownian bridges of appropriate variance.

%-----------------------------------------------------------------------------------------------------
%
% Section 2.1
%
%-----------------------------------------------------------------------------------------------------
\subsection{Geometric line ensembles}\label{Section2.1} Given two integers $a \leq b$, we let $\llbracket a, b \rrbracket$ denote the set $\{a, a+1, \dots, b\}$. We also set $\llbracket a,b \rrbracket = \emptyset$ when $a > b$, $\llbracket a, \infty \rrbracket = \{a, a+1, a+2 , \dots \}$, $\llbracket - \infty, b\rrbracket = \{b, b-1, b-2, \dots\}$ and $\llbracket - \infty, \infty \rrbracket = \mathbb{Z}$. Given an interval $\Lambda \subset \mathbb{R}$, we endow it with the subspace topology of the usual topology on $\mathbb{R}$. We let $(C(\Lambda), \mathcal{C})$ denote the space of continuous functions $f: \Lambda \rightarrow \mathbb{R}$ with the topology of uniform convergence over compacts, see \cite[Chapter 7, Section 46]{Munkres}, and Borel $\sigma$-algebra $\mathcal{C}$. Given a set $\Sigma \subset \mathbb{Z}$, we endow it with the discrete topology and denote by $\Sigma \times \Lambda$ the set of all pairs $(i,x)$ with $i \in \Sigma$ and $x \in \Lambda$ with the product topology. We also denote by $\left(C (\Sigma \times \Lambda), \mathcal{C}_{\Sigma}\right)$ the space of real-valued continuous functions on $\Sigma \times \Lambda$ with the topology of uniform convergence over compact sets and Borel $\sigma$-algebra $\mathcal{C}_{\Sigma}$. We typically take $\Sigma = \llbracket 1, N \rrbracket$ with $N \in \mathbb{N} \cup \{\infty\}$.

The following defines the notion of a line ensemble.
\begin{definition}\label{CLEDef}
Let $\Sigma \subseteq \mathbb{Z}$ and $\Lambda \subseteq \mathbb{R}$ be an interval. A {\em $\Sigma$-indexed line ensemble $\mathcal{L}$} is a random variable defined on a probability space $(\Omega, \mathcal{F}, \mathbb{P})$ that takes values in $\left(C (\Sigma \times \Lambda), \mathcal{C}_{\Sigma}\right)$. Intuitively, $\mathcal{L}$ is a collection of random continuous curves (sometimes referred to as {\em lines}), indexed by $\Sigma$,  each of which maps $\Lambda$ in $\mathbb{R}$. We will often slightly abuse notation and write $\mathcal{L}: \Sigma \times \Lambda \rightarrow \mathbb{R}$, even though it is not $\mathcal{L}$ which is such a function, but $\mathcal{L}(\omega)$ for every $\omega \in \Omega$. For $i \in \Sigma$ we write $\mathcal{L}_i(\omega) = (\mathcal{L}(\omega))(i, \cdot)$ for the curve of index $i$ and note that the latter is a map $\mathcal{L}_i: \Omega \rightarrow C(\Lambda)$, which is $(\mathcal{C}, \mathcal{F})-$measurable. If $a,b \in \Lambda$ satisfy $a \leq b$ we let $\mathcal{L}_i[a,b]$ denote the restriction of $\mathcal{L}_i$ to $[a,b]$.
\end{definition}
\begin{remark}\label{RemPolish} As shown in \cite[Lemma 2.2]{DEA21}, we have that $C(\Sigma \times \Lambda)$ is a Polish space, and so a line ensemble $\mathcal{L}$ is just a random element in $C(\Sigma \times \Lambda)$ in the sense of \cite[Section 3]{Bill}.
\end{remark}

\begin{definition}\label{DefDLE}
Let $\Sigma \subseteq \mathbb{Z}$ and $\llbracket T_0, T_1 \rrbracket$ be a non-empty integer interval in $\mathbb{Z}$. Consider the set $Y$ of functions $f: \Sigma \times \llbracket T_0, T_1 \rrbracket \rightarrow \mathbb{Z}$ such that $f(i, j+1) - f(i,j) \in \mathbb{Z}_{\geq 0}$ when $i \in \Sigma$ and $j \in\llbracket T_0, T_1 -1 \rrbracket$ and let $\mathcal{D}$ denote the discrete $\sigma$-algebra on $Y$. We call a function $g: \llbracket T_0, T_1 \rrbracket \rightarrow \mathbb{Z}$ such that $f( j+1) - f(j) \in \mathbb{Z}_{\geq 0}$ when $j \in\llbracket T_0, T_1 -1 \rrbracket$  an {\em increasing path} and elements in $Y$ {\em collections of increasing paths}. A $\Sigma$-{\em indexed geometric line ensemble $\mathfrak{L}$ on $\llbracket T_0, T_1 \rrbracket$}  is a random variable defined on a probability space $(\Omega, \mathcal{B}, \mathbb{P})$, taking values in $Y$ such that $\mathfrak{L}$ is a $(\mathcal{B}, \mathcal{D})$-measurable function. Unless otherwise specified, we will assume that $T_0 \leq T_1$ are both integers, although the above definition makes sense if $T_0 = -\infty$, or $T_1 = \infty$, or both.
\end{definition}

\begin{remark} The condition $f(i, j+1) - f(i,j) \in \mathbb{Z}_{\geq 0}$ when $i \in \Sigma$ and $j \in\llbracket T_0, T_1 -1 \rrbracket$ essentially means that for each $i \in \Sigma$ the function $f(i, \cdot)$ can be thought of as the trajectory of a geometric random walk from time $T_0$ to time $T_1$. Here, and throughout the paper a geometric random variable $X$ with parameter $q \in (0,1)$ has probability mass function $\mathbb{P}(X = k) = (1-q)q^k$ for $k \in \mathbb{Z}_{\geq 0}$.
\end{remark}

We think of geometric line ensembles as collections of random increasing paths on the integer lattice, indexed by $\Sigma$. Observe that one can view an increasing path $L$ on $\llbracket T_0, T_1 \rrbracket$ as a continuous curve by linearly interpolating the points $(i, L(i))$, see Figure \ref{S1_1}. This allows us to define $ (\mathfrak{L}(\omega)) (i, s)$ for non-integer $s \in [T_0,T_1]$ and to view geometric line ensembles as line ensembles in the sense of Definition \ref{CLEDef}. In particular, we can think of $\mathfrak{L}$ as a random element in $C (\Sigma \times \Lambda)$ with $\Lambda = [T_0, T_1]$.

We will often slightly abuse notation and write $\mathfrak{L}: \Sigma \times \llbracket T_0, T_1 \rrbracket \rightarrow \mathbb{Z}$, even though it is not $\mathfrak{L}$ which is such a function, but rather $\mathfrak{L}(\omega)$ for each $\omega \in \Omega$. Furthermore we write $L_i = (\mathfrak{L}(\omega)) (i, \cdot)$ for the index $i \in \Sigma$ path. If $L$ is an increasing path on $\llbracket T_0, T_1 \rrbracket$ and $a, b \in \llbracket T_0, T_1 \rrbracket$ satisfy $a \leq b$ we let $L\llbracket a, b \rrbracket$ denote the restriction of $L$ to $\llbracket a,b\rrbracket$. \\

Let $t_i, z_i \in \mathbb{Z}$ for $i = 1,2$ be given such that $t_1 \leq t_2$ and $z_1 \leq z_2$. We denote by $\Omega(t_1,t_2,z_1,z_2)$ the collection of increasing paths that start from $(t_1,z_1)$ and end at $(t_2,z_2)$, by $\mathbb{P}_{\operatorname{Geom}}^{t_1,t_2, z_1, z_2}$ the uniform distribution on $\Omega(t_1,t_2,z_1,z_2)$ and write $\mathbb{E}^{t_1,t_2,z_1,z_2}_{\operatorname{Geom}}$ for the expectation with respect to this measure. One thinks of the distribution $\mathbb{P}_{\operatorname{Geom}}^{t_1,t_2, z_1, z_2}$ as the law of a simple random walk with i.i.d. geometric increments with parameter $q \in (0,1)$ that starts from $z_1$ at time $t_1$ and is conditioned to end in $z_2$ at time $t_2$ -- this interpretation does not depend on the choice of $q \in (0,1)$. Notice that by our assumptions on the parameters the state space $\Omega(t_1,t_2,z_1,z_2)$ is non-empty.  

Given $k \in \mathbb{N}$, $T_0, T_1 \in \mathbb{Z}$ with $T_0 < T_1$ and $\vec{x}, \vec{y} \in \mathbb{Z}^k$ we let $\mathbb{P}^{T_0,T_1, \vec{x},\vec{y}}_{\operatorname{Geom}}$ denote the law of $k$ independent geometric bridges $\{B_i: \llbracket T_0, T_1 \rrbracket  \rightarrow \mathbb{Z} \}_{i = 1}^k$ from $B_i(T_0) = x_i$ to $B_i(T_1) = y_i$. Equivalently, this is the law of $k$ independent random increasing paths $B_i \in \Omega(T_0,T_1,x_i,y_i)$ for $i \in \llbracket 1, k \rrbracket$ that are uniformly distributed or just the uniform measure on 
$$\Omega_{\operatorname{Geom}}(T_0, T_1, \vec{x}, \vec{y}) = \Omega(T_0,T_1,x_1,y_1) \times \cdots \times \Omega(T_0,T_1,x_k,y_k).$$
This measure is well-defined provided that $\Omega(T_0,T_1,x_i,y_i)$ are non-empty for $i \in \llbracket 1, k \rrbracket$, which holds if $y_i \geq x_i$ for all $i \in \llbracket 1, k \rrbracket$.

The following definition introduces the notion of an $(f,g)$-interlacing geometric line ensemble, which in simple terms is a collection of $k$ independent geometric bridges, conditioned on interlacing with each other and the graphs of two functions $f$ and $g$.
\begin{definition}\label{DefAvoidingLawBer}
Let $k \in \mathbb{N}$ and $\mathfrak{W}_k$ denote the set of signatures of length $k$, i.e.
\begin{equation}\label{DefSig}
\mathfrak{W}_k = \{ \vec{x} = (x_1, \dots, x_k) \in \mathbb{Z}^k: x_1 \geq  x_2 \geq  \cdots \geq  x_k \}.
\end{equation}
Let $\vec{x}, \vec{y} \in \mathfrak{W}_k$, $T_0, T_1 \in \mathbb{Z}$ with $T_0 < T_1$, and $f: \llbracket T_0, T_1 \rrbracket \rightarrow (-\infty, \infty]$ and $g: \llbracket T_0, T_1 \rrbracket \rightarrow [-\infty, \infty)$ be two functions. With the above data we define the {\em $(f,g)$-interlacing geometric line ensemble on the interval $\llbracket T_0, T_1 \rrbracket$ with entrance data $\vec{x}$ and exit data $\vec{y}$} to be the $\Sigma$-indexed geometric line ensemble $\mathfrak{Q}$ with $\Sigma = \llbracket 1, k\rrbracket$ on $\llbracket T_0, T_1 \rrbracket$ and with the law of $\mathfrak{Q}$ equal to $\mathbb{P}^{T_0,T_1, \vec{x},\vec{y}}_{\operatorname{Geom}}$ (the law of $k$ independent uniform increasing paths $\{B_i: \llbracket T_0, T_1 \rrbracket \rightarrow \mathbb{Z} \}_{i = 1}^k$ from $B_i(T_0) = x_i$ to $B_i(T_1) = y_i$), conditioned on 
\begin{equation}\label{EventInter}
\begin{split}
\ice  = &\left\{ B_i(r-1) \geq B_{i+1}(r)  \mbox{ for all $r \in \llbracket T_0 + 1, T_1 \rrbracket$ and $i \in \llbracket 0 , k \rrbracket$} \right\},
\end{split}
\end{equation}
with the convention that $B_0(x) = f(x)$ and $B_{k+1}(x) = g(x)$.

The above definition is well-posed if there exist $B_i \in \Omega(T_0,T_1,x_i,y_i)$ for $i \in \llbracket 1, k \rrbracket$ that satisfy the conditions in $\ice$. We denote by $\Omega_{\ice}(T_0, T_1, \vec{x}, \vec{y}, f,g)$ the set of collections of $k$ increasing paths that satisfy the conditions in $\ice$ and then the distribution of $\mathfrak{Q}$ is simply the uniform measure on $\Omega_{\ice}(T_0, T_1, \vec{x}, \vec{y}, f,g)$. We denote the probability distribution of $\mathfrak{Q}$ as $\mathbb{P}_{\ice, \operatorname{Geom}}^{T_0,T_1, \vec{x}, \vec{y}, f, g}$ and write $\mathbb{E}_{\ice, \operatorname{Geom}}^{T_0, T_1, \vec{x}, \vec{y}, f, g}$ for the expectation with respect to this measure. If $f=+\infty$ and $g=-\infty$, we drop them from the notation and simply write $\Omega_{\ice}(T_0,T_1,\vec{x},\vec{y})$, $\mathbb{P}^{T_0, T_1, \vec{x},\vec{y}}_{\ice, \operatorname{Geom}}$, and $\mathbb{E}^{T_0, T_1, \vec{x},\vec{y}}_{\ice, \operatorname{Geom}}$.
\end{definition}

The following definition introduces the notion of the interlacing Gibbs property.
\begin{definition}\label{DefSGP}
Fix a set $\Sigma = \llbracket 1, N \rrbracket$ with $N \in \mathbb{N}$ or $N = \infty$ and $T_0, T_1\in \mathbb{Z}$ with $T_0 \leq T_1$. A $\Sigma$-indexed geometric line ensemble $\mathfrak{L} : \Sigma \times \llbracket T_0, T_1 \rrbracket \rightarrow \mathbb{Z}$ is said to satisfy the {\em interlacing Gibbs property} if it is interlacing, meaning that 
$$ L_i(j-1) \geq L_{i+1}(j) \mbox{ for all $i \in \llbracket 1, N - 1 \rrbracket$ and $j \in \llbracket T_0 + 1, T_1 \rrbracket$},$$
and for any finite $K = \llbracket k_1, k_2 \rrbracket \subseteq \llbracket 1, N - 1 \rrbracket$ and $a,b \in \llbracket T_0, T_1 \rrbracket$ with $a < b$ the following holds.  Suppose that $f, g$ are two increasing paths drawn in $\{ (r,z) \in \mathbb{Z}^2 : a \leq r \leq b\}$ and $\vec{x}, \vec{y} \in \mathfrak{W}_k$ with $k=k_2-k_1+1$ altogether satisfy that $\mathbb{P}(A) > 0$ where $A$ denotes the event $$A =\{ \vec{x} = ({L}_{k_1}(a), \dots, {L}_{k_2}(a)), \vec{y} = ({L}_{k_1}(b), \dots, {L}_{k_2}(b)), L_{k_1-1} \llbracket a,b \rrbracket = f, L_{k_2+1} \llbracket a,b \rrbracket = g \},$$
where if $k_1 = 1$ we adopt the convention $f = \infty = L_0$. Then, we have for any $\{ B_i \in \Omega(a, b, x_i , y_i) \}_{i = 1}^k$ that
\begin{equation}\label{SchurEq}
\mathbb{P}\left( L_{i + k_1-1}\llbracket a,b \rrbracket = B_{i} \mbox{ for $i \in \llbracket 1, k \rrbracket$} \, \vert  A \, \right) = \mathbb{P}_{\ice, \operatorname{Geom}}^{a,b, \vec{x}, \vec{y}, f, g} \left( \cap_{i = 1}^k\{ Q_i = B_i \} \right).
\end{equation}
\end{definition}
\begin{remark}\label{RemSGB} In simple words, a geometric line ensemble is said to satisfy the interlacing Gibbs property if the distribution of any finite number of consecutive paths, conditioned on their end-points and the paths above and below them is simply the uniform measure on all collections of increasing paths that have the same end-points and interlace with each other and the two paths above and below. 
\end{remark}

\begin{remark}\label{RemSGB2} Observe that in Definition \ref{DefSGP} the index $k_2$ is assumed to be less than or equal to $N-1$, so that if $N < \infty$ the $N$-th path is special and is not conditionally uniform. If $N = 1$, then the conditions in Definition \ref{DefSGP} become void, i.e., any geometric line ensemble with one line satisfies the geometric Gibbs property. We mention that since $\mathfrak{L}$ is interlacing we automatically have for each $j \in \llbracket T_0, T_1\rrbracket$ that $L_1(j) \geq L_2(j) \geq  \cdots$. In addition, the well-posedness of $\mathbb{P}_{\ice, \operatorname{Geom}}^{T_0,T_1, \vec{x}, \vec{y}, f, g}$ in (\ref{SchurEq}) is a consequence of our assumption that $\mathbb{P}(A) > 0$.
\end{remark}

\begin{remark} \label{restrict}  An immediate consequence of Definition \ref{DefSGP} is that if $M \leq N$, we have that the induced law on $\{L_i\}_{i = 1}^M$ also satisfies the interlacing Gibbs property as an $\llbracket 1, M \rrbracket$-indexed geometric line ensemble on $\llbracket T_0, T_1 \rrbracket$.
\end{remark}

We end this section with the following two results. The first, Lemma \ref{Lem.FinEnsSatGB}, shows that line ensembles with law $\mathbb{P}_{\ice, \operatorname{Geom}}^{T_0,T_1, \vec{x}, \vec{y}, \infty, g}$ as in Definition \ref{DefAvoidingLawBer} satisfy the interlacing Gibbs property. The second, Lemma \ref{Lem.StrongGP}, provides a stronger property than (\ref{SchurEq}) enjoyed by line ensembles that satisfy the interlacing Gibbs property. This condition is automatically satisfied; however, it is a bit harder to check in practice, which is why we formulated the interlacing Gibbs property using (\ref{SchurEq}). Both lemmas are proved in Appendix \ref{AppendixA3}.
\begin{lemma}\label{Lem.FinEnsSatGB} Let $\vec{x}, \vec{y} \in \mathfrak{W}_k$, $T_0, T_1 \in \mathbb{Z}$ with $T_0 < T_1$, and $g: \llbracket T_0, T_1 \rrbracket \rightarrow \mathbb{Z}$ be an increasing path. Suppose that $\Omega_{\ice}(T_0, T_1, \vec{x}, \vec{y}, \infty,g) \neq \emptyset$ and $\mathfrak{Q}$ has law $\mathbb{P}_{\ice, \operatorname{Geom}}^{T_0,T_1, \vec{x}, \vec{y}, \infty, g}$ as in Definition \ref{DefSGP}. Define the $\llbracket 1, k+1 \rrbracket$-indexed geometric line ensemble $\mathfrak{L}: \llbracket 1, k+1 \rrbracket \times \llbracket T_0, T_1 \rrbracket \rightarrow \mathbb{Z}$ through
$$\mathfrak{L}(i,j) = \mathfrak{Q}(i,j) \mbox{ for } i \in \llbracket 1, k \rrbracket, j \in \llbracket T_0, T_1 \rrbracket, \mbox{ and } \mathfrak{L}(k+1,j) = g(j) \mbox{ for } j \in \llbracket T_0, T_1 \rrbracket.$$
Then, $\mathfrak{L}$ satisfies the interlacing Gibbs property from Definition \ref{DefSGP} with $N = k+1$. 
\end{lemma}

\begin{lemma}\label{Lem.StrongGP} Assume the same notation as in Definition \ref{DefSGP}. Then, for any $\{ B_i \in \Omega(a, b, x_i , y_i) \}_{i = 1}^k$
\begin{equation}\label{SchurEqV2}
{\bf 1}_{A} \cdot \mathbb{P}\left( L_{i + k_1-1}\llbracket a,b \rrbracket = B_{i} \mbox{ for $i \in \llbracket 1, k \rrbracket$} \, \vert  \mathcal{F}_{\mathrm{ext}} \right) = {\bf 1}_{A} \cdot \mathbb{P}_{\ice, \operatorname{Geom}}^{a,b, \vec{x}, \vec{y}, f, g} \left( \cap_{i = 1}^k\{ Q_i = B_i \} \right),
\end{equation}
where $\mathcal{F}_{\mathrm{ext}} = \sigma(L_i(j): (i,j) \in \Sigma \times \llbracket T_0, T_1 \rrbracket \setminus \llbracket k_1, k_2 \rrbracket \times \llbracket a+1, b-1 \rrbracket)$.
\end{lemma}

%-----------------------------------------------------------------------------------------------------
%
% Section 2.2
%
%-----------------------------------------------------------------------------------------------------
\subsection{Some properties}\label{Section2.2} In this section we state some of the properties of line ensembles with law $\mathbb{P}_{\ice, \operatorname{Geom}}^{T_0,T_1, \vec{x}, \vec{y}, f, g}$ as in Definition \ref{DefAvoidingLawBer}. The first result is a lemma that shows that these ensembles can be monotonically coupled in their boundary data. Its proof is given in Appendix \ref{AppendixA1}. 

\begin{lemma}\label{MCL} Assume the same notation as in Definition \ref{DefAvoidingLawBer}. Fix $k \in \mathbb{N}$, $T_0, T_1 \in \mathbb{Z}$ with $T_0 < T_1$, two functions $g^t, g^b: \llbracket T_0, T_1 \rrbracket  \rightarrow [-\infty, \infty)$ as well as $\vec{x}\,^b, \vec{y}\,^b, \vec{x}\,^t, \vec{y}\,^t \in \mathfrak{W}_k$. We assume that 
$$g^t(r) \geq g^b(r) \mbox{ for } r \in \llbracket T_0, T_1 \rrbracket \mbox{, } x_i^t \geq x^b_i \mbox{, and } y_i^t \geq y_i^b \mbox{ for } i = 1, \dots, k,$$
and also that $\Omega_{\ice}(T_0, T_1, \vec{x}\,^b, \vec{y}\,^b, \infty,g^b)$ and $\Omega_{\ice}(T_0, T_1, \vec{x}\,^t, \vec{y}\,^t, \infty,g^t)$ are both non-empty. Then, there exists a probability space $(\Omega, \mathcal{F}, \mathbb{P})$, which supports two $\llbracket 1, k \rrbracket$-indexed geometric line ensembles $\mathfrak{Q}^t$ and $\mathfrak{Q}^b$ on $\llbracket T_0, T_1 \rrbracket$ such that the law of $\mathfrak{Q}^{t}$ {\big (}resp. $\mathfrak{Q}^b${\big )} under $\mathbb{P}$ is given by $\mathbb{P}_{\ice,\operatorname{Geom}}^{T_0, T_1, \vec{x}\,^t, \vec{y}\,^t, \infty, g^t}$ {\big (}resp. $\mathbb{P}_{\ice,\operatorname{Geom}}^{T_0, T_1, \vec{x}\,^b, \vec{y}\,^b, \infty, g^b}${\big )} and such that $\mathbb{P}$-almost surely $Q_i^t(r) \geq Q^b_i(r)$ for all $i \in \llbracket 1, k \rrbracket$ and $r \in \llbracket T_0, T_1 \rrbracket$.
\end{lemma}

We next state a strong coupling between geometric random walk bridges and Brownian bridges, which follows from the results in \cite{DW19}. If $W_t$ denotes a standard one-dimensional Brownian motion and $\sigma > 0$, then the process
\begin{equation}\label{S22E1}
B^{\sigma}_t = \sigma (W_t - t W_1), \hspace{5mm} 0 \leq t \leq 1,
\end{equation}
is called a {\em Brownian bridge (from $B^{\sigma}_0 = 0$ to $B_1^{\sigma} = 0$) with variance $\sigma^2$.} When $\sigma^2 =1$ we drop it from the notation in (\ref{S22E1}) and refer to the latter process as a {\em standard Brownian bridge}. 

With the above notation we state the strong coupling result we use.
\begin{proposition}\label{KMT} Fix $p \in (0, \infty)$ and set $\sigma = \sqrt{p(p+1)}$. For any $\epsilon, A > 0$ we can find $N_0$ large enough, depending on $p, \epsilon, A $, such that the following holds. For any $n \geq N_0$ and $z \in \mathbb{Z}_{\geq 0}$ such that $|z - pn| \leq A \sqrt{n}$ we can find a probability space that supports a Brownian bridge $B^{\sigma}$ with variance $\sigma^2$ and a random increasing path $\ell^{(n,z)}$ with law $\mathbb{P}_{\operatorname{Geom}}^{0, n, 0,z}$ such that 
\begin{equation}\label{S22E2}
\mathbb{P} \left( \Delta(n,z) \geq n^{1/4} \right) < \epsilon, \mbox{ where $\Delta(n,z):=  \sup_{0 \leq t \leq n} \left| \sqrt{n} B^\sigma_{t/n} + \frac{t}{n} \cdot z - \ell^{(n,z)}(t) \right|.$}
\end{equation}
\end{proposition}
\begin{proof} We fix $q \in (0,1)$ such that $p = \frac{q}{1-q}$. From \cite[Example 2, page 719]{DW19} we have that the geometric distribution satisfies the conditions of \cite[Theorem 4]{DW19}. The latter ensures the existence of constants $0< C, a, \alpha' < \infty$ (depending on $p$ alone) such that for any integer $n$ and any $z \in \mathbb{Z}_{\geq 0}$ we can find a Brownian bridge $B^{\sigma}$ with variance $\sigma^2 = p(1+p)$ and a random increasing path $\ell^{(n,z)}$ with law $\mathbb{P}_{\operatorname{Geom}}^{0, n, 0,z}$ defined on the same probability space such that 
$$\mathbb{E}\left[e^{a \Delta(n,z)} \right] \leq C n^{\alpha'} e^{(z- pn)^2/n}.$$
Using the latter and Chebyshev's inequality gives for $|z - pn| \leq A \sqrt{n}$
$$\mathbb{P} \left( \Delta(n,z) \geq n^{1/4} \right) \leq e^{-a n^{1/4}} \mathbb{E}\left[e^{a \Delta(n,z)} \right] \leq C n^{\alpha'} e^{A^2 - an^{1/4}}, $$
which implies the statement of the proposition.
\end{proof}

Before we turn to the next and last statement of this section we introduce a bit of notation. If $a,b, x,y \in \mathbb{R}$ with $a < b$ we define a random element in $C([a,b])$ via
\begin{equation}\label{S22E3}
B(t) = (b-a)^{1/2} \cdot \tilde{B} \left( \frac{t-a}{b-a} \right) + \left(\frac{b-t}{b-a} \right) \cdot x + \left( \frac{t-a}{b-a}\right) \cdot y,
\end{equation}
where $\tilde{B}$ is a standard Brownian bridge. We refer to $B(t)$ in (\ref{S22E3}) as a {\em Brownian bridge} (from $B(a) = x$ to $B(b) = y$) with variance $1$, and denote its distribution as $\mathbb{P}_{\operatorname{free}}^{a,b,x,y}$. If $\vec{x}, \vec{y} \in \mathbb{R}^k$, then we denote by $\mathbb{P}_{\operatorname{free}}^{a,b,\vec{x},\vec{y}}$ the law of $k$ independent Brownian bridges $\{B_i\}_{i = 1}^k$ from $B_i(a) = x_i$ to $B_i(b) = y_i$ all with variance $1$. For $k \in \mathbb{N}$ we let $W_k^{\circ}$ denote the open {\em Weyl chamber} in $\mathbb{R}^k$
\begin{equation}\label{S22E4}
W_k^{\circ} = \{\vec{x} = (x_1, \dots, x_k) \in \mathbb{R}^k: x_1 > x_2 > \cdots > x_k\}.
\end{equation}

\begin{definition}\label{AvoidBB}
Fix $k \in \mathbb{N}$. Let $\vec{x}, \vec{y} \in W_k^{\circ}$, $a,b \in \mathbb{R}$ with $a < b$ and $f: [a,b] \rightarrow (-\infty, \infty]$ and $g: [-\infty, \infty)$ be two continuous functions. The latter means that either $f: [a,b] \rightarrow \mathbb{R}$ is continuous or $f = \infty$ everywhere and similarly for $g$. We also assume that $f(t) > g(t)$ for all $t \in [a,b]$, $f(a) > x_1$, $f(b) > y_1$, $g(a) < x_k$ and $g(b) < y_k$. 

With the above data we define the $(f,g)$-avoiding Brownian line ensemble on the interval $[a,b]$ with entrance data $\vec{x}$ and exit data $\vec{y}$ to be the $\llbracket 1 , k \rrbracket$-indexed line ensemble $\mathcal{Q}$ on $\Lambda = [a,b]$ whose law is the same as that of $\{B_i\}_{i = 1}^k$ (with law $\mathbb{P}_{\operatorname{free}}^{a,b, \vec{x}, \vec{y}}$), conditioned on the event
$$E = \{f(r) > B_1(r) > B_2(r) > \cdots > B_k(r) > g(r) \mbox{ for all $r \in [a,b]$} \}.$$
We denote the law of $\mathcal{Q}$ by $\mathbb{P}_{\operatorname{avoid}}^{a,b,\vec{x},\vec{y},f,g}$, and expectations with respect to this measure by $\mathbb{E}_{\operatorname{avoid}}^{a,b,\vec{x},\vec{y},f,g}$. 
\end{definition}
\begin{remark} We refer the reader to \cite[Definition 2.7]{DEA21} for more details on $\mathbb{P}_{\operatorname{avoid}}^{a,b,\vec{x},\vec{y},f,g}$ and in particular why it is well-defined.
\end{remark}

With the above notation in place we can state our final result of the section, Lemma \ref{lem:RW} below. In plain words, the lemma states that if we have a sequence of $(F_n,G_n)$-interlacing geometric line ensembles $\mathfrak{Q}^n$ with laws $\mathbb{P}_{\ice, \operatorname{Geom}}^{A_n,B_n, \vec{X}^n, \vec{Y}^n, F_n, G_n}$ as in Definition \ref{DefAvoidingLawBer} whose boundary data converges regularly to some limiting boundary data, then under an appropriate diffuse scaling the ensembles $\mathfrak{Q}_n$ converges to an $(f,g)$-avoiding Brownian line ensemble as in Definition \ref{AvoidBB}. The proof of Lemma \ref{lem:RW} is given in Appendix \ref{AppendixA2}.

\begin{lemma}\label{lem:RW} Let $k, \vec{x}, \vec{y}, a,b, f,g$ be as in the first paragraph of Definition \ref{AvoidBB}. Suppose that $d_n$ is a sequence of positive reals such that $d_n \rightarrow \infty$ as $n \rightarrow \infty$, and set $A_n = \lfloor a d_n \rfloor$, $B_n = \lceil b d_n \rceil$. Let $f_n: [d_n^{-1}A_n, d_n^{-1} B_n] \rightarrow (-\infty, \infty]$ and $g_n: [d_n^{-1}A_n, d_n^{-1} B_n] \rightarrow [-\infty, \infty)$ be continuous functions such that $f_n \rightarrow f$ and $g_n \rightarrow g$ uniformly on $[a,b]$. In the case when $f = \infty$, the last statement means that $f_n = \infty$ for all large $n$, and when $g = -\infty$, the last statement means that $g_n = -\infty$ for all large $n$. We also suppose that 
\begin{equation}\label{EdgeLim}
\lim_{n \rightarrow \infty} |g_n(d_n^{-1}B_n) - g_n(b)| = 0 \mbox{ if } g \neq -\infty \mbox{ and }\lim_{n \rightarrow \infty} |f_n(d_n^{-1}A_n) - f_n(a)| = 0 \mbox{ if } f \neq -\infty.
\end{equation}

Fix $p \in (0,\infty)$, $\sigma = \sqrt{p(1+p)}$ and suppose that $\vec{X}^n, \vec{Y}^n \in \mathfrak{W}_k$ are sequences such that 
\begin{equation}\label{SideLim}
\lim_n \sigma^{-1} d_n^{-1/2} (X_i^n - p A_n)  = x_i \mbox{ and } \lim_n \sigma^{-1} d_n^{-1/2} (Y_i^n - p B_n)   = y_i \mbox{ for } i = 1, \dots, k.
\end{equation}
Define $F_n: [A_n, B_n] \rightarrow (-\infty, \infty]$ and $G_n: [A_n, B_n] \rightarrow [-\infty, \infty)$ through
$$F_n(t) = \sigma d_n^{1/2} \cdot f_n(t/d_n) + pt \mbox{ and } G_n(t) = \sigma d_n^{1/2} \cdot g_n(t/d_n) + p t.$$
Then, we have the following statements.
\begin{enumerate}
\item There exists $N_0 \in \mathbb{N}$ such that for $n \geq N_0$ the laws $\mathbb{P}_{\ice, \operatorname{Geom}}^{A_n,B_n, \vec{X}^n, \vec{Y}^n, F_n, G_n}$ are well-defined, i.e. the sets $\Omega_{\ice}(A_n, B_n, \vec{X}^n, \vec{Y}^n, F_n, G_n)$ are non-empty.
\item If $\mathfrak{Q}^n$ has law $\mathbb{P}_{\ice, \operatorname{Geom}}^{A_n,B_n, \vec{X}^n, \vec{Y}^n, F_n, G_n}$, then the sequence of $\llbracket 1,k \rrbracket$-indexed line ensembles $\mathcal{Q}^n$ on $[a,b]$ defined by
\begin{equation}\label{S22E5}
\mathcal{Q}_i^n(t) = \sigma^{-1} d_n^{-1/2} \cdot \left( Q_i(t d_n) - td_n p \right) \mbox{ for } n \geq N_0, \hspace{2mm} t \in [a,b] \mbox{ and } i = 1, \dots, k,
\end{equation}
converges weakly to $\mathbb{P}_{\operatorname{avoid}}^{a,b,\vec{x},\vec{y},f,g}$ from Definition \ref{AvoidBB} as $n \rightarrow \infty$.
\end{enumerate}
\end{lemma}

%---------------------------------------------------------------------------------------------------------------
% Section 3
%
%----------------------------------------------------------------------------------------------------------------
\section{Estimates for single-curve ensembles}\label{Section3} In this section we derive various statements about the measures $\mathbb{P}_{\ice, \operatorname{Geom}}^{T_0,T_1, \vec{x}, \vec{y}, f, g}$ from Definition \ref{DefAvoidingLawBer} when $k = 1$, i.e. when these ensembles contain a single curve. To keep the notation simple we assume $T_0 = 0$, $T_1 = n$, and focus on our primary case of interest when $f = \infty$. We mention that in the special case when $g = -\infty$ we have that $\mathbb{P}_{\ice, \operatorname{Geom}}^{0,n, x, y, \infty, -\infty} = \mathbb{P}_{ \operatorname{Geom}}^{0,n, x, y}$. The main tools we use in deriving the results of this section are the monotone coupling from Lemma \ref{MCL} and the strong coupling from Proposition \ref{KMT}. We continue with the same notation as in Section \ref{Section2}.

%---------------------------------------------------------------------------------------------------------------
% Section 3.1
%
%----------------------------------------------------------------------------------------------------------------
\subsection{No concentration in the bulk}\label{Section3.1} In this section we show that if $Q_1$ has law $\mathbb{P}_{ \operatorname{Geom}}^{0,n, x, y}$, then for any $t \in (0,1)$ the distribution of the random variables $n^{-1/2}(Q_1(tn) -ptn)$ does not concentrate on small intervals. The latter is expected since by Proposition \ref{KMT} the variables $n^{-1/2}(Q_1(tn) -ptn)$ are approximately Gaussian with a non-zero variance, and the latter are diffuse. The precise statement and proof are as follows.

\begin{lemma}\label{S31L} Fix $p, \eta, M \in (0, \infty)$ and $\delta_1 \in (0,1/2)$. There exist $N_1 \in \mathbb{N}$ and $\delta > 0$, depending on $p, \eta, M, \delta_1$, such that the following holds. If $n \geq N_1$, $x, y \in \mathbb{Z}$ are such that $x \leq y$ and $|x| \leq M n^{1/2}$, $|y - pn| \leq M n^{1/2}$, $t \in [\delta_1, 1-\delta_1]$, and $a \in \mathbb{R}$, we have
\begin{equation}\label{S31E1}
\mathbb{P}_{ \operatorname{Geom}}^{0,n, x, y} \left( n^{-1/2} (Q_1(tn) - ptn) \in [a -\delta, a + \delta ]  \right) < \eta.
\end{equation}
\end{lemma}
\begin{proof} By translation we have
$$\mathbb{P}_{ \operatorname{Geom}}^{0,n, x, y} \left( n^{-1/2} (Q_1(tn) - ptn) \in [a,b]  \right) = \mathbb{P}_{ \operatorname{Geom}}^{0,n, 0, y-x} \left( n^{-1/2} (Q_1(tn) - ptn) \in [a-c,b-c]  \right),  $$
where $c = xn^{-1/2}$. We conclude that it suffices to show that we can find $N_1 \in \mathbb{N}$ and $\delta > 0$ such that if $n \geq N_1$, $z \in \mathbb{Z}_{\geq 0}$ with $|z - pn| \leq 2M n^{1/2}$, $t \in [\delta_1, 1- \delta_1]$ and $a \in \mathbb{R}$ we have
\begin{equation}\label{S31E2}
\mathbb{P}_{ \operatorname{Geom}}^{0,n, 0, z} \left( n^{-1/2} (Q_1(tn) - ptn) \in [a -\delta, a + \delta ]  \right) < \eta.
\end{equation}

Let $\delta > 0$ be small enough so that $\frac{4\delta}{\sqrt{2\pi} \sigma t(1-t)} < \eta/2$ for all $t \in [\delta_1, 1- \delta_1]$ and observe that if $B^{\sigma}$ is a Brownian bridge (from $B^{\sigma}_0 = 0$ to $B_1^{\sigma} = 0$) with variance $\sigma^2 = p(1+p)$ we have for $t \in [\delta_1, 1- \delta_1]$ and any interval $I$ of length at most $4\delta$ that
\begin{equation}\label{S31E3}
\mathbb{P}(B^{\sigma}_t \in I) = \int_{I} \frac{e^{-x^2/[2\sigma^2 t (1-t)]}}{\sqrt{2\pi} \sigma t(1-t)} \leq \int_{I} \frac{1}{\sqrt{2\pi} \sigma t(1-t)} \leq \frac{4\delta}{\sqrt{2\pi} \sigma t(1-t)} < \eta/2.
\end{equation}
This specifies our choice of $\delta$. Let $N_0$ be as in the statement of Proposition \ref{KMT} for $A = 2M$, $\epsilon = \eta/2$ and $p$ as in the statement of the present lemma. Let $N_1$ be sufficiently large so that $N_1 \geq N_0$ and for $n \geq N_1$ we have $n^{-1/4} < \delta$. We proceed to show (\ref{S31E2}) holds for the above choice of $\delta$ and $n \geq N_1$.

By Proposition \ref{KMT} we can find a probability space $(\Omega, \mathcal{F},\mathbb{P})$ supporting a random increasing path $\ell^{(n,z)}$ with law $\mathbb{P}_{\operatorname{Geom}}^{0,n, 0, z}$ as well as a Brownian bridge $B^{\sigma}$ such that for $ n \geq N_1$
\begin{equation}\label{S31E4}
\mathbb{P}( \Delta(n,z) \geq n^{1/4} ) < \eta/2,
\end{equation}
where we recall that $\Delta(n,z)$ is as in (\ref{S22E2}). By the triangle inequality we have for all $t \in [0,1]$
$$\left| n^{-1/2} (\ell^{(n,z)}(tn) - ptn) - B_{t}^{\sigma} + c(t,n)\right| \leq n^{-1/2}\Delta(n,z), \mbox{ where } c(t,n) = ptn^{1/2} - t zn^{-1/2} .$$
Combining the latter with the inequality $n^{-1/4} < \delta$ we conclude for all $a \in \mathbb{R}$ and $t \in [\delta_1, 1- \delta_1]$
\begin{equation}\label{S31E5}
\begin{split}
&\mathbb{P}\left( n^{-1/2} (\ell^{(n,z)}(tn) - ptn) \in [a -\delta, a + \delta ]  \right) \leq \mathbb{P}( \Delta(n,z) \geq n^{1/4} ) \\
& + \mathbb{P}( B^{\sigma}_t - c(t,n) \in [a-2\delta, a + 2\delta]) < \eta/2 + \eta/2= \eta ,
\end{split}
\end{equation}
where in the last inequality we used (\ref{S31E3}) and (\ref{S31E4}). Since under $\mathbb{P}$ the law of $\ell^{(n,z)}$ is $\mathbb{P}_{\operatorname{Geom}}^{0,n, 0, z}$, we see that (\ref{S31E5}) implies (\ref{S31E2}) and hence the lemma.
\end{proof}

%---------------------------------------------------------------------------------------------------------------
% Section 3.2
%
%----------------------------------------------------------------------------------------------------------------
\subsection{Being high sometimes}\label{Section3.2} In this section we show that if $Q_1$ has law $\mathbb{P}_{ \operatorname{Geom}}^{0,n, x, y, \infty, g}$, and $x,y$ are not too low, then neither is $Q_1(s)$ for any $s \in [0,n]$ at least some of the time. The latter is expected, since by Proposition \ref{KMT} the curve $Q_1$ behaves like a Brownian bridge $B$ of appropriate variance between $(0,x)$ and $(n,y)$ and $B(s) \geq (1-s/n)x + (s/n)y$ with probability $1/2$.

\begin{lemma}\label{S32L} Fix $p, M\in (0, \infty)$. There exists $N_2 \in \mathbb{N}$, depending on $p$ and $M$, such that the following holds. If $M_1, M_2 \in \mathbb{R}$ are such that $|M_1 - M_2| \leq M$, $n \geq N_2$, $x, y \in \mathbb{Z}$ are such that $x \leq y$ and $x \geq M_1 n^{1/2}$, $y \geq pn + M_2 n^{1/2}$, $g: \llbracket 0, n \rrbracket \rightarrow [-\infty, \infty)$ is arbitrary such that $\Omega_{\ice}(0,n,x,y, \infty, g) \neq \emptyset$, and $s \in [0, n]$, then we have
\begin{equation}\label{S32E1}
\mathbb{P}_{\ice, \operatorname{Geom}}^{0,n, x, y, \infty, g} \left( Q_1(s) \geq \frac{n - s}{n} \cdot M_1 n^{1/2} + \frac{s}{n} \cdot (pn + M_2 n^{1/2}) - 3n^{1/4} \right) \geq \frac{1}{3}.
\end{equation}
\end{lemma}
\begin{proof} Let $N_0$ be as in the statement of Proposition \ref{KMT} for $A = M + 2$, $\epsilon = 1/6$ and $p$ as in the statement of the present lemma. Let $N_2$ be sufficiently large so that $N_2 \geq N_0$ and for $n \geq N_2$ we have $A_n = \lfloor M_1 n^{1/2} \rfloor \leq \lfloor p n + M_2 n^{1/2} \rfloor = B_n$. We proceed to show (\ref{S32E1}) holds for $n \geq N_2$.

From the monotone coupling in Lemma \ref{lem:RW} it suffices to show
\begin{equation}\label{S32E2}
\mathbb{P}_{\ice, \operatorname{Geom}}^{0,n, A_n, B_n, \infty, -\infty} \left( Q_1(s) \geq \frac{n - s}{n} \cdot M_1 n^{1/2} + \frac{s}{n} \cdot (pn + M_2 n^{1/2}) - 3n^{1/4} \right) \geq \frac{1}{3}.
\end{equation}
Notice that $\mathbb{P}_{\ice, \operatorname{Geom}}^{0,n, A_n, B_n, \infty, -\infty} = \mathbb{P}_{ \operatorname{Geom}}^{0,n, A_n, B_n}$ and so by translation it suffices to prove
\begin{equation}\label{S32E3}
\mathbb{P}_{\operatorname{Geom}}^{0,n, 0, B_n - A_n} \left( Q_1(s) \geq \frac{n - s}{n} \cdot M_1 n^{1/2} + \frac{s}{n} \cdot (pn + M_2 n^{1/2}) - 3 n^{1/4} - A_n \right) \geq \frac{1}{3}.
\end{equation}
Setting $z = B_n - A_n$, we have that $|z - pn| \leq (|M_1 - M_2| + 2) \cdot n^{1/2} \leq A \cdot n^{1/2}$ and so by Proposition \ref{KMT} we can find a probability space $(\Omega, \mathcal{F},\mathbb{P})$ supporting a random increasing path $\ell^{(n,z)}$ with law $\mathbb{P}_{\operatorname{Geom}}^{0,n, 0, z}$ as well as a Brownian bridge $B^{\sigma}$ such that for $ n \geq N_2$
\begin{equation}\label{S32E4}
\mathbb{P}( \Delta(n,z) \geq n^{1/4} ) < 1/6.
\end{equation}
From the triangle inequality we have for $s \in [0,n]$
$$\left|  n^{1/2} \cdot  B_{s/n}^{\sigma} + \frac{n-s}{n} \cdot M_1 n^{1/2} + \frac{s}{n} \cdot (pn + M_2 n^{1/2}) - A_n - \ell^{(n,z)}(s) \right| \leq \Delta(n,z) + 2,$$
where the extra ``$2$'' comes from taking floors. The latter and (\ref{S32E4}) imply
\begin{equation}\label{S32E5}
\begin{split}
&\mathbb{P}\left( \ell^{(n,z)}(s) \geq \frac{n - s}{n} \cdot M_1 n^{1/2} + \frac{s}{n} \cdot (pn + M_2 n^{1/2}) - 3 n^{1/4} - A_n \right) \\
& \geq \mathbb{P}\left( n^{1/2}  B^{\sigma}_{s/n} - \Delta(n,z) - 2 \geq - 3 n^{1/4}  \right) \geq \mathbb{P}( n^{1/2}  B^{\sigma}_{s/n} \geq 0 ) - \mathbb{P}(\Delta(n,z) \geq n^{1/4} ) \geq 1/3,
\end{split} 
\end{equation}
where in the last inequality we used that $B^{\sigma}_{s/n}$ is a zero-centered normal variable. Since under $\mathbb{P}$ the law of $\ell^{(n,z)}$ is $\mathbb{P}_{\operatorname{Geom}}^{0,n, 0, B_n - A_n}$, we see that (\ref{S32E5}) implies (\ref{S32E3}) and hence the lemma.
\end{proof}

%---------------------------------------------------------------------------------------------------------------
% Section 3.3
%
%----------------------------------------------------------------------------------------------------------------
\subsection{Modulus of continuity estimates}\label{Section3.3} Suppose that $a,b \in \mathbb{R}$ are such that $a < b$. For a function $f \in C([a,b])$ we define its {\em modulus of continuity} for $\delta > 0$ by 
\begin{equation}\label{S33E1}
w(f,\delta) = \sup_{\substack{x,y \in [a,b] \\ |x-y| \leq \delta}} |f(x) - f(y)|.
\end{equation}
Our next result states that if $Q_1$ has law $\mathbb{P}_{ \operatorname{Geom}}^{0,n, x, y}$, and $x,y$ are not too high or low, then $Q_1$ has a well-behaved modulus of continuity. The latter is expected, since by Proposition \ref{KMT} the curve $Q_1$ behaves like a Brownian bridge, which has a well-behaved modulus of continuity.
\begin{lemma}\label{S33L} Fix $p, M, \epsilon, \eta \in (0, \infty)$, and $\sigma = \sqrt{p(1+p)}$. There exist $N_3 \in \mathbb{N}$ and $\delta$, depending on $p, M, \epsilon, \eta$, such that the following holds. If $n \geq N_3$, $x, y \in \mathbb{Z}$ are such that $x \leq y$ and $|x| \leq M n^{1/2}$, $|y - pn| \leq M n^{1/2}$, and $\mathcal{Q}_1\in C([0,1])$ is defined by $\mathcal{Q}_1(t) = \sigma^{-1}n^{-1/2}(Q_1(tn) - ptn)$, then
\begin{equation}\label{S33E2}
\mathbb{P}_{\operatorname{Geom}}^{0,n, x, y} \left( w(\mathcal{Q}_1, \delta) > \epsilon \right) < \eta.
\end{equation}
\end{lemma}
\begin{proof} By translation we have that 
$$\mathbb{P}_{\operatorname{Geom}}^{0,n, x, y} \left( w(\mathcal{Q}_1, \delta) > \epsilon \right) = \mathbb{P}_{\operatorname{Geom}}^{0,n, 0, y- x} \left( w(\mathcal{Q}_1, \delta) > \epsilon \right),$$
and so it suffices to show that we can find $N_3$, $\delta$ such that if $n \geq N_3$, $z \in \mathbb{Z}_{\geq 0}$ and $|z - pn| \leq 2Mn^{1/2}$
\begin{equation}\label{S33E3}
\mathbb{P}_{\operatorname{Geom}}^{0,n, 0, z} \left( w(\mathcal{Q}_1, \delta) > \epsilon \right) < \eta.
\end{equation}
Let $N_0$ be as in the statement of Proposition \ref{KMT} for $A = 2M$, $\epsilon = \eta/4$ and $p$ as in the statement of the present lemma. From Proposition \ref{KMT} we can find a probability space $(\Omega, \mathcal{F},\mathbb{P})$ supporting a random increasing path $\ell^{(n,z)}$ with law $\mathbb{P}_{\operatorname{Geom}}^{0,n, 0, z}$ as well as a Brownian bridge $B^{\sigma}$ such that for $ n \geq N_0$ and $z \in \mathbb{Z}_{\geq 0}$, $|z - pn| \leq 2Mn^{1/2}$
\begin{equation}\label{S33E4}
\mathbb{P}( \Delta(n,z) \geq n^{1/4} ) < \eta/4.
\end{equation}
By the triangle inequality we have for all $t_1, t_2 \in [0,1]$
$$|\mathcal{L}^{(n,z)}_1(t_1) - \mathcal{L}^{(n,z)}_1(t_2)| \leq \frac{2 \Delta(n,z)}{\sigma n^{1/2}} + \frac{|t_2 - t_1| \cdot |z-np|}{\sigma n^{1/2}} + \sigma^{-1} \left| B_{t_1}^{\sigma} - B_{t_2}^{\sigma} \right|,$$
where $\mathcal{L}^{(n,z)}_1(t) = \sigma^{-1} n^{-1/2} (\ell^{(n,z)}(tn) - ptn)$. The latter shows that for any $\tilde{\delta} > 0$
\begin{equation}\label{S33E5}
w(\mathcal{L}^{(n,z)}_1,\tilde{\delta}) \leq \frac{2 \Delta(n,z)}{\sigma n^{1/2}} +  2\sigma^{-1} \tilde{\delta} M + \sigma^{-1} w(B^{\sigma}, \tilde{\delta}).
\end{equation}
Since $B^{\sigma}$ is continuous we have that $w(B^{\sigma}, \tilde{\delta}) \rightarrow 0$ a.s. as $\tilde{\delta} \rightarrow 0$, and so we can find $\delta > 0$ small enough so that 
\begin{equation}\label{S33E6}
2\sigma^{-1} \delta M < \epsilon/4, \mbox{ and } \mathbb{P}(\sigma^{-1} w(B^{\sigma}, \delta) > \epsilon/4) < \eta/4.
\end{equation}
This specifies our choice of $\delta$ and we pick $N_3$ sufficiently large so that $N_3 \geq N_0$ and for $n \geq N_3$ we have $2\sigma^{-1} n^{-1/4} < \epsilon/4$. By combining (\ref{S33E4}), (\ref{S33E5}) and (\ref{S33E6}) we get 
\begin{equation*}
\begin{split}
&\mathbb{P}\left(w(\mathcal{L}^{(n,z)}_1,\delta) > \epsilon  \right) \leq \mathbb{P}( \Delta(n,z) \geq n^{1/4} ) + \mathbb{P} \left( \frac{2}{\sigma n^{1/4}} +  2\sigma^{-1} \delta M + \sigma^{-1} w(B^{\sigma}, \delta ) > \epsilon \right) \\
&\leq \eta/4 + \mathbb{P} \left( \sigma^{-1} w(B^{\sigma}, \delta ) > \epsilon /2 \right) < \eta/4 + \eta/4 = \eta/2.
\end{split}
\end{equation*}
Since under $\mathbb{P}$ the law of $\mathcal{L}^{(n,z)}_1$ is the same as $\mathcal{Q}_1$, we conclude (\ref{S33E3}) and hence the lemma.
\end{proof}

%---------------------------------------------------------------------------------------------------------------
% Section 3.4
%
%----------------------------------------------------------------------------------------------------------------
\subsection{Staying in the corridor}\label{Section3.4} We start with a simple statement we will require, which is an immediate consequence of \cite[Corollary 2.10]{CorHamA}.
\begin{lemma}\label{CHL} Let $a,b, x,y \in \mathbb{R}$, $a < b$, and suppose that $B$ is a Brownian bridge from $B(a) = x$ to $B(b) = y$ with variance $1$ as in (\ref{S22E3}). Let $U$ be an open set in $C([a,b])$ that contains a function $f$ such that $f(a) = x$ and $f(b) = y$. Then $\mathbb{P}(B \in U) > 0$.
\end{lemma}
\begin{proof} Consider the map $F: C([0,1]) \rightarrow C([a,b])$, given by 
$$F(g)(t) = (b-a)^{1/2} \cdot g \left(\frac{t-a}{b-a} \right) + \left( \frac{b-t}{b-a} \right) \cdot x + \left( \frac{t-a}{b-a} \right) \cdot y.$$
We readily observe that $F$ is a homeomorphism, $F^{-1}(B)$ is a standard Brownian motion, and $0 = F^{-1}(f)(0) = F^{-1}(f)(1)$. From \cite[Lemma 2.10]{CorHamA} we get $\mathbb{P}(B \in U) = \mathbb{P}(F^{-1}(B) \in F^{-1}(U)) > 0$.
\end{proof}

Given $A, B \in \mathbb{R}$, $x,y \in \mathbb{R}$ and $\delta \in (0,1/2)$ we define the function $f(t|A,B,x,y,\delta)$ in $C([0,1])$ to be the function such that 
\begin{equation}\label{S34E1}
\begin{split}
&f(0|A,B,x,y, \delta) = x + A, \hspace{2mm} f(1|A,B,x,y, \delta) = y+ A, \hspace{2mm} \\
&f(\delta |A,B,x,y, \delta) = f(1-\delta|A,B,x,y, \delta) = B + A,
\end{split}
\end{equation}
and $f(t|A,B,x,y, \delta)$ is linear on the intervals $[0,\delta]$, $[\delta, 1-\delta]$ and $[1-\delta, 1]$, see Figure \ref{S3_1}.

\begin{figure}[ht]
    \centering
     \begin{tikzpicture}[scale=2.7]

        \def\tra{3} 
        % Picture on the left
        % Coordinate System
        \draw[->, thick, gray] (-1.4,0)--(1.4,0);
        \draw[->, thick, gray] (-1,-1)--(-1,1.2);
        
        \draw[black, fill = black] (-1,0) circle (0.02);
        \draw[black, fill = black] (1,0) circle (0.02);        
        \draw (-1.1,-0.1) node{$0$};
        \draw (1, -0.1) node{$1$};

        \draw[black, fill = black] (1,-0.8) circle (0.02);
        \draw (1.2,-0.8) node{$(1,y)$};

        \draw[black, fill = black] (1,-0.6) circle (0.02);
        \draw (1.34,-0.6) node{$(1,y + A)$};

        \draw[black, fill = black] (-1,0.3) circle (0.02);
        \draw (-1.2,0.3) node{$(0,x)$};

        \draw[black, fill = black] (-1,0.5) circle (0.02);
        \draw (-1.34,0.5) node{$(0,x + A)$};

        \draw[-] (-0.5,0.8)--(0.5, 0.8);
        \draw[-] (-1,0.5)--(-0.5, 0.8);
        \draw[-] (0.5,0.8)--(1, -0.6);

        \draw[black, fill = black] (-0.5,0.8) circle (0.02);
        \draw (-0.5,0.9) node{$(\delta, B+A)$};

        \draw[black, fill = black] (0.5,0.8) circle (0.02);
        \draw (0.5,0.9) node{$(1-\delta, B+A)$};

        % ---------------------------------------------------------

        \draw[black, fill = black] (-0.5,0.4) circle (0.02);

        \draw[black, fill = black] (0.5,0.4) circle (0.02);

        \draw[->] (0,0.8)--(0, 0.4);
        \draw[->] (0,0.4)--(0, 0.8);
        \draw (0.1,0.6) node{$2A$};

        \draw[black, fill = black] (1,-1) circle (0.02);
        \draw (1.34,-1) node{$(1,y - A)$};

        \draw[black, fill = black] (-1,0.1) circle (0.02);
        \draw (-1.34,0.1) node{$(0,x - A)$};

        \draw[-] (-0.5,0.4)--(0.5, 0.4);
        \draw[-] (-1,0.1)--(-0.5, 0.4);
        \draw[-] (0.5,0.4)--(1, -1);

        \draw[black, fill = black] (-0.5,0) circle (0.02);
        \draw[black, fill = black] (0.5,0) circle (0.02);
        \draw (-0.5,-0.1) node{$\delta$};
        \draw (0.5, -0.1) node{$1 -\delta$};

    \end{tikzpicture} 

    \caption{The figure depicts the functions $f(\cdot |A,B,x,y, \delta)$ and $f(\cdot |-A,B,x,y, \delta)$ from (\ref{S34E1}) for $A > 0$.}
    \label{S3_1}
\end{figure}

The main result in this section, Lemma \ref{S34L} below, shows that if $Q_1$ has law $\mathbb{P}_{ \operatorname{Geom}}^{0,n, x, y}$, then with some non-vanishing probability its scaled version $\mathcal{Q}_1(t) = \sigma^{-1}n^{-1/2}(Q_1(tn) - ptn)$ stays in the corridor determined by two functions of the kind in (\ref{S34E1}). 

\begin{lemma}\label{S34L} Fix $p, M_1, M_2, A \in (0,\infty)$, $\delta_1 \in (0,1/2)$, and $\sigma = \sqrt{p(1+p)}$. There exist $N_4 \in \mathbb{N}$ and $\epsilon > 0$, depending on $p,M_1,M_2, A, \delta_1$ such that the following holds for all $n \geq N_4$. Suppose that $x,y \in \mathbb{Z}$ are such that $x\leq y$, $|x| \leq M_1 n^{1/2}$, $|y - np| \leq M_1 n^{1/2}$ and $B \in [-M_2, M_2]$. Let $f^t, f^b \in C([0,1])$ be defined through
$$f^t(t) = f(t|A, B, \sigma^{-1} n^{-1/2} x, \sigma^{-1} n^{-1/2}(y- pn), \delta_1) $$
$$f^b(t) = f(t|-A, B, \sigma^{-1} n^{-1/2} x, \sigma^{-1} n^{-1/2}(y- pn), \delta_1), $$
where we recall that $f(t|A,B,x,y,\delta)$ is as in (\ref{S34E1}). If $\mathcal{Q}_1\in C([0,1])$ is defined by $\mathcal{Q}_1(t) = \sigma^{-1}n^{-1/2}(Q_1(tn) - ptn)$, we have
\begin{equation}\label{S34E2}
\mathbb{P}_{\operatorname{Geom}}^{0,n, x, y} \left( f^t(t) \geq \mathcal{Q}_1(t) \geq f^b(t) \mbox{ for all } t\in [0,1] \right) > \epsilon.
\end{equation}
\end{lemma}
\begin{proof} Suppose for the sake of contradiction that no such $N_4$ and $\epsilon$ exist. Then, we can find sequences $n_k$ increasing to infinity, $B_k \in [-M_2, M_2]$, $x_k, y_k \in \mathbb{Z}$ with $x_k \leq y_k$, $|x_k| \leq M_1 n_k^{1/2}$, $|y_k - p n_k| \leq M_1 n_k^{1/2}$ such that 
\begin{equation}\label{S34E3}
\lim_{k \rightarrow \infty} \mathbb{P}_{\operatorname{Geom}}^{0,n_k, x_k, y_k} \left( f^t_k(t) \geq \mathcal{Q}_1(t) \geq f^b_k(t) \mbox{ for all } t\in [0,1] \right) = 0,
\end{equation}
where $f^t_k$ and $f^b_k$ are as in the statement of the lemma but with $B, x,y, n$ replaced with $B_k, x_k, y_k, n_k$, respectively. By possibly passing to a subsequence, which we continue to call $n_k$, we may assume 
\begin{equation}\label{S34E4}
\lim_k B_k = B, \hspace{2mm} \lim_k \sigma^{-1} n_k^{-1/2} x_k = x, \hspace{2mm} \lim_k \sigma^{-1} n_k^{-1/2} (y_k - p n_k) = y.
\end{equation}
Using Lemma \ref{lem:RW} applied to $d_n = n$, $a = 0$, $b = 1$, $A_n = 0$, $B_n = n$, $f = f_n = \infty$, $g = g_n = -\infty$ we know that if $Q^k_1$ has law $\mathbb{P}_{\operatorname{Geom}}^{0,n_k, x_k, y_k}$, then $\mathcal{Q}^k_1 = \sigma^{-1}n^{-1/2}(Q^k_1(tn) - ptn)$ converges weakly to a Brownian bridge $B$ from $B(0) = x$ to $B(1) = y$ with variance $1$. Let $f(t|0, B,x,y,\delta_1)$ be as in (\ref{S34E1}) and let $U$ be the open set 
$$U = \{g \in C([0,1]): \sup_{t \in [0,1]} |g(t) - f(t|0, B,x,y,\delta_1)| < A/2 \}.$$
We note that $U$ contains the function $f(s|0, B,x,y,\delta_1)$ and so by Lemma \ref{CHL} $\mathbb{P}(B \in U) > 0$. On the other hand, from (\ref{S34E4}) we have for all large $k$
$$U \subset \{g \in C([0,1]): f_k^t(t) \geq g(t) \geq f_k^b(t) \}.$$
The last few statements imply
\begin{equation*}
\begin{split}
&\liminf_{k \rightarrow \infty} \mathbb{P}_{\operatorname{Geom}}^{0,n_k, x_k, y_k} \left( f^t_k(t) \geq \mathcal{Q}_1(t) \geq f^b_k(t) \mbox{ for all } t\in [0,1] \right) \\
&\geq \liminf_{k \rightarrow \infty} \mathbb{P}_{\operatorname{Geom}}^{0,n_k, x_k, y_k}(\mathcal{Q}_1 \in U) \geq \mathbb{P}(B \in U) > 0,
\end{split}
\end{equation*}
where the next to last inequality used the Portmanteau theorem, see \cite[Theorem 2.1]{Bill}. The last inequality leads to our desired contradiction with (\ref{S34E3}), concluding the proof.
\end{proof}

%---------------------------------------------------------------------------------------------------------------
% Section 4
%
%----------------------------------------------------------------------------------------------------------------
\section{Estimates for multi-curve ensembles}\label{Section4} In this section we derive two statements about the measures $\mathbb{P}_{\ice, \operatorname{Geom}}^{T_0,T_1, \vec{x}, \vec{y}, f, g}$ from Definition \ref{DefAvoidingLawBer}. As in Section \ref{Section3} we assume $T_0 = 0$, $T_1 = n$, and focus on the case $f = \infty$. We continue with the same notation as in Section \ref{Section2} and denote elements of $\Omega(0,n, \vec{x},\vec{y})$ by $\mathfrak{Q} = \{Q_i\}_{i = 1}^k$.

%---------------------------------------------------------------------------------------------------------------
% Section 4.1
%
%----------------------------------------------------------------------------------------------------------------
\subsection{Bottom boundary separation}\label{Section4.1} The goal of the section  is to prove the following lemma.

\begin{lemma}\label{S41L} Fix $p, \eta, M^{\mathsf{bot}}, M^{\mathsf{side}} \in (0, \infty)$, $k \in \mathbb{N}$ and $\delta_2 \in (0,1/2)$. There exist $N_5 \in \mathbb{N}$, $\delta^{\mathsf{sep}} > 0$ and $\Delta^{\mathsf{sep}} \in (0,\delta_2)$, depending on $p, \eta, M^{\mathsf{bot}}, M^{\mathsf{side}}, k , \delta_2$, such that the following holds for all $n \geq N_5$. If we assume:
\begin{itemize}
\item $\vec{x}, \vec{y}\in \mathfrak{W}_k$ as in (\ref{DefSig}), and $|x_i| \leq M^{\mathsf{side}} \cdot n^{1/2}$, $|y_i - pn| \leq M^{\mathsf{side}} \cdot n^{1/2}$ for $i  \in \llbracket 1, k \rrbracket$;
\item $g: \llbracket 0, n \rrbracket \rightarrow [-\infty, \infty)$ is such that $g(s) \leq ps + M^{\mathsf{bot}} \cdot n^{1/2} $ for all $s \in \llbracket 0, n \rrbracket$;
\item the set $\Omega_{\ice}(0,n,\vec{x},\vec{y}, \infty, g)$ is non-empty;
\item $t_0 \in [\delta_2, 1- \delta_2]$, and $F_n(t_0 n ) = \mathbb{Z} \cap [t_0n -\Delta^{\mathsf{sep}} \cdot n, t_0 n +\Delta^{\mathsf{sep}} \cdot n]$,
\end{itemize}
then the following inequality holds
\begin{equation}\label{S41E1}
\mathbb{P}_{\ice, \operatorname{Geom}}^{0, n, \vec{x}, \vec{y}, \infty, g}\left( Q_k(t_0n) - pt_0n \geq g(s) - ps + \delta^{\mathsf{sep}} \cdot n^{1/2}  \mbox{ for all } s \in F_n (t_0 n)\right) > 1- \eta.
\end{equation}
\end{lemma}
\begin{remark} In plain words, the lemma states that if the side boundaries $\vec{x}, \vec{y}$ are well-behaved (i.e. their coordinates are not too high or low), and the bottom boundary $g$ is not too high, then the $k$-th curve in the ensemble with law $\mathbb{P}_{\ice, \operatorname{Geom}}^{0, n, \vec{x}, \vec{y}, \infty, g}$ at a point $t_0n$ is well-separated from the bottom boundary $g$. Due to interlacing one should expect that $Q_k(t_0n)$ is well-separated from $g(t_0n + 1)$, but the separation in (\ref{S41E1}) is stronger, since it states that $Q_k(t_0n)$ is well-separated from the maximum of the properly re-centered $g(s)$ over a macroscopic interval around $t_0n$.
\end{remark}
\begin{proof} Let us discuss the main ideas of the proof. Suppose $s_n \in F_n(t_0n)$ is such that $g(s_n) -p s_n = \max_{s \in F_n(t_0n)} [g(s) - ps]$. Firstly, one can replace $\vec{x}, \vec{y}$ and $g$ in the statement of the lemma with lower side and bottom boundaries $\vec{x}\,',\vec{y}\,'$ and $\tilde{g}$ such that $\tilde{g}(s_n) = g(s_n)$. In view of the monotone coupling in Lemma \ref{MCL}, showing (\ref{S41E1}) for these lower boundaries would imply the statement for the original ones. We pick $\vec{x}\,'$ and $\vec{y}\,'$ to be lower than $\vec{x}$ and $\vec{y}$, respectively, and with well-separated coordinates, while $\tilde{g}$ is such that $\tilde{g}(s_n) = g(s_n)$ and $\tilde{g}(s) = -\infty$ when $s \neq s_n$. 

The idea of having $\vec{x}\,'$ and $\vec{y}\,'$ with well-separated coordinates is that under this condition one can show that any event that is unlikely under $\mathbb{P}_{\operatorname{Geom}}^{0, n, \vec{x}\,', \vec{y}\,'}$ is also unlikely under $\mathbb{P}_{\ice, \operatorname{Geom}}^{0, n, \vec{x}\,', \vec{y}\,', \infty, \tilde{g}}$, see (\ref{S41E10}). In other words, we can compare events for interlacing bridges with non-interlacing ones. The crucial statement that allows the comparison is that $\mathbb{P}_{\operatorname{Geom}}^{0, n, \vec{x}\,', \vec{y}\,'}\left(\Omega_{\ice}(0,n,\vec{x}\,',\vec{y}\,', \infty, \tilde{g}) \right)$ is bounded away from zero. The latter is established by constructing a likely event $A_1^n \times \cdots \times A_k^n \subseteq \Omega_{\ice}(0,n,\vec{x}\,',\vec{y}\,', \infty, \tilde{g})$, where $A_i^n \subseteq \Omega_{\ice}(0,n, x_i', y_i')$ are events that paths stay inside of corridors of the type in Lemma \ref{S34L}.

Once we can compare the interlacing bridges with free ones, we can use the results of Section \ref{Section3} to show that: (1) $Q_k(s_n-1) - p(s_n-1)$ is well-separated from $\tilde{g}(s_n) - ps_n$ with high probability -- this uses the no bulk concentration statement in Lemma \ref{S31L}, and also (2) $Q_k(t_0n) - pt_0n$ is close to $Q_k(s_n-1) - p(s_n-1)$ and hence also well-separated from $\tilde{g}(s_n) - ps_n$ with high probability -- this uses the modulus of continuity estimates from Lemma \ref{S33L}. We now turn to the details of the proof. \\

Throughout the proof we let $\sigma = \sqrt{p(1+p)}$ and set $\mathcal{Q}_i(t)= \sigma^{-1}n^{-1/2}(Q_i(tn) - ptn)$ for $Q_i \in \Omega(0,n, x,y)$ and $t \in [0,1]$. For clarity we split the proof into three steps.\\

{\bf \raggedleft Step 1.} Fix $\tilde{A} = 1$ and let $\vec{x}\,' = (x_1', \dots, x_k')$, $\vec{y}\,' = (y_1', \dots, y_k')$, where for $i \in \llbracket 1, k \rrbracket$ 
$$x_i' = \lfloor -M^{\mathsf{side}} \cdot n^{1/2}  - 4(i-1)\tilde{A} \cdot n^{1/2} \rfloor \mbox{ and } y_i' =  \lfloor pn -M^{\mathsf{side}} \cdot n^{1/2}  - 4(i-1)\tilde{A} \cdot n^{1/2} \rfloor.$$
We define the functions
$$f^t_i(s) = f(s|\sigma^{-1}\tilde{A}, \sigma^{-1}M^{\mathsf{bot}} + 4\sigma^{-1} (k-i + 1)\tilde{A}, \sigma^{-1} n^{-1/2} x_i', \sigma^{-1} n^{-1/2}(y_i'- pn), \delta_2/2) $$
$$f^b_i(s) = f(s|-\sigma^{-1}\tilde{A}, \sigma^{-1} M^{\mathsf{bot}} + 4 \sigma^{-1}(k-i + 1)\tilde{A}, \sigma^{-1} n^{-1/2} x_i', \sigma^{-1} n^{-1/2}(y_i'- pn), \delta_2/2), $$
where we recall that $f(s|A,B,x,y,\delta)$ is as in (\ref{S34E1}). In words, the functions $(f^t_i, f^b_i)$ for $i \in \llbracket 1, k \rrbracket$ form $k$ disjoint corridors like in Figure \ref{S3_1} of width $2\sigma^{-1} \tilde{A}$ that are approximately vertical distance $2\sigma^{-1} \tilde{A}$ away from each other. The distance is not exactly $2\sigma^{-1} \tilde{A}$, since $x_i' - x_{i+1}' = -4\tilde{A} \cdot n^{1/2} + O(1)$ and $y_i' - y_{i+1}' = -4\tilde{A} \cdot n^{1/2} + O(1)$ due to taking floors.

From Lemma \ref{S34L} with $p$ as in the present lemma, $\delta_1 =\delta_2/2$, $A = \sigma^{-1}\tilde{A}$, $M_1 = M^{\mathsf{side}} + 4k\tilde{A} + 1$, $M_2 = \sigma^{-1}M^{\mathsf{bot}}+ 4\sigma^{-1} k \tilde{A} $, we can find $N_4 \in \mathbb{N}$ and $\epsilon_0 > 0$ such that for $n \geq N_4$ and $i \in \llbracket 1, k \rrbracket$
\begin{equation}\label{S41E2}
\mathbb{P}_{\operatorname{Geom}}^{0,n, x_i', y_i'} \left( f^t_i(s) \geq \mathcal{Q}_1(s) \geq f^b_i(s) \mbox{ for all } s\in [0,1] \right) > \epsilon_0.
\end{equation}
If $A^n_i$ is the subset of increasing paths $Q_1 \in \Omega(0,n, x_i', y_i')$ that satisfy the conditions in (\ref{S41E2}), then (\ref{S41E2}) is equivalent to 
\begin{equation}\label{S41E3}
\frac{|A_i^n|}{| \Omega(0,n, x_i', y_i')|} > \epsilon_0.
\end{equation}

From Lemma \ref{S31L} with $p$ as in the present lemma, $\delta_1 =\delta_2/2$, $M = M^{\mathsf{side}} + 4k\tilde{A} + 1$, and $\eta = \epsilon_0^k \cdot (\eta/4)$, we can find $N_1 \in \mathbb{N}$ and $\delta^{\mathsf{sep}} > 0$ such that for $n \geq N_1$, $a \in \mathbb{R}$ and $t \in [\delta_1, 1-\delta_1]$
\begin{equation}\label{S41E4}
\mathbb{P}_{ \operatorname{Geom}}^{0,n, x_k', y_k'} \left( n^{-1/2} (Q_1(tn) - ptn) \in [a -2 \delta^{\mathsf{sep}}, a + 2\delta^{\mathsf{sep}} ]  \right) < \epsilon_0^k \cdot (\eta/4).
\end{equation}
This specifies our choice of $\delta^{\mathsf{sep}}$ in the lemma.

From Lemma \ref{S33L} with $p$ as in the present lemma, $M = M^{\mathsf{side}} + 4k\tilde{A} + 1$, $\epsilon = \sigma^{-1}\delta^{\mathsf{sep}}$ as above, and $\eta =\epsilon_0^k \cdot (\eta/4)$, we can find $N_3 \in \mathbb{N}$ and $\delta > 0$ such that for $n \geq N_3$
\begin{equation}\label{S41E5}
\mathbb{P}_{\operatorname{Geom}}^{0,n, x_k', y_k'} \left( w(\mathcal{Q}_1, \delta) > \sigma^{-1}\delta^{\mathsf{sep}}  \right) < \epsilon_0^k \cdot (\eta/4).
\end{equation}
We now set $\Delta^{\mathsf{sep}} = \min(\delta/2, \delta_2/4)$ and this specifies our choice of $\Delta^{\mathsf{sep}}$ in the lemma.

We finally let $N_5 \in \mathbb{N}$ be sufficiently large so that for $n \geq N_5$ we have $n \geq \max(N_1, N_3, N_4)$ as above, $\Delta^{\mathsf{sep}} > n^{-1}$, $\tilde{A} n^{1/2} \geq p$, and also for $s \in \llbracket 0, n-1\rrbracket$ and $i \in \llbracket 1, k-1\rrbracket$
\begin{equation}\label{S41E6}
f_i^{b}\left( \frac{s}{n} \right) \geq f_{i+1}^{t}\left( \frac{s+1}{n}\right) + \frac{p}{\sigma n^{1/2}}.
\end{equation}
We mention that our choice of $N_5$ is possible since the functions $f_{i+1}^{t}$ and $f_i^b$ are Lipschitz continuous with constant $(\delta_2/2)^{-1} \cdot \sigma^{-1} \left(M^{\mathsf{side}} + M^{\mathsf{bot}} + 8k \tilde{A} \right)$, while by construction 
$$\inf_{t \in [0,1]} \left(f_i^b(t) - f_{i+1}^t(t) \right) = 2\sigma^{-1} \tilde{A} + O(n^{-1/2}),$$
where the constant in the big $O$ notation depends on $p$ alone. \\

{\bf \raggedleft Step 2.} In this and the next step we fix $t_0 \in [\delta_2, 1-\delta_2]$ and proceed to show that (\ref{S41E1}) holds when $n \geq N_5$, for the $\delta^{\mathsf{sep}}, \Delta^{\mathsf{sep}}$ and $N_5$ we constructed in Step 1. Let $s_n \in F_n(t_0n)$ be such that $g(s_n) -p s_n = \max_{s \in F_n(t_0n)} [g(s) - ps]$, and define
$$\tilde{g}(s_n) = g(s_n) \mbox{ and } \tilde{g}(s) = -\infty \mbox{ for }s \in \llbracket 0,n \rrbracket \setminus \{s_n\}.$$
Using the monotone coupling in Lemma \ref{MCL} we see that it suffices to show
\begin{equation}\label{S41E7}
\mathbb{P}_{\ice, \operatorname{Geom}}^{0, n, \vec{x}\,', \vec{y}\,', \infty, \tilde{g}}\left( Q_k(t_0 n) - pt_0n \geq \tilde{g}(s_n) - ps_n + \delta^{\mathsf{sep}} \cdot n^{1/2}  \right) > 1- \eta.
\end{equation}

Let us elaborate why Lemma \ref{MCL} is applicable. From the definition of $\vec{x}\,', \vec{y}\,'$ and $\tilde{g}$ it is clear that $x_i' \leq x_i$, $y_i' \leq y_i$ for $i \in \llbracket 1, k \rrbracket$ and $\tilde{g}(s) \leq g(s)$ for $s \in \llbracket 0, n \rrbracket$. In addition, $\Omega_{\ice}(0,n,\vec{x},\vec{y}, \infty, g)$ is non-empty by assumption. Thus, we only need to check that $\Omega_{\ice}(0,n,\vec{x}\,',\vec{y}\,', \infty, \tilde{g})$ is non-empty. The latter statement would hold if we can show that
\begin{equation}\label{S41E8}
A_1^n \times A_2^n \times \cdots \times A_k^n \subseteq \Omega_{\ice}(0,n,\vec{x}\,',\vec{y}\,', \infty, \tilde{g}),
\end{equation}
since from (\ref{S41E3}) the sets $A_i^n$ are non-empty when $n \geq N_5$.

Let us verify (\ref{S41E8}) briefly. If $Q_i \in A_i^n$ for $i \in \llbracket 1, k \rrbracket$, we need to check
\begin{equation}\label{S41E9}
Q_i(s) \geq Q_{i+1}(s+1) \mbox{ for $i \in \llbracket 1, k-1 \rrbracket$ and $s \in \llbracket 0, n -1\rrbracket$, and } Q_k(s_n-1) \geq \tilde{g}(s_n).
\end{equation}
The first inequality in (\ref{S41E9}) is equivalent to 
$$\mathcal{Q}_i\left( \frac{s}{n} \right) \geq \mathcal{Q}_{i+1}\left( \frac{s+1}{n} \right) + \frac{p}{\sigma n^{1/2}},$$
which follows from (\ref{S41E6}) and the the definition of $A_i^n$. The second inequality in (\ref{S41E9}) follows from
\begin{equation*}
\begin{split}
&Q_k(s_n-1) = p(s_n-1) + \sigma n^{1/2} \mathcal{Q}_k\left(\frac{s_n-1}{n}\right) \geq p(s_n-1) +  \sigma n^{1/2} \cdot f_k^b\left(\frac{s_n-1}{n}\right) \\
& \geq  p(s_n-1) +  \sigma n^{1/2} \cdot  \sigma^{-1} (M^{\mathsf{bot}} + 3 \tilde{A}) \geq ps_n + M^{\mathsf{bot}} \cdot n^{1/2} \geq g(s_n) = \tilde{g}(s_n).
\end{split}
\end{equation*}
The last two equations imply (\ref{S41E9}), which in turn verifies (\ref{S41E8}). Overall, we have reduced the proof of the lemma to showing (\ref{S41E7}), which we do in the next and final step.\\

{\bf \raggedleft Step 3.} From (\ref{S41E3}) and (\ref{S41E8}) we have
\begin{equation*}
\mathbb{P}_{\operatorname{Geom}}^{0, n, \vec{x}\,', \vec{y}\,'}\left(\Omega_{\ice}(0,n,\vec{x}\,',\vec{y}\,', \infty, \tilde{g}) \right) \geq \epsilon_0^k.
\end{equation*}
If $E \subseteq \Omega(0,n, \vec{x}\,',\vec{y}\,')$, we conclude that
\begin{equation}\label{S41E10}
\mathbb{P}_{\ice, \operatorname{Geom}}^{0, n, \vec{x}\,', \vec{y}\,', \infty, \tilde{g}} (E) = \frac{\mathbb{P}_{\operatorname{Geom}}^{0, n, \vec{x}\,', \vec{y}\,'}(E \cap\Omega_{\ice}(0,n,\vec{x}\,',\vec{y}\,', \infty, \tilde{g}) )}{\mathbb{P}_{\operatorname{Geom}}^{0, n, \vec{x}\,', \vec{y}\,'}(\Omega_{\ice}(0,n,\vec{x}\,',\vec{y}\,', \infty, \tilde{g}))} \leq \epsilon_0^{-k} \cdot \mathbb{P}_{\operatorname{Geom}}^{0, n, \vec{x}\,', \vec{y}\,'}(E).
\end{equation}
If we define $E_1, E_2, E_3 \subseteq \Omega(0,n, \vec{x}\,',\vec{y}\,')$ by
$$E_1 = \left\{  \left(Q_k \left( s_n-1 \right) - p (s_n-1) \right) \in [\tilde{g}(s_n) - ps_n -2\delta^{\mathsf{sep}} \cdot n^{1/2}, \tilde{g}(s_n) - ps_n + 2\delta^{\mathsf{sep}} \cdot n^{1/2} ] \right\},$$
$$E_2 = \{ w(\mathcal{Q}_k, 2\Delta^{\mathsf{sep}}) > \sigma^{-1}\delta^{\mathsf{sep}}  \},\hspace{2mm} E_3 = \{ Q_k(t_0n) - pt_0n \geq \tilde{g}(s_n) - ps_n + \delta^{\mathsf{sep}} \cdot n^{1/2}   \},$$
then by (\ref{S41E4}) and (\ref{S41E5}) we have
\begin{equation*}
\mathbb{P}_{\operatorname{Geom}}^{0, n, \vec{x}\,', \vec{y}\,'}(E_1) \leq \epsilon_0^k \cdot (\eta/4) \mbox{ and } \mathbb{P}_{\operatorname{Geom}}^{0, n, \vec{x}\,', \vec{y}\,'}(E_2) \leq \epsilon_0^k \cdot (\eta/4),
\end{equation*}
which together with (\ref{S41E10}) gives
\begin{equation}\label{S41E11}
\mathbb{P}_{\ice, \operatorname{Geom}}^{0, n, \vec{x}\,', \vec{y}\,', \infty, \tilde{g}}(E_1) \leq \eta/4 \mbox{ and } \mathbb{P}_{\ice, \operatorname{Geom}}^{0, n, \vec{x}\,', \vec{y}\,', \infty, \tilde{g}}(E_2) \leq \eta/4.
\end{equation}

We claim that
\begin{equation}\label{S41E12}
E_1^c \cap E_2^c \cap \Omega_{\ice}(0,n,\vec{x}\,',\vec{y}\,', \infty, \tilde{g})  \subseteq E_3 \cap \Omega_{\ice}(0,n,\vec{x}\,',\vec{y}\,', \infty, \tilde{g}).
\end{equation}
If true, then (\ref{S41E11}) would imply $\mathbb{P}_{\ice, \operatorname{Geom}}^{0, n, \vec{x}\,', \vec{y}\,', \infty, \tilde{g}}(E_3) \geq 1- \eta/2$ which gives (\ref{S41E7}). We have thus reduced the proof of the lemma to showing (\ref{S41E12}).

Suppose $\mathfrak{Q} = \{Q_i\}_{i = 1}^k$ is in the left hand side of (\ref{S41E12}). Since $\mathfrak{Q} \in E_1^c$, we have
$$\left(Q_k \left( s_n-1 \right) - p (s_n-1) \right) \not \in [\tilde{g}(s_n) - ps_n -2\delta^{\mathsf{sep}} \cdot n^{1/2}, \tilde{g}(s_n) - ps_n + 2\delta^{\mathsf{sep}} \cdot n^{1/2} ],$$
and since $\mathfrak{Q} \in \Omega_{\ice}(0,n,\vec{x}\,',\vec{y}\,', \infty, \tilde{g})$, we have $Q_k \left( s_n-1 \right) \geq  \tilde{g}(s_n)$. The last two conditions show
\begin{equation}\label{S41E13}
Q_k ( s_n-1 )  - p(s_n-1) \geq \tilde{g}(s_n) - ps_n + 2\delta^{\mathsf{sep}} \cdot n^{1/2}.
\end{equation}
Since $s_n \in F_n(t_0n)$, we have $|(s_n -1) - t_0n | \leq 2 \Delta^{\mathsf{sep}} \cdot n $. The latter shows that for $\mathfrak{Q} \in E_2^c$ we have
\begin{equation}\label{S41E14}
\left| \left( Q_k ( s_n-1 )  - p(s_n-1) \right) -  \left( Q_k ( t_0n )  - pt_0n \right) \right| \leq \delta^{\mathsf{sep}} \cdot n^{1/2}.
\end{equation}
Equations (\ref{S41E13}) and (\ref{S41E14}) together show $\mathfrak{Q} \in E_3 \cap \Omega_{\ice}(0,n,\vec{x}\,',\vec{y}\,', \infty, \tilde{g})$. This completes the proof of (\ref{S41E12}) and hence the lemma. 
\end{proof}

%---------------------------------------------------------------------------------------------------------------
% Section 4.2
%
%----------------------------------------------------------------------------------------------------------------
\subsection{Modulus of continuity for several curves}\label{Section4.2} The goal of the section  is to prove the following lemma.

\begin{lemma}\label{S42L} Fix $p, \epsilon, \eta, M^{\mathsf{bot}}, M^{\mathsf{side}}, A^{\mathsf{sep}}, A^{\mathsf{gap}}\in (0, \infty)$, $\Delta^{\mathsf{gap}}  \in (0,1/2)$, $k \in \mathbb{N}$, $\sigma = \sqrt{p(1+p)}$. There exist $N_6 \in \mathbb{N}$ and $\delta > 0 $, depending on $p, \epsilon, \eta, M^{\mathsf{bot}}, M^{\mathsf{side}}, A^{\mathsf{sep}}, A^{\mathsf{gap}}, \Delta^{\mathsf{gap}}, k$, such that the following holds for all $n \geq N_6$. If we assume:
\begin{itemize}
\item $\vec{x}, \vec{y}\in \mathfrak{W}_k$ as in (\ref{DefSig}), and $|x_i| \leq M^{\mathsf{side}} \cdot n^{1/2}$, $|y_i - pn| \leq M^{\mathsf{side}} \cdot n^{1/2}$ for $i \in \llbracket 1, k \rrbracket$;
\item $x_i - x_{i+1} \geq A^{\mathsf{sep}} \cdot n^{1/2}$ and $y_i - y_{i+1} \geq A^{\mathsf{sep}} \cdot n^{1/2}$ for $i \in \llbracket 1, k -1 \rrbracket$;
\item $g: \llbracket 0, n \rrbracket \rightarrow [-\infty, \infty)$ is such that $g(s) - ps\leq  M^{\mathsf{bot}} \cdot n^{1/2} $ for all $s \in \llbracket 0, n \rrbracket$;
\item $x_k \geq g(s) - ps + A^{\mathsf{gap}} \cdot n^{1/2} $ for all $s \in \llbracket 0, n\rrbracket$ with $s \leq \Delta^{\mathsf{gap}} \cdot n$;
\item $y_k -pn \geq g(s) - ps + A^{\mathsf{gap}} \cdot n^{1/2} $ for all $s \in \llbracket 0, n\rrbracket$ with $s \geq n - \Delta^{\mathsf{gap}} \cdot n$;
\item the set $\Omega_{\ice}(0,n,\vec{x},\vec{y}, \infty, g)$ is non-empty;
\item for $Q_i \in \Omega(0,n, x,y)$, we set $\mathcal{Q}_i(t)= \sigma^{-1}n^{-1/2}(Q_i(tn) - ptn)$ for $t \in [0,1]$;
\end{itemize}
then the following inequality holds for each $i \in \llbracket 1, k \rrbracket$
\begin{equation}\label{S42E1}
\mathbb{P}_{\ice, \operatorname{Geom}}^{0, n, \vec{x}, \vec{y}, \infty, g}\left( w(\mathcal{Q}_i, \delta) > \epsilon \right) < \eta.
\end{equation}
We recall that $w(f,\delta)$ is the modulus of continuity from (\ref{S33E1}).
\end{lemma}
\begin{remark} In plain words, the lemma states that if the bottom boundary $g$ is not too high, and the side boundaries $\vec{x}, \vec{y}$ are well-behaved (i.e. their coordinates are not too high or low, and are well-separated from each other and the bottom boundary $g$), then the curves in the ensemble with law $\mathbb{P}_{\ice, \operatorname{Geom}}^{0, n, \vec{x}, \vec{y}, \infty, g}$ have a well-behaved modulus of continuity.
\end{remark}

\begin{proof} The proof of the lemma has similar ideas to that of Lemma \ref{S41L}. Namely, we construct likely events $A_i^n \subseteq \Omega_{\ice}(0,n, x_i, y_i)$, which are events that paths stay inside of corridors of the type in Lemma \ref{S34L}. We show for these that $A_1^n \times \cdots \times A_k^n \subseteq \Omega_{\ice}(0,n,\vec{x},\vec{y}, \infty, g)$, which in turn shows $\mathbb{P}_{\operatorname{Geom}}^{0, n, \vec{x}, \vec{y}}\left(\Omega_{\ice}(0,n,\vec{x},\vec{y}, \infty, g) \right)$ is bounded away from zero. The latter allows us to compare events for interlacing bridges with non-interlacing ones. Since from Lemma \ref{S33L} the modulus of continuity of free bridges is well-behaved, we can conclude the same for the interlacing bridges in the statement of the lemma. We now turn to the details of the proof, which is split in two steps for clarity.\\

{\bf \raggedleft Step 1.} Fix $\tilde{A} = \min(A^{\mathsf{sep}}, A^{\mathsf{gap}})/4$ and define the functions
$$f^t_i(s) = f(s|\sigma^{-1}\tilde{A}, \sigma^{-1} (M^{\mathsf{bot}} +M^{\mathsf{side}}  + 4 (k-i + 1)\tilde{A}), \sigma^{-1} n^{-1/2} x_i, \sigma^{-1} n^{-1/2}(y_i- pn), \Delta^{\mathsf{gap}}/2) $$
$$f^b_i(s) = f(s|-\sigma^{-1}\tilde{A}, \sigma^{-1} (M^{\mathsf{bot}} +M^{\mathsf{side}}  + 4 (k-i + 1)\tilde{A}), \sigma^{-1} n^{-1/2} x_i, \sigma^{-1} n^{-1/2}(y_i- pn), \Delta^{\mathsf{gap}}/2), $$
where we recall that $f(s|A,B,x,y,\delta)$ is as in (\ref{S34E1}).

Applying Lemma \ref{S34L} with $p$ as in the present lemma, $\delta_1 =\Delta^{\mathsf{gap}}/2$, $A = \sigma^{-1}\tilde{A}$, $M_1 = M^{\mathsf{side}}$, $M_2 = \sigma^{-1}(M^{\mathsf{bot}} +M^{\mathsf{side}} + 4k \tilde{A}) $, we conclude that there exist $N_4 \in \mathbb{N}$ and $\epsilon_0 > 0$ such that for $n \geq N_4$ and $i \in \llbracket 1, k \rrbracket$
\begin{equation}\label{S42E2}
\mathbb{P}_{\operatorname{Geom}}^{0,n, x_i, y_i} \left( f^t_i(s) \geq \mathcal{Q}_1(s) \geq f^b_i(s) \mbox{ for all } s\in [0,1] \right) > \epsilon_0.
\end{equation}
If $A^n_i$ is the subset of increasing paths $Q_1 \in \Omega(0,n, x_i', y_i')$ that satisfy the conditions in (\ref{S42E2}), then (\ref{S42E2}) is equivalent to 
\begin{equation}\label{S42E3}
\frac{|A_i^n|}{| \Omega(0,n, x_i, y_i)|} > \epsilon_0.
\end{equation}

Applying Lemma \ref{S33L} with $p, \epsilon$ as in the present lemma, $M = M^{\mathsf{side}}$, and $\eta =\epsilon_0^k \cdot \eta$, we can find $N_3 \in \mathbb{N}$ and $\delta > 0$ such that for $n \geq N_3$ and $i \in \llbracket 1, k \rrbracket$
\begin{equation}\label{S42E4}
\mathbb{P}_{\operatorname{Geom}}^{0,n, x_i, y_i} \left( w(\mathcal{Q}_1, \delta) > \epsilon  \right) < \epsilon_0^k \cdot \eta.
\end{equation}
This specifies our choice of $\delta$ in the lemma. We also let $N_6$ be sufficiently large so that for $n\geq N_6$ we have $n \geq \max(N_3, N_4)$ as above, $\tilde{A} n^{1/2} \geq p$, $\Delta^{\mathsf{gap}}/2 \geq 1/n$ and for $s \in \llbracket 0, n-1\rrbracket$ and $i \in \llbracket 1, k-1\rrbracket$
\begin{equation}\label{S42E5}
f_i^{b}\left( \frac{s}{n} \right) \geq f_{i+1}^{t}\left( \frac{s+1}{n}\right) + \frac{p}{\sigma n^{1/2}}.
\end{equation}
We mention that our choice of $N_6$ is possible since the functions $f_{i+1}^{t}$ and $f_i^b$ are Lipschitz continuous with constant $(\Delta^{\mathsf{gap}}/2)^{-1} \cdot \sigma^{-1} \left(M^{\mathsf{side}} + M^{\mathsf{bot}} + 4k \tilde{A} \right)$, while by construction 
$$\inf_{t \in [0,1]} \left(f_i^b(t) - f_{i+1}(t) \right) = 2\sigma^{-1} \tilde{A} + O(n^{-1/2}),$$
where the constant in the big $O$ notation depends on $p$ alone. This specifies our choice of $N_6$.\\

We proceed to prove (\ref{S42E1}) for fixed $i \in \llbracket 1, k \rrbracket$, $n \geq N_6$ and $\delta$ as above. We claim that 
\begin{equation}\label{S42E6}
A_1^n \times A_2^n \times \cdots \times A_k^n \subseteq \Omega_{\ice}(0,n,\vec{x},\vec{y}, \infty, g).
\end{equation}
We prove (\ref{S42E6}) in the second step. Here, we assume its validity and conclude the proof of (\ref{S42E1}).

Using (\ref{S42E3}) and (\ref{S42E6}) we have for each $E \subseteq \Omega(0,n, \vec{x},\vec{y})$ 
\begin{equation}\label{S42E7}
\mathbb{P}_{\ice, \operatorname{Geom}}^{0, n, \vec{x}, \vec{y}, \infty, g} (E) = \frac{\mathbb{P}_{\operatorname{Geom}}^{0, n, \vec{x}, \vec{y}}(E \cap\Omega_{\ice}(0,n,\vec{x},\vec{y}, \infty, g) )}{\mathbb{P}_{\operatorname{Geom}}^{0, n, \vec{x}, \vec{y}}(\Omega_{\ice}(0,n,\vec{x},\vec{y}, \infty, g))} \leq \epsilon_0^{-k} \cdot \mathbb{P}_{\operatorname{Geom}}^{0, n, \vec{x}, \vec{y}}(E).
\end{equation}
Letting $E$ be the set of $\mathfrak{Q} = \{Q_j\}_{j = 1}^k$ such that $w(\mathcal{Q}_i, \delta) > \epsilon$ in (\ref{S42E7}) and using (\ref{S42E4}) gives (\ref{S42E1}).\\

{\bf \raggedleft Step 2.} In this step we verify (\ref{S42E6}). If $Q_i \in A_i^n$ for $i \in \llbracket 1, k \rrbracket$, we need to check
\begin{equation}\label{S42E8}
\begin{split}
&Q_i(s) \geq Q_{i+1}(s+1) \mbox{ for $i \in \llbracket 1, k-1 \rrbracket$ and $s \in \llbracket 0, n -1\rrbracket$, and } \\
& Q_k(s) \geq g(s+1) \mbox{ for } s \in \llbracket 0, n -1\rrbracket.
\end{split}
\end{equation}
The first line in (\ref{S42E8}) is equivalent to 
$$\mathcal{Q}_i\left( \frac{s}{n} \right) \geq \mathcal{Q}_{i+1}\left( \frac{s+1}{n} \right) + \frac{p}{\sigma n^{1/2}},$$
which follows from (\ref{S42E5}) and the the definition of $A_i^n$. In the remainder of this step we verify the second line in (\ref{S42E8}).\\

Suppose first that $s \in [0, (\Delta^{\mathsf{gap}}/2) \cdot n]$ and so $s + 1 \in [0, \Delta^{\mathsf{gap}}\cdot n]$. We have 
\begin{equation*}
\begin{split}
&Q_k(s) = \sigma n^{1/2} \cdot \mathcal{Q}_k(s/n) + ps \geq \sigma n^{1/2} \cdot f_k^b(s/n) + ps \geq \sigma n^{1/2} \cdot f_k^b(0) + ps \\
& = x_k - \tilde{A} \cdot n^{1/2} + ps \geq [g(s + 1) - p(s+1) +  A^{\mathsf{gap}} \cdot n^{1/2} ]- \tilde{A} \cdot n^{1/2} + ps \geq g(s + 1),
\end{split}
\end{equation*}
where in the first inequality we used $Q_k \in A_k^n$, in the second one and in going from the first to the second line we used the definition of $f_k^b$, in the first inequality on the second line we used the $x_k$ lower bound in the statement of the lemma, and in the last inequality we used that $A^{\mathsf{gap}} \geq 4 \tilde{A}$.

Suppose next that $s \in [(1-\Delta^{\mathsf{gap}}/2) \cdot n, n-1]$. We have 
\begin{equation*}
\begin{split}
&Q_k(s) = \sigma n^{1/2} \cdot \mathcal{Q}_k(s/n) + ps \geq \sigma n^{1/2} \cdot f_k^b(s/n) + ps \geq \sigma n^{1/2} \cdot f_k^b(1) + ps \\
& = y_k - pn - \tilde{A} \cdot n^{1/2} + ps \geq [g(s + 1) - p(s+1) +  A^{\mathsf{gap}} \cdot n^{1/2} ]- \tilde{A} \cdot n^{1/2} + ps \geq g(s + 1),
\end{split}
\end{equation*}
where in the first inequality we used $Q_k \in A_k^n$, in the second one and in going from the first to the second line we used the definition of $f_k^b$, in the first inequality on the second line we used the $y_k$ lower bound in statement of the lemma, and in the last inequality we used that $A^{\mathsf{gap}} \geq 4 \tilde{A}$.

Finally, suppose that $s \in [(\Delta^{\mathsf{gap}}/2) \cdot n, (1-\Delta^{\mathsf{gap}}/2) \cdot n]$. We have
\begin{equation*}
\begin{split}
&Q_k(s) = \sigma n^{1/2} \cdot \mathcal{Q}_k(s/n) + ps \geq \sigma n^{1/2} \cdot f_k^b(s/n) + ps =  n^{1/2} \cdot (M^{\mathsf{bot}} + M^{\mathsf{side}} +3\tilde{A})  + ps \\
& \geq [g(s + 1) - p(s+1)] + n^{1/2} \cdot ( M^{\mathsf{side}} +3\tilde{A})  + ps \geq g(s + 1),
\end{split}
\end{equation*}
where in the first inequality we used $Q_k \in A_k^n$, in the second one we used the definition of $f_k^b$, and in going from the first to the second line we used the upper bound for $g$ in the statement of the lemma. The last three displayed equations prove the second line in (\ref{S42E8}), and hence the lemma.
\end{proof}

%---------------------------------------------------------------------------------------------------------------
% Section 5
%
%----------------------------------------------------------------------------------------------------------------
\section{General conditions for tightness}\label{Section5} We state the main result of the paper, after which we deduce Theorem \ref{S1T1} from it. We continue with the same notation as in Section \ref{Section2}.

\begin{theorem}\label{S5T} Fix $p \in (0, \infty)$, $\sigma = \sqrt{p(1+p)}$, $K \in \mathbb{N} \cup \{\infty\}$, $\Sigma = \llbracket 1, K + 1 \rrbracket$ and $\Lambda = (\alpha,\beta) \subseteq \mathbb{R}$ to be any non-empty open interval. We further suppose that we have sequences $d_N \in (0,\infty)$, $\hat{A}_N, \hat{B}_N \in \mathbb{Z}$ with $\hat{A}_N \leq \hat{B}_N$, $K_N \in \Sigma \cup \{\infty\}$ and $\Sigma$-indexed geometric line ensembles $\mathfrak{L}^N = \{L^N_i\}_{i \in \Sigma}$ on $\mathbb{Z}$ that satisfy the following conditions:
\begin{itemize}
\item $d_N \rightarrow \infty$, $\hat{A}_N/d_N \rightarrow \alpha$, $\hat{B}_N/d_N \rightarrow \beta$, $K_N \rightarrow K+1$ as $N \rightarrow \infty$;
\item for each $t \in \Lambda$ and $i \in \llbracket 1, K \rrbracket$ the random variables $\sigma^{-1} d_N^{-1/2}  \left( L_i^N(\lfloor t d_N \rfloor) - ptd_N \right)$ are tight;
\item the restriction $\{L_i^N(s): i \in \llbracket 1, K_N \rrbracket$ and $s \in \llbracket \hat{A}_N, \hat{B}_N\rrbracket \}$, satisfies the interlacing Gibbs property from Definition \ref{DefSGP}.
\end{itemize}
Then, the sequence of line ensembles $\mathcal{L}^N = \{\mathcal{L}_i^N\}_{i = 1}^{K} \in  C(\llbracket 1, K \rrbracket \times \Lambda)$, defined through 
$\mathcal{L}_i^N(t) = \sigma^{-1} d_N^{-1/2}  \left( L_i^N(td_N) - ptd_N \right),$
is tight. Moreover, any subsequential limit satisfies the partial Brownian Gibbs property from \cite[Definition 2.7]{DM21}.
\end{theorem}

\begin{proof}[Proof of Theorem \ref{S1T1}] Define $\tilde{\mathfrak{L}}^N = \{\tilde{L}^N_i\}_{i \geq 1}$ via $\tilde{L}_i^N(t) = L_i^N(t) - \lfloor C_N \rfloor$ for $t \in \mathbb{R}$ and $N \geq 1$. We also define $\tilde{\mathcal{L}}^N = \{\tilde{\mathcal{L}}^N_i\}_{i \geq 1}$ via $\tilde{\mathcal{L}}_i^N(t) = \sigma^{-1} d_N^{-1/2} (\tilde{L}_i^N(td_N) -  ptd_N)$. From Assumptions 1 and 2, we see that $\tilde{\mathfrak{L}}^N$ satisfies the conditions of Theorem \ref{S5T} with the same $p, d_N, \hat{A}_N, \hat{B}_N$ as in the present theorem, $K = K_N = \infty$, and $\Lambda = \mathbb{R}$. Indeed, all but the third point in Theorem \ref{S5T} are trivially satisfied. Also, from Assumption 2, we have for each $m \geq 1$ that $(\tilde{L}^N_i(s): i \in \llbracket 1, m\rrbracket, s \in \llbracket \hat{A}_N, \hat{B}_N \rrbracket)$ is a convex combination of measures of the form $\mathbb{P}_{\ice, \operatorname{Geom}}^{\hat{A}_N, \hat{B}_N, \vec{x}, \vec{y}, \infty, g}$ as in Definition \ref{DefSGP}, with different $\vec{x}, \vec{y}, g$. The latter and Lemma \ref{Lem.FinEnsSatGB} show that $\{\tilde{L}_i^N(s): i \in \llbracket 1, \infty \rrbracket$ and $s \in \llbracket \hat{A}_N, \hat{B}_N\rrbracket \}$ satisfies the interlacing Gibbs property from Definition \ref{DefSGP}, verifying the third point in Theorem \ref{S5T}.

From Theorem \ref{S5T} we conclude that $\tilde{\mathcal{L}}^N $ is tight and any subsequential limit satisfies the partial Brownian Gibbs property. Since for $i \in \mathbb{N}, t \in \mathbb{R}$
$$ \left|\mathcal{L}^N_i(t) - \tilde{\mathcal{L}}_i^N(t)\right| = \sigma^{-1} d_N^{-1/2} (C_N - \lfloor C_N \rfloor) \leq \sigma^{-1} d_N^{-1/2} \rightarrow 0,$$
we conclude the same for $\mathcal{L}^N$.
\end{proof}

The proof of Theorem \ref{S5T} is the content of the remainder of this section, and in all statements we make the same assumptions as in the statement of the theorem. We mention that the tightness assumption in Theorem \ref{S5T} ensures the existence of functions $\psi(\cdot |i, t): (0,\infty) \rightarrow (0, \infty)$ such that for each $i \in \Sigma$, $t \in \Lambda$, $\epsilon \in (0, \infty)$ we have for all $N \in \mathbb{N}$
\begin{equation}\label{S5E1}
\mathbb{P}\left( \left| L_i^N(\lfloor t d_N \rfloor) - ptd_N \right| > d_N^{1/2} \cdot \psi(\epsilon| i, t) \right) \leq \epsilon.
\end{equation}
Throughout our proofs in the following sections we will encounter various constants that depend on $p$, the sequences $d_N, \hat{A}_N, \hat{B}_N$ in the statement of Theorem \ref{S5T} and also the functions $\psi(\cdot|i,t)$. We will not list this dependence explicitly. 

%---------------------------------------------------------------------------------------------------------------
% Section 5.1
%
%----------------------------------------------------------------------------------------------------------------
\subsection{No big max}\label{Section5.1} The goal of this section is to prove the following lemma.

\begin{lemma}\label{S51L} For any $a, b \in \Lambda$ with $a < b$, and $\varepsilon > 0$ we can find $W_1 \in \mathbb{N}$ and $M^{\mathsf{top}} > 0$, depending on $a,b$ and $\varepsilon$, such that for $N \geq W_1$ we have 
\begin{equation}\label{S51E1}
\mathbb{P}\left( \max_{t \in [a,b]} \left(L_1^N(td_N) - p td_N \right) \geq d_N^{1/2} \cdot M^{\mathsf{top}}  \right) < \varepsilon.
\end{equation}
\end{lemma}
\begin{proof} The proof we present is similar to \cite[Lemma 4.2]{DEA21}. Let us fix $d \in \Lambda$ with $d > b$, and set $A_N = \lfloor a \cdot d_N \rfloor$, $D_N = \lfloor d \cdot d_N \rfloor$, $M_N = \lfloor (1/2)(a+d) \cdot d_N \rfloor$. We show below that we can find $M^{\mathsf{top}} > 0$, depending on $a,d, \varepsilon$, such that for all large $N$
\begin{equation}\label{S51E2}
\begin{split}
&\mathbb{P}\left( \max_{s \in [A_N,M_N]} \left(L_1^N(s) - p s \right) \geq M^{\mathsf{top}} \cdot d_N^{1/2} \right) < \varepsilon/2 \mbox{, and } \\
&\mathbb{P}\left( \max_{s \in [M_N,D_N]} \left(L_1^N(s) - p s \right) \geq M^{\mathsf{top}} \cdot d_N^{1/2}   \right) < \varepsilon/2.
\end{split}
\end{equation}
Since for all large $N$ we have $d_N \cdot [a,b] \subset [A_N, D_N]$, we see that (\ref{S51E2}) implies (\ref{S51E1}). In the remainder we only establish the second line in (\ref{S51E2}), the first one being handled analogously.\\

Define the events
$$E(R) = \{|L_1^N(A_N) - p A_N| \geq R  d_N^{1/2} \}, \hspace{2mm} F(R) = \{L_1^N(M_N) - pM_N \geq R  d_N^{1/2} \}, \hspace{2mm} G(R) = \sqcup_{ m \in M_N}^{D_N} G_m(R),$$
$$\mbox{ with } G_m(R) = \{ L_1^N(m) - p m \geq (5R+3) d_N^{1/2} \mbox{ and } L_1^N(u) - p u < (5R+3)  d_N^{1/2} \mbox{ for } u \in \llbracket m+1, D_N \rrbracket \}.$$
From (\ref{S5E1}) we can find $R > 0$ large enough so that
\begin{equation}\label{S51E3}
\mathbb{P}(E(R)) < \varepsilon/ 4 \mbox{ and } \mathbb{P}(F(R)) < \varepsilon/12.
\end{equation}
For all large $N$ we have that $\llbracket A_N, D_N \rrbracket \subseteq \llbracket \hat{A}_N, \hat{B}_N \rrbracket$ and by the interlacing Gibbs property, see Lemma \ref{Lem.StrongGP}, we get 
\begin{equation}\label{S51E4}
\mathbb{P}(E(R)^c \cap F(R) \cap G_m(R)) = \mathbb{E} \left[{\bf 1}_{E(R)^c} \cdot {\bf 1}_{G_m(R)} \cdot \mathbb{P}_{\ice, \operatorname{Geom}}^{A_N, m, x, y, \infty, g} \left( Q_1(M_N) - pM_N \geq Rd_N^{1/2} \right) \right],
\end{equation}
where $x = L^N_1(A_N)$, $y = L^N_1(m)$, $g = L^N_2\llbracket A_N ,m \rrbracket$. Setting $\tilde{x} = x - \lfloor pA_N \rfloor$, $\tilde{y} = y - \lfloor pA_N \rfloor$, $\tilde{n} = m - A_N$, $\tilde{s} = M_N - A_N$, $\tilde{g}(s) = g(s+A_N) - \lfloor p A_N \rfloor$ for $s \in \llbracket 0, m - A_N \rrbracket$, we get by translation
\begin{equation*}
\begin{split}
 \mathbb{P}_{\ice, \operatorname{Geom}}^{A_N, m, x, y, \infty, g} \left( Q_1(M_N) - pM_N \geq Rd_N^{1/2} \right) =  \mathbb{P}_{\ice, \operatorname{Geom}}^{0, \tilde{n}, \tilde{x},  \tilde{y}, \infty, \tilde{g}} \left( Q_1(\tilde{s}) - p \tilde{s} + \lfloor p A_N \rfloor - pA_N \geq Rd_N^{1/2} \right).
 \end{split}
\end{equation*}

From Lemma \ref{S32L} applied to $M = (12R+6) \cdot (d-a)^{-1/2}$, $M_1 = \tilde{n}^{-1/2} \left( -R d_N^{1/2} + pA_N - \lfloor p A_N \rfloor \right)$, $M_2 = \tilde{n}^{-1/2} \left( (5R+3) d_N^{1/2} + pA_N - \lfloor p A_N \rfloor \right)$, we conclude that if $x \geq p A_N - Rd_N^{1/2}, y \geq pm + (5R+3)d_N^{1/2}$ (and hence $\tilde{x} \geq M_1 \tilde{n}^{1/2} $, $\tilde{y} \geq p \tilde{n} + M_2 \tilde{n}^{1/2}$) and $N$ is large enough
$$\mathbb{P}_{\ice, \operatorname{Geom}}^{0,\tilde{n}, \tilde{x}, \tilde{y}, \infty, \tilde{g}} \left( Q_1(\tilde{s}) \geq \frac{\tilde{n} - \tilde{s}}{\tilde{n}} \cdot \left(-R d_N^{1/2} \right) + \frac{\tilde{s}}{\tilde{n}} \cdot \left(p \tilde{n} + (5R + 3)d_N^{1/2} \right) - d_N^{1/2} \right) \geq \frac{1}{3}.$$
Combining the last two displayed equations with the fact that $\tilde{s}/\tilde{n} \in [1/3, 1]$ for all large $N$, when $m \in \llbracket M_N, D_N \rrbracket$, we conclude for $x \geq p A_N - Rd_N^{1/2}$ and $y \geq pm + (5R+3)d_N^{1/2}$
\begin{equation*}
\begin{split}
 \mathbb{P}_{\ice, \operatorname{Geom}}^{A_N, m, x, y, \infty, g} \left( Q_1(M_N) - pM_N \geq Rd_N^{1/2} \right) \geq 1/3.
 \end{split}
\end{equation*}
Using the definitions of $E(R)$ and $G_m(R)$ we conclude that for all large $N$ we have a.s.
\begin{equation}\label{S51E5}
{\bf 1}_{E(R)^c} \cdot {\bf 1}_{G_m(R)} \cdot \mathbb{P}_{\ice, \operatorname{Geom}}^{A_N, m, x, y, \infty, g} \left( Q_1(M_N) - pM_N \geq Rd_N^{1/2} \right) \geq  {\bf 1}_{E(R)^c} \cdot {\bf 1}_{G_m(R)} \cdot (1/3).
\end{equation}
Taking expectations on both sides of (\ref{S51E5}), using (\ref{S51E4}), and summing over $m \in \llbracket M_N, D_N\rrbracket$ we get
\begin{equation}\label{S51E6}
\mathbb{P}(E(R)^c \cap F(R) \cap G(R)) \geq (1/3) \cdot \mathbb{P}(E(R)^c \cap G(R)).
\end{equation}
Combining (\ref{S51E3}) with (\ref{S51E6}) we conclude
\begin{equation}\label{S51E7}
\mathbb{P}(G(R)) \leq \mathbb{P}(E(R)) + \mathbb{P}(G(R) \cap E(R)^c) < \varepsilon/4 + 3\mathbb{P}(E(R)^c \cap F(R) \cap G(R)) < \varepsilon/2,
\end{equation}
which implies the second line in (\ref{S51E2}) with $M^{\mathsf{top}} = 5R + 3$.
\end{proof}

%---------------------------------------------------------------------------------------------------------------
% Section 5.2
%
%----------------------------------------------------------------------------------------------------------------
\subsection{Likely curve separation}\label{Section5.2} The goal of this section is to prove the following lemma.

\begin{lemma}\label{S52L} For any $k \in \llbracket 1, K \rrbracket$, $a, b \in \Lambda$ with $a < b$, and $\varepsilon \in (0,1)$ we can find $W_2 \in \mathbb{N}$, $\delta^{\mathsf{sep}}, \Delta^{\mathsf{sep}} > 0$, depending on $a,b, k$ and $\varepsilon$, such that for $N \geq W_2$, and $ s_0 \in \mathbb{Z} \cap [a \cdot d_N,b \cdot d_N]$ we have 
\begin{equation}\label{S52E1}
\begin{split}
&\mathbb{P}\left( \cap_{m = 1}^k E^{\mathsf{sep}}_m  \right) > 1- \varepsilon, \mbox{ where } E^{\mathsf{sep}}_m = \Big\{ L_m^N(s_0) - ps_0 \geq L_{m+1}^N(s) - ps + \delta^{\mathsf{sep}} \cdot d_N^{1/2} \\
& \mbox{ for all } m \in \llbracket 1, k \rrbracket \mbox{, } s \in \mathbb{Z} \cap [ s_0 - \Delta^{\mathsf{sep}} \cdot d_N,s_0 + \Delta^{\mathsf{sep}} \cdot d_N] \Big\}.
\end{split}
\end{equation}
\end{lemma}
\begin{proof} For clarity we split the proof into two steps. In the first we specify the parameters $W_2, \delta^{\mathsf{sep}}$ and $\Delta^{\mathsf{sep}}$ in the statement of the lemma, and in the second step we show that they satisfy (\ref{S52E1}).\\

{\bf \raggedleft Step 1.} Let $\epsilon > 0$ be small enough so that $b -a > 2\epsilon$, $c = a - 2\epsilon \in \Lambda$, and $d = b + 2 \epsilon \in \Lambda$. We put $C_N = \lfloor c \cdot d_N \rfloor$, $D_N = \lfloor d \cdot d_N \rfloor$, and let $W_{2,1} \in \mathbb{N}$ be sufficiently large so that for $N \geq W_{2,1}$ and $ s_0 \in \mathbb{Z} \cap [a \cdot d_N,b \cdot d_N]$ we have 
\begin{equation}\label{S52E2}
 K_N \geq k+1, \hspace{2mm} \hat{A}_N \leq C_N, \hspace{2mm} \hat{B}_N \geq D_N, \hspace{2mm} \min \left(\frac{s_0 - C_N}{D_N - C_N},  \frac{D_N - s_0}{D_N - C_N} \right) \geq \frac{\epsilon}{d-c} \in (0, 1/3).
\end{equation}
Let us set $\tilde{n} = D_N - C_N$ and note that from (\ref{S5E1}) we can find $W_{2,2} \in \mathbb{N}$ and $M^{\mathsf{side}} > 0$ such that for $N \geq W_{2,2}$ we have
\begin{equation}\label{S52E3}
\begin{split}
&\mathbb{P} \left( E^{\mathsf{side}}\right) \geq 1- \varepsilon/(4k), \mbox{ where } E^{\mathsf{side}} = \left\{ \left|L_i^{N}(C_N) - \lfloor p C_N \rfloor \right| \leq M^{\mathsf{side}} \cdot \tilde{n}^{1/2} \mbox{ for } i \in \llbracket 1, k\rrbracket \right\} \\ 
& \cap \left\{ \left|L_i^{N}(D_N) - \lfloor p C_N \rfloor - p \tilde{n} \right| \leq M^{\mathsf{side}} \cdot \tilde{n}^{1/2} \mbox{ for } i \in \llbracket 1, k\rrbracket \right\}.
\end{split}
\end{equation}
From Lemma \ref{S51L} we can find $W_{2,3} \in \mathbb{N}$ and $M^{\mathsf{top}} > 0$ such that for $N \geq W_{2,3}$ we have
\begin{equation}\label{S52E3.5}
\begin{split}
&\mathbb{P} \left( \max_{u \in [C_N, D_N] } \left( L_1^N(u) - p u + pC_N - \lfloor p C_N \rfloor  \right) \leq M^{\mathsf{top}} \cdot \tilde{n}^{1/2} \right) > 1 - \varepsilon/(4k).
\end{split}
\end{equation}

For $m \in \llbracket 1, k \rrbracket$, we let $N_5(m)$, $\delta^{\mathsf{sep}}(m), \Delta^{\mathsf{sep}}(m)$ be as in Lemma \ref{S41L} applied to $p$ as in the present setup, $\eta = \varepsilon/(4k)$, $M^{\mathsf{bot}}$ and $M^{\mathsf{side}}$ as above, $k = m$ and $\delta_2 = \frac{\epsilon}{d-c}$. In particular, if 
\begin{itemize}
\item $\vec{x}, \vec{y}\in \mathfrak{W}_m$, and $|x_i| \leq M^{\mathsf{side}} \cdot n^{1/2}$, $|y_i - pn| \leq M^{\mathsf{side}} \cdot n^{1/2}$ for $i  \in \llbracket 1, m \rrbracket$;
\item $g: \llbracket 0, n \rrbracket \rightarrow [-\infty, \infty)$ is such that $g(s) \leq ps + M^{\mathsf{bot}} \cdot n^{1/2} $ for all $s \in \llbracket 0, n \rrbracket$;
\item the set $\Omega_{\ice}(0,n,\vec{x},\vec{y}, \infty, g)$ is non-empty;
\item $t_0 \in [\delta_2, 1- \delta_2]$, and $F^m_n(t_0 n ) = \mathbb{Z} \cap [t_0n -\Delta^{\mathsf{sep}}(m) \cdot n, t_0 n +\Delta^{\mathsf{sep}}(m) \cdot n]$,
\end{itemize}
then the following inequality holds for $n \geq N_5(m)$
\begin{equation}\label{S52E4}
\mathbb{P}_{\ice, \operatorname{Geom}}^{0, n, \vec{x}, \vec{y}, \infty, g}\left( Q_m(t_0n) - pt_0n \geq g(s) - ps + \delta^{\mathsf{sep}}(m) \cdot n^{1/2}  \mbox{ for all } s \in F^m_n (t_0 n)\right) > 1- \epsilon/(4k).
\end{equation}

We now let $W_2 \in \mathbb{N}$ to be large enough so that $W_2 \geq \max(W_{2,1}, W_{2,2}, W_{2,3})$ and for $N \geq W_2$ we have $\tilde{n} \geq \max_{m \in \llbracket 1, k \rrbracket} N_5(m)$. We also let $\delta^{\mathsf{sep}}, \Delta^{\mathsf{sep}} > 0 $ be small enough so that for $N \geq W_2$
\begin{equation}\label{S52E5}
\delta^{\mathsf{sep}} \cdot d_N^{1/2} \leq \min_{m \in \llbracket 1, k \rrbracket} \delta^{\mathsf{sep}}(m) \cdot \tilde{n}^{1/2} \mbox{ and } \Delta^{\mathsf{sep}} \cdot d_N \leq \min_{m \in \llbracket 1, k \rrbracket} \Delta^{\mathsf{sep}}(m) \cdot \tilde{n}.    
\end{equation}
The above specifies our choice of $W_2, \delta^{\mathsf{sep}}$ and $\Delta^{\mathsf{sep}}$ in the statement of the lemma.\\

{\bf \raggedleft Step 2.} In this step we fix $s_0 \in \mathbb{Z} \cap  [a \cdot d_N,b \cdot d_N]$, $N \geq W_2$ and proceed to prove (\ref{S52E1}). It suffices to show that for a fixed $m \in \llbracket 1, k \rrbracket$ we have
\begin{equation}\label{S52E6}
\mathbb{P}\left(  E^{\mathsf{sep}}_m  \right) > 1- \varepsilon/k.
\end{equation}
If we define
\begin{equation*}
\begin{split}
\tilde{E}^{\mathsf{sep}}_m = \left\{ L^N_m(s_0) - p s_0 \geq L_{m+1}^N(s) - p s + \delta^{\mathsf{sep}}(m) \cdot\tilde{n}^{1/2} \mbox{ for } s \in [s_0 - \Delta^{\mathsf{sep}}(m)  \tilde{n}, s_0 + \Delta^{\mathsf{sep}}(m) \tilde{n}]  \right\},
\end{split}
\end{equation*}
then in view of (\ref{S52E5}) we have $\tilde{E}^{\mathsf{sep}}_m \subseteq E^{\mathsf{sep}}_m$ and so it suffices to show 
\begin{equation}\label{S52E7}
\mathbb{P}\left(  \tilde{E}^{\mathsf{sep}}_m  \right) > 1- \varepsilon/k.
\end{equation}

Let us define
$$E_m^{\mathsf{top}} = \left\{\max_{u \in [C_N, D_N] } \left(L_m^N(u) - p u + pC_N - \lfloor p C_N \rfloor  \right) \leq M^{\mathsf{top}} \cdot \tilde{n}^{1/2} \right\},$$
and note that by the interlacing Gibbs property we have $L_i^N(s) \geq L_{i+1}^N(s)$ for all $i \in \llbracket 1, k \rrbracket$ and $s \in \llbracket C_N, D_N\rrbracket$, and so $E_1^{\mathsf{top}} \subseteq E_m^{\mathsf{top}}$. In view of (\ref{S52E3.5}) conclude for $N \geq W_2$ 
\begin{equation}\label{S52E8}
\begin{split}
&\mathbb{P} \left( E_m^{\mathsf{top}}  \right) > 1 - \varepsilon/(4k).
\end{split}
\end{equation}

Let us set 
\begin{equation}\label{S52E9}
\begin{split}
&\vec{x} = \left(L^N_1(C_N) - \lfloor p C_N \rfloor, \dots, L^N_m(C_N) - \lfloor p C_N \rfloor \right),  \\
&\vec{y} = \left(L^N_1(D_N) - \lfloor p C_N \rfloor , \dots, L^N_m(D_N) - \lfloor p C_N \rfloor\right), \\
&g(s) = L^{N}_{m+1}(s + C_N) - \lfloor p C_N \rfloor \mbox{ for } s \in \llbracket 0, \tilde{n} \rrbracket
\end{split}
\end{equation}

By the interlacing Gibbs property (see Lemma \ref{Lem.StrongGP}) and translation we have for $N \geq W_2$ that
\begin{equation}\label{S52E10}
\begin{split}
&\mathbb{P} \left(E^{\mathsf{side}} \cap E^{\mathsf{top}}_{m+1} \cap \tilde{E}^{\mathsf{sep}}\right) = \mathbb{E} \Big[ {\bf 1}_{E^{\mathsf{side}}} \cdot {\bf 1}_{E^{\mathsf{top}}_{m+1}}   \\
& \times \mathbb{P}_{\ice, \operatorname{Geom}}^{0, \tilde{n}, \vec{x}, \vec{y}, \infty, g} \left( Q_m(s_1 ) - ps_1 \geq g(s) - ps + \delta^{\mathsf{sep}} (m) \cdot \tilde{n}^{1/2} \mbox{ for all } s \in F_{\tilde{n}}^m(s_1)  \right)   \Big],
\end{split}
\end{equation}
where $s_1 = s_0 - C_N$, and we recall that $\tilde{n} = D_N - C_N$ and $F_n^m$ is as above (\ref{S52E4}). We now observe that when $N \geq W_2$, and hence (\ref{S52E2}) holds, we have  
\begin{equation*}
s_1 \geq \frac{\epsilon}{d-c} \cdot (D_N - C_N) = \delta_2 \cdot \tilde{n} \mbox{ and } \tilde{n} - s_1 \geq \frac{\epsilon}{d-c} \cdot (D_N - C_N) = \delta_2 \cdot \tilde{n},
\end{equation*}
and so 
\begin{equation}\label{S52E11}
s_1/\tilde{n} \in [\delta_2, 1- \delta_2].
\end{equation}
From the definition of $E^{\mathsf{top}}_{m+1}$ and $E^{\mathsf{side}}$ we see that if $\vec{x}, \vec{y}, g$ are as in (\ref{S52E9}), then they a.s. satisfy the conditions above (\ref{S52E4}) with $n = \tilde{n}$, and so from (\ref{S52E4}) with $t_0 = s_1/\tilde{n}$ (which by (\ref{S52E11}) is in $[\delta_2, 1- \delta_2]$), we conclude
\begin{equation*}
\begin{split}
&{\bf 1}_{E^{\mathsf{side}}} \cdot {\bf 1}_{E^{\mathsf{top}}_{m+1}} \cdot \mathbb{P}_{\ice, \operatorname{Geom}}^{0, \tilde{n}, \vec{x}, \vec{y}, \infty, g} \left( Q_m(s_1 ) - ps_1 \geq g(s) - ps + \delta^{\mathsf{sep}} (m) \cdot \tilde{n}^{1/2} \mbox{ for all } s \in F_{\tilde{n}}^m(s_1)  \right) \\
&\geq{\bf 1}_{E^{\mathsf{side}}} \cdot {\bf 1}_{E^{\mathsf{top}}_{m+1}} \cdot (1 - \varepsilon/(4k)).
\end{split}
\end{equation*}
Taking expectations on both sides of the last equation and using (\ref{S52E10}) we conclude
$$\mathbb{P} \left(E^{\mathsf{side}} \cap E^{\mathsf{top}}_{m+1} \cap \tilde{E}^{\mathsf{sep}}\right) \geq (1 - \varepsilon/(4k)) \cdot \mathbb{P} (E^{\mathsf{side}} \cap E^{\mathsf{top}}_{m+1} ). $$
Combining the latter with the inequalities in (\ref{S52E3}) and (\ref{S52E8}) we get
$$\mathbb{P}\left(  \tilde{E}^{\mathsf{sep}}_m  \right)  \geq \mathbb{P} \left(E^{\mathsf{side}} \cap E^{\mathsf{top}}_{m+1} \cap \tilde{E}^{\mathsf{sep}}\right) \geq (1 - \varepsilon/(4k)) \cdot (1 - \varepsilon/(2k)) \geq 1- \varepsilon/k, $$
which proves (\ref{S52E7}) and hence the lemma.
\end{proof}

%---------------------------------------------------------------------------------------------------------------
% Section 5.3
%
%----------------------------------------------------------------------------------------------------------------
\subsection{Tightness of ensembles}\label{Section5.3} The goal of this section is to prove that $\mathcal{L}^N$ in Theorem \ref{S5T} is a tight sequence in $C(\llbracket 1,K \rrbracket \times \Lambda)$. In view of \cite[Lemma 2.4]{DEA21} it suffices to show that for each $[c,d] \subset \Lambda$ and $k \in \llbracket 1, K \rrbracket$
\begin{equation}\label{S53E1}
\begin{split}
&\lim_{a \rightarrow \infty} \limsup_{N \rightarrow \infty} \mathbb{P}\left(|\mathcal{L}_k^N(d_N^{-1}\lfloor c d_N \rfloor)| \geq a \right) = 0, \mbox{ and for each } \epsilon > 0,\\
& \lim_{\delta \rightarrow 0} \limsup_{N \rightarrow \infty} \mathbb{P} \left( \sup_{x,y \in [d_N^{-1} C_N, d_N^{-1} D_N],  |x - y| \leq \delta} \left| \mathcal{L}^N_k(x) - \mathcal{L}^N_k(y) \right| \geq \epsilon \right) = 0,
\end{split}
\end{equation}
where we have set
\begin{equation}\label{S53E2}
 C_N = \lfloor c \cdot d_N \rfloor, \hspace{2mm} D_N = \lfloor d \cdot d_N \rfloor, \mbox{ and } \tilde{n} = D_N - C_N.
\end{equation}
The first line in (\ref{S53E1}) follows by the one-point tightness assumption in the statement of the theorem. To show the second line in (\ref{S53E1}) it suffices to show that for any $\epsilon, \varepsilon > 0$, there exist $W_3 \in \mathbb{N}$ and $\delta > 0$, such that for $N \geq W_3$, we have
\begin{equation}\label{S53E3}
\begin{split}
& \mathbb{P} \left(w_N(L^N_k, \delta) > \epsilon \right) < \varepsilon, \mbox{ where } \\
&w_N(L^N_k, \delta) = \sup_{x,y \in [C_N, D_N],  |x - y| \leq \delta \tilde{n}} \left| \sigma^{-1} \tilde{n}^{-1/2} (L^N_k(x) - px) -  \sigma^{-1} \tilde{n}^{-1/2} (L^N_k(y) - py) \right| 
\end{split}
\end{equation}
We next proceed to specify the values of $W_3$ and $\delta$.\\

From Lemma \ref{S51L} we can find $W_{3,1}$ and $M^{\mathsf{top}}$ such that for $N \geq W_{3,1}$
\begin{equation}\label{S53E4}
\begin{split}
&\mathbb{P} \left( \max_{u \in [C_N, D_N] } \left( L_1^N(u) - p u + pC_N - \lfloor p C_N \rfloor  \right) \leq M^{\mathsf{top}} \cdot \tilde{n}^{1/2} \right) > 1 - \varepsilon/4.
\end{split}
\end{equation}
In addition, from (\ref{S5E1}) we can find $W_{3,2} \in \mathbb{N}$ and $M^{\mathsf{side}} >0$ such that for $N \geq W_{3,2}$
\begin{equation}\label{S53E5}
\begin{split}
&\mathbb{P} \left( E^{\mathsf{side}}\right) \geq 1- \varepsilon/4, \mbox{ where } E^{\mathsf{side}} = \left\{ \left|L_i^{N}(C_N) - \lfloor p C_N \rfloor \right| \leq M^{\mathsf{side}} \cdot \tilde{n}^{1/2} \mbox{ for } i \in \llbracket 1, k\rrbracket \right\} \\ 
& \cap \left\{ \left|L_i^{N}(D_N) - \lfloor p C_N \rfloor - p \tilde{n} \right| \leq M^{\mathsf{side}} \cdot \tilde{n}^{1/2} \mbox{ for } i \in \llbracket 1, k\rrbracket \right\}.
\end{split}
\end{equation}

Let $[\hat{c}, \hat{d}] \subset \Lambda$ be such that $[c,d] \subset (\hat{c}, \hat{d})$ and note that for all large $N$ we have $C_N, D_N \in [\hat{c} \cdot d_N, \hat{d} \cdot d_N]$. From Lemma \ref{S52L} applied to $a = \hat{c}$, $b = \hat{d}$, $s_0 \in \{C_N, D_N \}$ we can find $W_{3,3} \in \mathbb{N}$, $\delta^{\mathsf{sep}} > 0$, $\Delta^{\mathsf{sep}} \in (0,1/2)$ such that for $N \geq W_{3,3}$ we have  
\begin{equation}\label{S53E6}
\begin{split}
&\mathbb{P}\left(  E^{\mathsf{sep}}_m  \right) > 1- \varepsilon/4, \mbox{ where } E^{\mathsf{sep}} = \Big\{ L_m^N(s_0) - ps_0 \geq L_{m+1}^N(s) - ps + \delta^{\mathsf{sep}} \cdot \tilde{n}^{1/2} \\
& \mbox{ for all } m \in \llbracket 1, k \rrbracket \mbox{, } s \in \mathbb{Z} \cap [ s_0 - \Delta^{\mathsf{sep}} \cdot \tilde{n},s_0 + \Delta^{\mathsf{sep}} \cdot \tilde{n}] \mbox{ and }s_0 \in \{C_N, D_N \}  \Big\}.
\end{split}
\end{equation}

We let $N_6, \delta$ be as in Lemma \ref{S42L} for our present choice of $p, k, \epsilon, M^{\mathsf{side}}$, $M^{\mathsf{bot}} = M^{\mathsf{top}}$ as above, $A^{\mathsf{sep}} = A^{\mathsf{gap}} = \delta^{\mathsf{sep}}$, $\Delta^{\mathsf{gap}} = \Delta^{\mathsf{sep}}$ and $\eta = \varepsilon/4$. This specifies our choice of $\delta$. We also let $W_3$ be large enough so that $W_3 \geq \max(W_{3,1}, W_{3,2}, W_{3,3})$ and for $N \geq W_{3}$ we have $[C_N, D_N] \subseteq [\hat{A}_N, \hat{B}_N]$, $K_N \geq k+1$, and $\tilde{n} \geq N_6$. This specifies our choice of $W_3$. \\

We now proceed to prove (\ref{S53E3}). For $N \geq W_3$ we have $[C_N, D_N] \subseteq [\hat{A}_N, \hat{B}_N]$ and by interlacing we have $L^N_i(s) \geq L^{N}_{i+1}(s)$ for all $s \in [C_N, D_N]$ and $i \in \llbracket 1, k \rrbracket$. The latter and (\ref{S53E4}) implies that
\begin{equation}\label{S53E7}
\begin{split}
&\mathbb{P} \left( E^{\mathsf{top}}  \right) > 1 - \frac{\varepsilon}{4} \mbox{, where }E^{\mathsf{top}} = \left\{\max_{u \in [C_N, D_N] } \left(L_{k+1}^N(u) - p u + pC_N - \lfloor p C_N \rfloor  \right) \leq M^{\mathsf{top}}  \tilde{n}^{1/2} \hspace{-0.5mm} \right\} \hspace{-0.5mm}  .
\end{split}
\end{equation}
Let us set 
\begin{equation}\label{S53E8}
\begin{split}
&\vec{x} = \left(L^N_1(C_N) - \lfloor p C_N \rfloor, \dots, L^N_k(C_N) - \lfloor p C_N \rfloor \right),  \\
&\vec{y} = \left(L^N_1(D_N) - \lfloor p C_N \rfloor , \dots, L^N_k(D_N) - \lfloor p C_N \rfloor\right), \\
&g(s) = L^{N}_{k+1}(s + C_N) - \lfloor p C_N \rfloor \mbox{ for } s \in \llbracket 0, \tilde{n} \rrbracket.
\end{split}
\end{equation}
By the interlacing Gibbs property (see Lemma \ref{Lem.StrongGP}) and translation we have for $N \geq W_3$ that
\begin{equation}\label{S53E9}
\begin{split}
&\mathbb{P} \left(E^{\mathsf{side}} \cap E^{\mathsf{top}} \cap \{w_N(L^N_k, \delta) > \epsilon \} \right) = \mathbb{E} \left[ {\bf 1}_{E^{\mathsf{side}}} \cdot {\bf 1}_{E^{\mathsf{top}}}  \cdot \mathbb{P}_{\ice, \operatorname{Geom}}^{0, \tilde{n}, \vec{x}, \vec{y}, \infty, g} \left( w(\mathcal{Q}_k, \delta) > \epsilon  \right)   \right],
\end{split}
\end{equation}
where $w(f,\delta)$ is the modulus of continuity on $[0,1]$ from (\ref{S33E1}). We observe that on the event $E^{\mathsf{side}} \cap E^{\mathsf{top}}$ the $\vec{x}, \vec{y}, g$ from (\ref{S53E8}) almost surely satisfy the conditions of Lemma \ref{S42L} with parameter specialization as detailed below (\ref{S53E6}). We conclude from the lemma that
\begin{equation*}
\begin{split}
& {\bf 1}_{E^{\mathsf{side}}} \cdot {\bf 1}_{E^{\mathsf{top}}}  \cdot \mathbb{P}_{\ice, \operatorname{Geom}}^{0, \tilde{n}, \vec{x}, \vec{y}, \infty, g} \left( w(\mathcal{Q}_k, \delta) > \epsilon  \right) \leq {\bf 1}_{E^{\mathsf{side}}} \cdot {\bf 1}_{E^{\mathsf{top}}} \cdot \varepsilon/4.
\end{split}
\end{equation*}
Taking expectations on both sides of the last equation and using (\ref{S53E9}) gives
$$\mathbb{P} \left(E^{\mathsf{side}} \cap E^{\mathsf{top}} \cap \{w_N(L^N_k, \delta) > \epsilon \} \right) \leq \mathbb{P} \left(E^{\mathsf{side}} \cap E^{\mathsf{top}} \right) \cdot (\varepsilon/4).$$
Using the last equation, (\ref{S53E5}) and (\ref{S53E7}) gives
$$\mathbb{P}(w_N(L^N_k, \delta) > \epsilon) \leq \mathbb{P} \left(E^{\mathsf{side}} \cap E^{\mathsf{top}} \cap \{w_N(L^N_k, \delta) > \epsilon \} \right) + \varepsilon/2 \leq 3\varepsilon/4, $$
which proves (\ref{S53E3}) and hence the tightness part in Theorem \ref{S5T}.

%---------------------------------------------------------------------------------------------------------------
% Section 5.4
%
%----------------------------------------------------------------------------------------------------------------
\subsection{Brownian Gibbs property}\label{Section5.4} In this section we complete the proof of Theorem \ref{S5T} by showing that any subsequential limit of $\mathcal{L}^N$ satisfies the partial Brownian Gibbs property from \cite[Definition 2.7]{DM21}. We assume that $\mathcal{L}^{\infty}$ is any subsequential limit and by possibly passing to a subsequence, which we continue to call $\mathcal{L}^N$, we assume that $\mathcal{L}^N \Rightarrow \mathcal{L}^{\infty}$. By Skorohod's Representation Theorem, see \cite[Theorem 6.7]{Bill}, we may assume that the sequence $\{\mathcal{L}^N\}_{N \geq 1}$ and $\mathcal{L}^{\infty}$ are all defined on the same probability space $(\Omega, \mathcal{F}, \mathbb{P})$ and $\lim_N \mathcal{L}^N(\omega) = \mathcal{L}^{\infty}(\omega)$ for each $\omega \in \Omega$. We mention that in applying \cite[Theorem 6.7]{Bill} we implicitly used that $C(\llbracket 1, K \rrbracket \times \Lambda)$ is a Polish space, cf. Remark \ref{RemPolish}. In view of Lemma \ref{S52L} we have that for each fixed $t \in \Lambda$ we have a.s.
\begin{equation}\label{S54E1}
\mathcal{L}^{\infty}_i(t) > \mathcal{L}^{\infty}_{i+1}(t) \mbox{ for all } i \in \llbracket 1, K -1 \rrbracket.
\end{equation}

Let us recall the conditions we need to verify in \cite[Definition 2.7]{DM21}. When $K = 1$ there is nothing to check, and so we assume that $K \geq 2$. We seek to show that $\mathcal{L}^{\infty}$ is $\mathbb{P}$-a.s. non-intersecting, i.e.
\begin{equation}\label{S54E2}
\mathbb{P}\left(\mathcal{L}^{\infty}_i(t) >\mathcal{L}^{\infty}_{i+1}(t) \mbox{ for all }  t\in \Lambda , i \in \llbracket 1, K -1 \rrbracket \right) = 1,
\end{equation}
and that for any $[a,b] \subset \Lambda$, finite $S = \llbracket s_1, s_2 \rrbracket \subseteq \llbracket 1, K-1 \rrbracket$, bounded Borel-measurable $F: C(S \times [a,b]) \rightarrow \mathbb{R}$ and $E \in \mathcal{F}_{\operatorname{ext}}(S \times (a,b))$ we have
\begin{equation}\label{S54E3}
\mathbb{E}\left[ F\left( \mathcal{L}^{\infty} \vert_{S \times [a,b]} \right) \cdot {\bf 1}_E \right] = \mathbb{E}\left[ \mathbb{E}_{\operatorname{avoid}}^{a,b,\vec{x},\vec{y},f,g} \left[ F(\mathcal{Q}) \right]\cdot {\bf 1}_E \right].  
\end{equation}
In (\ref{S54E3}) we have that $\mathcal{L}^{\infty} \vert_{S \times [a,b]}$ is the restriction of $\mathcal{L}^{\infty}$ to $S \times [a,b]$, $\vec{x} = (\mathcal{L}^{\infty}_{s_1} (a), \dots, \mathcal{L}^{\infty}_{s_2} (a))$, $\vec{y} = (\mathcal{L}^{\infty}_{s_1} (b), \dots, \mathcal{L}^{\infty}_{s_2} (b))$, $g = \mathcal{L}_{s_2+1}[a,b]$, $f = \mathcal{L}_{s_1-1}^{\infty}[a,b]$ for $s_1 \geq 2$, and $f = \infty$ if $s_1 = 1$. Also,
$$\mathcal{F}_{\operatorname{ext}}(S \times (a,b)) = \sigma \left\{ \mathcal{L}^{\infty}_i(s): (i,s) \in \llbracket 1, K \rrbracket \times \Lambda \setminus S \times (a,b) \right\},$$
and $\mathcal{Q}$ has distribution $\mathbb{P}_{\operatorname{avoid}}^{a,b,\vec{x},\vec{y},f,g}$ as in Definition \ref{AvoidBB} with curves indexed by $S$ rather than $\llbracket 1, \dots, s_2 - s_1 + 1\rrbracket$. Note that $\mathbb{E}_{\operatorname{avoid}}^{a,b,\vec{x},\vec{y},f,g} \left[ F(\mathcal{Q}) \right]$ is well-defined, since by (\ref{S54E1}) we have $\mathbb{P}$-a.s. 
$$f(a) > \mathcal{L}^{\infty}_{s_1} (a) > \cdots > \mathcal{L}^{\infty}_{s_2} (a) > g(a) \mbox{ and } f(b) > \mathcal{L}^{\infty}_{s_1} (b) > \cdots > \mathcal{L}^{\infty}_{s_2} (b) > g(b).$$
We also mention that $\mathbb{E}_{\operatorname{avoid}}^{a,b,\vec{x},\vec{y},f,g} \left[ F(\mathcal{Q}) \right]$ defines a bounded measurable random variable on $(\Omega, \mathcal{F}, \mathbb{P})$ in view of \cite[Lemma 3.4]{DM21}. The latter implies that both sides of (\ref{S54E3}) are well-defined and finite, see also \cite[Remark 2.6]{DM21}.

Our arguments for establishing (\ref{S54E2}) and (\ref{S54E3}) are similar to those used in the proof of \cite[Theorem 2.26(ii)]{DEA21}, so we will be brief. For clarity, we split the remainder of the proof into two steps.\\

{\bf \raggedleft Step 1.} In this step we deduce (\ref{S54E2}) from (\ref{S54E3}). Let us fix $k \in \llbracket 1, K-1\rrbracket$, $[a,b] \subset \Lambda$, a bounded measurable $F_1: C(\llbracket 1, k -1 \rrbracket \times [a,b]) \rightarrow \mathbb{R}$, and let $\mathcal{H}_2$ be the collection of bounded measurable $F_2: C([a,b]) \rightarrow \mathbb{R}$ such that
\begin{equation}\label{S54E4}
\mathbb{E}\left[ F_1\left( \mathcal{L}^{\infty} \vert_{\llbracket 1, k-1 \rrbracket \times [a,b]} \right) \cdot F_2( \mathcal{L}_k[a,b]) \right] = \mathbb{E}\left[ \mathbb{E}_{\operatorname{avoid}}^{a,b,\vec{x},\vec{y},\infty,g} \left[ F_1(\mathcal{Q}) \right]\cdot F_2( \mathcal{L}_k[a,b]) \right].  
\end{equation}
In view of (\ref{S54E3}), we have that (\ref{S54E4}) holds when $F_2(h) = {\bf 1}\{h \in B\}$ for any Borel $B \subseteq C([a,b])$. By linearity and the bounded convergence theorem, we see that $\mathcal{H}_2$ is closed under linear combinations and bounded monotone limits, which by the monotone class theorem, see \cite[Theorem 5.2.2.]{Durrett} implies $\mathcal{H}_2$ contains all bounded measurable functions.

We next let $\mathcal{H}$ denote the class of bounded measurable $F: C(\llbracket 1, k \rrbracket \times [a,b]) \rightarrow \mathbb{R}$ such that
\begin{equation}\label{S54E5}
\mathbb{E}\left[ F\left( \mathcal{L}^{\infty} \vert_{\llbracket 1, k \rrbracket \times [a,b]} \right) \right] = \mathbb{E}\left[ \mathbb{E}_{\operatorname{avoid}}^{a,b,\vec{x},\vec{y},\infty,g} \left[ F(\tilde{\mathcal{Q}}) \right]\right],  
\end{equation}
where $\tilde{\mathcal{Q}}_i = \mathcal{Q}_i$ for $i \in \llbracket 1, k-1 \rrbracket$, $\tilde{\mathcal{Q}}_k = \mathcal{L}^{\infty}_k[a,b]$, and  $\{\mathcal{Q}_i\}_{i \in \llbracket 1, k-1\rrbracket}$ has law $\mathbb{P}_{\operatorname{avoid}}^{a,b,\vec{x},\vec{y},\infty,g} $. From (\ref{S54E4}), which we now know holds for any bounded measurable $F_1, F_2$, we see that $\mathcal{H}$ contains all functions
$$F(f_1, \dots, f_k) = \prod_{i = 1}^{k} {\bf 1}\{f_i \in B_i\}\mbox{ where } B_i \subseteq C([a,b]) \mbox{ are Borel for }i \in \llbracket 1, k \rrbracket.$$
As before, we conclude by the monotone class theorem that $\mathcal{H}$ contains all bounded measurable functions. Setting 
$$F(f_1, \dots, f_k) = {\bf 1}\{f_1(s) > f_2(s) > \cdots > f_k(s) \mbox{ for all } s \in [a,b]\}$$
in (\ref{S54E5}) and using that by Definition \ref{AvoidBB} we have $\mathbb{E}_{\operatorname{avoid}}^{a,b,\vec{x},\vec{y},\infty,g} \left[ F(\tilde{\mathcal{Q}}) \right] = 1$, we conclude that
\begin{equation}\label{S54E6}
\mathbb{P}\left(\mathcal{L}^{\infty}_i(t) >\mathcal{L}^{\infty}_{i+1}(t) \mbox{ for all }  t\in [a,b] , i \in \llbracket 1, k-1 \rrbracket \right) = 1.
\end{equation}
Taking a countable sequence of intervals $[a,b]$ that exhausts $\Lambda$ and an increasing sequence of $k$'s converging to $K$, and taking intersections of the events in (\ref{S54E6}) gives (\ref{S54E2}).\\

{\bf \raggedleft Step 2.} In this step we prove (\ref{S54E3}). Fix $m \in \mathbb{N}$, $n_1, \dots, n_m \in \llbracket 1 , K \rrbracket$, $t_1, \dots, t_m \in \Lambda$ and bounded continuous $h_1, \dots, h_m : \mathbb{R} \rightarrow \mathbb{R}$. Define $R = \{i \in \llbracket 1, m \rrbracket: n_i \in S, t_i \in [a,b]\}$. We first show that 
\begin{equation}\label{S54E7}
\mathbb{E}\left[ \prod_{i = 1}^m h_i(\mathcal{L}^{\infty}_{n_i}(t_i)) \right] = \mathbb{E}\left[ \prod_{i \not \in R} h_i(\mathcal{L}^{\infty}_{n_i}(t_i))  \cdot \mathbb{E}_{\operatorname{avoid}}^{a,b,\vec{x},\vec{y},f,g} \left[ \prod_{i  \in R} h_i(\mathcal{Q}_{n_i}(t_i))   \right] \right], 
\end{equation}
where we recall from the beginning of the section that in $\mathbb{P}_{\operatorname{avoid}}^{a,b,\vec{x},\vec{y},f,g}$ the curves are indexed by $S = \llbracket s_1, s_2\rrbracket$ as opposed to $\llbracket 1, s_2 - s_1 + 1\rrbracket$. From the a.s. convergence of $\mathcal{L}^N$ to $\mathcal{L}^{\infty}$ we get
\begin{equation}\label{S54E8}
\lim_{N \rightarrow \infty} h_i(\mathcal{L}^{N}_{n_i}(t_i)) = h_i(\mathcal{L}^{\infty}_{n_i}(t_i))
\end{equation}
$\mathbb{P}$-a.s. for each $i \in \llbracket 1, m \rrbracket$, and so by the bounded convergence theorem
\begin{equation}\label{S54E9}
\lim_{N \rightarrow \infty} \mathbb{E}\left[ \prod_{i = 1}^m h_i(\mathcal{L}^{N}_{n_i}(t_i)) \right] = \mathbb{E}\left[ \prod_{i = 1}^m h_i(\mathcal{L}^{\infty}_{n_i}(t_i)) \right].
\end{equation}
In addition, if we set $A_N = \lfloor a \cdot d_N\rfloor$, $B_N = \lceil b \cdot d_N \rfloor$, $\vec{X}^N = (L_{s_1}^N(A_N), \dots, L_{s_2}^N(A_N))$, $\vec{Y}^N = (L_{s_1}^N(B_N), \dots, L_{s_2}^N(B_N))$, $G_N(t) = L_{s_2+1}^N(t)$ for $t \in [A_N,B_N]$ and 
$$F_N(t) = L^N_{s_1 -1}(t) \mbox{ for $t \in [A_N, B_N]$ if $s_1 \geq 2$, or }F_N(t) = \infty \mbox{ if } s_1 = 1, $$
then from the a.s. convergence of $\mathcal{L}^N$ to $\mathcal{L}^{\infty}$ and (\ref{S54E1}), we know that $\mathbb{P}$-a.s. the sequences $A_N, B_N, d_N, \vec{X}^N, \vec{Y}^N, F_N, G_N$ satisfy the conditions of Lemma \ref{lem:RW} and so we conclude that $\mathbb{P}$-a.s.
\begin{equation}\label{S54E10}
\lim_{N \rightarrow \infty} \mathbb{E}^{A_N, B_N,\vec{X}^N,\vec{Y}^N,F_N,G_N}_{\ice, \operatorname{Geom}} \left[ \prod_{i  \in R} h_i(\mathcal{Q}^N_{n_i}(t_i))   \right] = \mathbb{E}_{\operatorname{avoid}}^{a,b,\vec{x},\vec{y},f,g} \left[ \prod_{i  \in R} h_i(\mathcal{Q}_{n_i}(t_i))   \right],
\end{equation}
where on the left side of (\ref{S54E10}) we have for $i \in \llbracket s_1, s_2 \rrbracket$
$$ \mathcal{Q}^N_{i}(t) = \sigma^{-1} d_N^{-1/2} \cdot (Q^N_{i -s_1 + 1}(td_N) - td_N p),$$
with $\mathfrak{Q}^N = \{Q_{i}^N\}_{i = 1}^{s_2 - s_1 +1 }$ having law $\mathbb{P}^{A_N, B_N,\vec{X}^N,\vec{Y}^N,F_N,G_N}_{\ice, \operatorname{Geom}}$. Combining (\ref{S54E8}) with (\ref{S54E10}) and the bounded convergence theorem gives
\begin{equation}\label{S54E11}
\begin{split}
&\lim_{N \rightarrow \infty} \mathbb{E}\left[ \prod_{i \not \in R} h_i(\mathcal{L}^{N}_{n_i}(t_i))  \cdot \mathbb{E}^{A_N, B_N,\vec{X}^N,\vec{Y}^N,F_N,G_N}_{\ice, \operatorname{Geom}} \left[ \prod_{i  \in R} h_i(\mathcal{Q}^N_{n_i}(t_i))   \right] \right]  \\
&= \mathbb{E}\left[ \prod_{i \not \in R} h_i(\mathcal{L}^{\infty}_{n_i}(t_i))  \cdot \mathbb{E}_{\operatorname{avoid}}^{a,b,\vec{x},\vec{y},f,g} \left[ \prod_{i  \in R} h_i(\mathcal{Q}_{n_i}(t_i))   \right] \right].
\end{split}
\end{equation}
Finally, if $N$ is large enough so that $[A_N, B_N] \subseteq [\hat{A}_N, \hat{B}_N]$ we have by the interlacing Gibbs property, see Lemma \ref{Lem.StrongGP}, that the terms on the first line in (\ref{S54E11}) agree with those on the left in (\ref{S54E9}), and so the limits agree, which is precisely (\ref{S54E7}).\\

Starting from (\ref{S54E7}) deducing (\ref{S54E3}) is a standard monotone class argument. If we set
$$h_i(x) =  \begin{cases} 1 &\mbox{ if } x < r_i \\ 1 - n(x - r_i) &\mbox{ if } x \in [r_i, r_i+1/n] \\ 0 &\mbox{ if } x > r_i + 1/n \end{cases},$$
in equation (\ref{S54E7}) and let $n \rightarrow \infty$ (using the bounded convergence theorem), we conclude that (\ref{S54E3}) holds when $E, F$ are of the form 
\begin{equation}\label{S54E12}
E = \cap_{i = 1}^u \{ \mathcal{L}^{\infty}_{n_i}(t_i) \leq r_i\} \mbox{ and } F(f_{s_1}, \dots, f_{s_2}) = \prod_{j = 1}^v {\bf 1}\{ f_{m_j}(s_j) \leq r_j'\},
\end{equation}
where $(n_i, t_i) \not \in S \times (a,b)$ and $(m_j,s_j) \in S \times [a,b]$. By the monotone class theorem for each fixed $E$ as in (\ref{S54E12}) we can extend the equality in (\ref{S54E3}) to any bounded measurable $F: C(S \times [a,b]) \rightarrow \mathbb{R}$. Finally, if we fix any bounded measurable $F: C(S \times [a,b]) \rightarrow \mathbb{R}$ in (\ref{S54E3}), then having the equality for $E$ as in (\ref{S54E12}) allows us to conclude by the $\pi-\lambda$ theorem that the equality holds for any $E \in \mathcal{F}_{\operatorname{ext}}(S \times (a,b))$.

%---------------------------------------------------------------------------------------------------------------
% Section 6
%
%----------------------------------------------------------------------------------------------------------------
\section{Applications}\label{Section6} In this section we apply Theorem \ref{S5T} to sequences of line ensembles that arise in Schur processes with spiked parameters, and show that they converge uniformly over compact sets to the Airy wanderer line ensembles constructed recently in \cite{ED24a}. In Section \ref{Section6.1} we introduce Schur processes and explain how we scale their parameters. In Section \ref{Section6.2} we show that their corresponding line ensembles converge weakly.

%---------------------------------------------------------------------------------------------------------------
% Section 6.1
%
%----------------------------------------------------------------------------------------------------------------
\subsection{Schur processes}\label{Section6.1} Our exposition in this section follows closely \cite[Section 3]{ED24a}, which in turn is based on \cite[Chapter I]{Mac} and \cite{BR05}.

A {\em partition} is a sequence $\lambda = (\lambda_1, \lambda_2, \dots)$ of non-negative integers such that $\lambda_1 \geq \lambda_2 \geq \cdots$ and all but finitely many terms are zero. We define the {\em weight} of a partition by $|\lambda| = \sum_{i \geq 1} \lambda_i$. We say that two partitions $\lambda, \mu$ {\em interlace}, denoted by $\lambda \succeq \mu$ or $\mu \preceq \lambda$, if $\lambda_1 \geq \mu_1 \geq \lambda_2 \geq \mu_2 \geq \cdots$. Given two partitions $\lambda, \mu$, we define the {\em skew Schur polynomial} in finitely many variables $x_1, \dots, x_n$ by
\begin{equation}\label{S61E1}
s_{\lambda/ \mu}(x_1, \dots, x_n) =  \sum_{\mu = \lambda^{0} \preceq  \lambda^{1} \preceq \cdots \preceq \lambda^{n} = \lambda} \prod_{i = 1}^n x_i^{|\lambda^{i}| - |\lambda^{i-1}|},
\end{equation}
with the convention that $s_{\lambda/ \mu}(x_1, \dots, x_n) = 0$ if the above sum is empty. If $\mu = (0,0,\dots)$ we drop it from the notation and write $s_{\lambda}(x_1, \dots,x_n)$ -- these are called {\em Schur polynomials}.

If $M, N \in \mathbb{N}$ and $\vec{X} = (x_1, \dots, x_M)$, $\vec{Y} = (y_1, \dots, y_N)$ with $x_i, y_j \geq 0$ and $x_i y_j \in [0,1)$, we define the {\em ascending Schur process} to be the probability measure $\mathbb{P}_{\vec{X}, \vec{Y}}$ on sequences of partitions $(\lambda^1, \dots, \lambda^M)$ such that 
\begin{equation}\label{S61E2}
\mathbb{P}_{\vec{X}, \vec{Y}}(\lambda^1 = \mu^1, \dots, \lambda^M = \mu^M) = \prod_{i = 1}^M \prod_{j = 1}^N (1 - x_i y_j) \cdot \prod_{i = 1}^M s_{\mu^i/ \mu^{i-1}}(x_i) \cdot s_{\mu^M}(y_1, \dots, y_M),
\end{equation}
where $\mu^0 = (0,0, \dots)$. See \cite[Section 3]{ED24a} for an explanation of why $\mathbb{P}_{\vec{X}, \vec{Y}}$ is a probability measure.

From \cite[(3.6)]{ED24a}, we have for $A, B \in \llbracket 1, M\rrbracket$ with $A \leq B$ 
\begin{equation}\label{S61E3}
\begin{split}
&\mathbb{P}_{\vec{X}, \vec{Y}}\left(\cap_{i \in \llbracket A,B \rrbracket} \{\lambda^{i} = \mu^i \} \right) = \prod_{i = A+1}^B s_{\mu^i/ \mu^{i-1}}(x_i) \\
&\times \prod_{i = 1}^B \prod_{j = 1}^N (1 - x_i y_j) \cdot s_{\mu^A}(x_1, \dots, x_A) \cdot  s_{\mu^B}(y_1, \dots, y_M). 
\end{split}
\end{equation}
In particular, if $x_i = q$ for $i \in \llbracket A + 1,B \rrbracket$, we have from (\ref{S61E1}) and (\ref{S61E3}) 
\begin{equation*}
\begin{split}
&\mathbb{P}_{\vec{X}, \vec{Y}}\left(\cap_{i \in \llbracket A,B \rrbracket} \{\lambda^{i} = \mu^i \} \right) = {\bf 1}\{ \mu^A \preceq \mu^{A+1} \preceq \cdots \preceq \mu^{B}\} \cdot q^{|\mu_B| - |\mu_A|} \\
&\times \prod_{i = 1}^B \prod_{j = 1}^N (1 - x_i y_j) \cdot s_{\mu^A}(x_1, \dots, x_A) \cdot  s_{\mu^B}(y_1, \dots, y_M).
\end{split}
\end{equation*}
The last equation implies that when $\mathbb{P}_{\vec{X}, \vec{Y}}(\lambda^A = \mu^A, \lambda^B = \mu^B) > 0$, we have
\begin{equation}\label{S6Gibbs}
\begin{split}
\mathbb{P}_{\vec{X}, \vec{Y}}\left(\cap_{i \in \llbracket A,B \rrbracket} \{\lambda^{i} = \mu^i \} \vert \lambda^A = \mu^A, \lambda^{B} = \mu^B \right) = {\bf 1}\{ \mu^A \preceq \mu^{A+1} \preceq \cdots \preceq \mu^{B}\}.
\end{split}
\end{equation}
From (\ref{S6Gibbs}), we deduce that for each $m \geq 1$, the law of the $\llbracket 1, m \rrbracket$-indexed geometric line ensemble $\{\lambda_i^j: j \in \llbracket A, B\rrbracket, i \in \llbracket 1, m \rrbracket \}$ is a convex combination of measures of the form $\mathbb{P}_{\ice, \operatorname{Geom}}^{A,B, \vec{x}, \vec{y}, \infty, g}$ as in Definition \ref{DefSGP}, with different $\vec{x}, \vec{y}, g$. By Lemma \ref{Lem.FinEnsSatGB}, it follows that $\{\lambda_i^j: j \in \llbracket A, B\rrbracket, i \in \mathbb{N} \}$ is an $\mathbb{N}$-indexed geometric line ensemble that satisfies the interlacing Gibbs property of Definition \ref{DefSGP}.\\

We end this section by explaining how we scale the parameters $(x_1, \dots, x_N)$ and $(y_1, \dots, y_M)$ in the ascending Schur process. We mention that the scaling below is a special case of the scaling in \cite[Definition 4.1]{ED24a} corresponding to $J_a^- = J_b^- = c^- = 0$.
\begin{definition}\label{ParScale} We assume that $c^+$, $\{a_i^+\}_{ i \geq 1}$, $\{b_i^+\}_{ i \geq 1}$ are non-negative reals such that
\begin{equation}\label{ParProp}
\sum_{i = 1}^{\infty} (a_i^+ + b_i^+ ) < \infty \mbox{ and } a_{i}^{+} \geq a_{i+1}^{+},  b_{i}^{+} \geq b_{i+1}^{+} \mbox{ for all } i \geq 1.
\end{equation}
We let $J_a^{+} = \inf \{ k \geq 1: a_{k}^{+} = 0\} - 1$ and $J_b^{-} = \inf \{ k \geq 1: b_{k}^{-} = 0\} - 1$. In words, $J_a^{+}$ is the largest index $k$ such that $a_{k}^{+} > 0$, with the convention that $J_a^{+} = 0$ if all $a_k^{+} = 0$ and $J_a^{+} = \infty$ if all $a_k^{+} > 0$, and analogously for $J_b^{+}$. We fix $q \in (0,1)$ and set
\begin{equation}\label{SigmaQ}
\sigma_q = \frac{q^{1/3} (1 + q)^{1/3}}{1- q} \mbox{ and } f_q = \frac{q^{1/3}}{2 (1 + q)^{2/3}}.
\end{equation}
For $N \in \mathbb{N}$ increasing to infinity, we consider three numbers $A_N, B_N, C_N^{+}$ and sequences $\{x^N_i \}_{ i\geq 1}$ and $\{y^N_i\}_{i \geq 1}$ such that
\begin{equation}\label{ABSeq}
\begin{split}
&x_i^N =1 - \frac{1}{N^{1/3} b_i^+ \sigma_q } \mbox{ for $i = 1, \dots, B_N$, and } y_i^N =1 -\frac{1}{N^{1/3} a_i^+ \sigma_q } \mbox{ for $i = 1, \dots, A_N$},
\end{split}
\end{equation}
where $B_N \leq \min\left(\lfloor N^{1/12} \rfloor, J_b^+ \right)$ is the largest integer such that $x^N_{B_N} \geq q$, and $A_N \leq \min\left(\lfloor N^{1/12} \rfloor, J_a^+ \right)$ is the largest integer such that $y^N_{A_N} \geq q$. Here, we use the convention $x_0^N = y_0^N = 1$ so that $A_N =0$ and $B_N = 0$ are possible. We also have
\begin{equation}\label{CSeq}
\begin{split}
x_{B_N + i }^N = y_{A_N + i}^N =1 - \frac{2}{N^{1/4}c^+ \sigma_q} \mbox{ for } i = 1 ,\dots, C_N^+ \mbox{, where } C^+_N =  \begin{cases} 0 &\mbox{ if $c^+ = 0$, } \\  \lfloor N^{1/12} \rfloor  &\mbox{ if } c^+ > 0 \end{cases}, 
\end{split}
\end{equation}
 and 
\begin{equation}\label{RemSeq}
\begin{split}
\mbox{ $x_i^N = q$ for $i > B_N + C_N^+ $ and $y_i^N = q$ for $i > A_N + C_N^+ $.}
\end{split}
\end{equation}
We let $N_0 \in \mathbb{N}$ be sufficiently large (depending on $q, c^+$) so that $x_i^N, y_i^N \geq q$ for all $i \in \mathbb{N}$, provided that $N \geq N_0$. Note that if $N \geq N_0$ and $M \geq 1$ we can define the ascending Schur process in (\ref{S61E2}) with parameters $N, M, \{x_i^N\}_{i = 1}^M$ and $\{y_i^N\}_{ i = 1}^N$ as above, since $x^N_i, y_i^N \in [q, 1)$ for all $i \in \mathbb{N}$.
\end{definition}

%---------------------------------------------------------------------------------------------------------------
% Section 6.2
%
%----------------------------------------------------------------------------------------------------------------
\subsection{Convergence to the Airy wanderer line ensembles}\label{Section6.2} In this section we state and prove the main result about the asymptotics of the measures $\mathbb{P}_{\vec{X}, \vec{Y}}$ from (\ref{S61E2}) when the parameters are scaled as in Definition \ref{ParScale} -- this is Proposition \ref{S62P}. In the sequel we assume the same notation as in Definition \ref{ParScale}.\\

We set $p = \frac{q}{1-q}$, $\sigma = \sqrt{p (1+ p)}$. We also assume that $N_1 \in \mathbb{N}$ is large enough so that for $N \geq N_1$ we have $N \geq N_0$ and $N \geq A_N + C_N^+ + N^{3/4} + 1$. Finally, we assume $M_N \geq N + N^{3/4} + 1$ and let $\mathbb{P}_N$ be the measure $\mathbb{P}_{\vec{X}, \vec{Y}}$ from (\ref{S61E2}) with $M = M_N$, $x_i = x_i^N$ for $i \in \llbracket 1, M_N \rrbracket$, $y_i = y_i^N$ for $i \in \llbracket 1, N \rrbracket$. For the sake of specificity, if $N < N_0$ we let $\mathbb{P}_N$ be the measure in (\ref{S61E2}) with $M = M_N = N = 1$ and $x_1 = y_1 = 0$. If $\{\lambda^j_i: j \in \llbracket 1, M_N \rrbracket, i \in \mathbb{N}\}$ has distribution $\mathbb{P}_N$, we define the sequence of $\mathbb{N}$-indexed geometric line ensembles $\mathfrak{L}^N = \{L_i^N\}_{i \geq 1}$ on $\mathbb{Z}$ via 
\begin{equation}\label{S62E1}
L_i^N(s) = \begin{cases} \lambda_i^{N + s}  &\mbox{ if } s \in \llbracket - N +1, M_N - N \rrbracket  \\ \lambda_i^{1}  &\mbox{ if } s \leq - N \\ \lambda_i^{M_N} &\mbox{ if } s > M_N - N.
\end{cases}
\end{equation}
In words, $\mathfrak{L}^N$ is just a translation of $\{\lambda^j_i: j \in \llbracket 1, M_N \rrbracket, i \in \mathbb{N}\}$ with a constant extension outside of the integer interval $\llbracket 1, M_N \rrbracket$. We finally define $\mathcal{L}^N = \{\mathcal{L}_i^N\}_{i \geq 1}$ via
\begin{equation}\label{S62E2}
\mathcal{L}_i^N(s) = \sigma^{-1} N^{-1/3} \cdot \left( L_i^N(s N^{2/3}) - p s N^{2/3}  - 2p N\right).
\end{equation}

\begin{proposition}\label{S62P} As $N \rightarrow \infty$, the sequence $\mathcal{L}^N$ converges weakly to some $\mathbb{N}$-indexed line ensemble $\mathcal{L}^{\infty}$ on $\mathbb{R}$, which satisfies the partial Brownian Gibbs property from \cite[Definition 2.7]{DM21}. Moreover, if $\mathcal{L}^{a,b,c}$ is as in \cite[Theorem 1.10]{ED24a} for $\{a_i^+\}_{i \geq 1}, \{b_i^+\}_{i \geq 1}$ and $c^+$ as in the present statement, and $J_a^- = J_b^- = c^- = 0$, then
\begin{equation}\label{S62PE1}
\mathcal{L}^{a,b,c} = \{f_q^{1/2} \cdot \mathcal{L}_i^{\infty}(f_q^{-1}t): i \in \mathbb{N}, t \in \mathbb{R}\},
\end{equation}
where the equality is in distribution as random elements in $C(\mathbb{N} \times \mathbb{R})$.
\end{proposition}
\begin{proof} We fix $t_1,\dots, t_m \in \mathbb{R}$ with $t_1 < \cdots < t_m$ and set $M_j(N) = N + \lfloor t_j N^{2/3} \rfloor$ for $j\in \llbracket 1, m \rrbracket$. If
$$X_i^{j,N} = \sigma_q^{-1} N^{-1/3} \cdot \left( \lambda_i^{M_j(N)} - \frac{2q \tilde{N} }{1-q} - \frac{q t_j N^{2/3}}{1-q} - i\right) \mbox{ for $i \in \mathbb{N}$ and $j \in \llbracket 1, m \rrbracket$},$$
where $\tilde{N} = N - A_N - C_N^+ = N + O(N^{1/12})$, then we have from the proof of \cite[Theorem 1.8]{ED24a} that $(X_i^{j,N}: i \in \mathbb{N}, j \in \llbracket 1, m \rrbracket)$ converges weakly (as random elements in $(\mathbb{R}^{\infty}, \mathcal{R}^{\infty})$) to $(Y_i(f_q t_j) - f_q^2t_j^2: i \in \mathbb{N}, j \in \llbracket 1, m \rrbracket)$, where $\{Y_i\}_{i \geq 1}$ is as in the statement of that theorem. Notice that for all large $N$, we have $M_j(N) \in \llbracket 1, M_N \rrbracket$ and so we have 
\begin{equation}\label{S62E3}
L_i^N(\lfloor t_j N^{2/3} \rfloor) - p t_jN^{2/3} - 2pN = \sigma_q N^{1/3} \cdot X_i^{j,N} + O(N^{1/12}),
\end{equation}
where the constant in the big $O$ notation depends on $q$ and $i$. In particular, we conclude that for each $t \in \mathbb{R}$ and $i \in \mathbb{N}$ the random variables $\sigma^{-1} N^{-1/3}  \left( L_i^N(\lfloor t N^{2/3} \rfloor) - ptN^{2/3} - 2p N\right)$ are tight, which satisfies Assumption 1 in Section \ref{Section1.2} with $d_N = N^{2/3}$, $p = \frac{q}{1-q}$, and $C_N = 2pN$.

In addition, we have for $\hat{A}_N = -\lfloor N^{3/4} \rfloor$ and $\hat{B}_N = \lfloor N^{3/4} \rfloor$ that $N + \hat{A}_N \geq A_N + C_N^+ + 1$, if $N \geq N_1$, and so $x_r^N = q$ for all $r \in \llbracket N + \hat{A}_N, N + \hat{B}_N \rrbracket$. By using the latter with equation (\ref{S6Gibbs}), we see that Assumption 2 in Section \ref{Section1.2} is satisfied when $N \geq N_1$. For concreteness, we set $\hat{A}_N = \hat{B}_N = 0$ for $N < N_1$ and then the assumption is satisfied for all $N$. 

The observations in the last two paragraphs show that $\mathfrak{L}^N$ satisfies the conditions of Theorem \ref{S1T1}. We conclude that $\mathcal{L}^N$ is tight and any subsequential limit satisfies the partial Brownian Gibbs property. If $\mathcal{L}^{\infty}$ is any subsequential limit, then from (\ref{S62E3}), the weak convergence of $(X_i^{j,N}: i \in \mathbb{N}, j \in \llbracket 1, m \rrbracket)$, and \cite[Theorem 3.1]{Bill}, we conclude that $\{\mathcal{L}_i^{\infty}(t): i \in \mathbb{N}, t \in \mathbb{R}\}$ has the same finite-dimensional distribution as $\{(\sigma_q/\sigma) Y_i(f_q t) - (\sigma_q/\sigma) f_q^2 t^2: i \in \mathbb{N}, t \in \mathbb{R}\}$. Since finite-dimensional sets form a separating class, see \cite[Lemma 3.1]{DM21}, we conclude that $\mathcal{L}^{\infty}$ is the unique subsequential limit of $\mathcal{L}^N$, and so the latter converge to it. 

We turn to the last part of the proposition. Using the definitions of $\sigma, \sigma_q$ we get $\sigma_q/\sigma = (2f_q)^{-1/2}$, and so from the previous paragraph we see that $\{f_q^{1/2} \cdot \mathcal{L}_i^{\infty}(f_q^{-1}t): i \in \mathbb{N}, t \in \mathbb{R}\}$ has the same finite-dimensional distribution as $\{2^{-1/2} \cdot Y_i(t) - 2^{-1/2} t^2: i \in \mathbb{N}, t \in \mathbb{R}\}$. From \cite[(1.12)]{ED24a} the same is true for $\{\mathcal{L}_i^{a,b,c}(t):  i \in \mathbb{N}, t \in \mathbb{R}\}$, which implies the equality in (\ref{S62PE1}).
\end{proof}

\begin{remark} The ensembles $\mathcal{L}^{a,b,c}$ in Proposition \ref{S62P} are called Airy wanderer line ensembles. In \cite[Theorem 1.10]{ED24a} we constructed $\mathcal{L}^{a,b,c}$ by taking limits of such ensembles when $J_a^+ = J_b^+ = m \in \mathbb{Z}_{\geq 0}$, $c^+ = 0$, which were previously constructed in \cite[Proposition 3.12]{CorHamA} using results from \cite{AFM10}. By translating the latter, one obtains all ensembles in \cite[Theorem 1.10]{ED24a}. If we combine \cite[Theorem 1.8]{ED24a} and Proposition \ref{S62P}, we obtain an alternative construction of $\mathcal{L}^{a,b,c}$ directly as a weak limit of Schur processes with spiked parameters, which is independent of the results in \cite{AFM10, CorHamA}.
\end{remark}

%---------------------------------------------------------------------------------------------------------------
% Appendix A
%
%----------------------------------------------------------------------------------------------------------------
\begin{appendix}

\section{Proofs of results from Section \ref{Section2}} \label{AppendixA} In this section we prove Lemmas \ref{Lem.FinEnsSatGB}, \ref{Lem.StrongGP}, \ref{MCL} and \ref{lem:RW}. We continue with the notation from Section \ref{Section2}.

%---------------------------------------------------------------------------------------------------------------
% Appendix A.1
%
%----------------------------------------------------------------------------------------------------------------
\subsection{Proof of Lemma \ref{MCL}} \label{AppendixA1} Our proof uses similar ideas to \cite[Lemmas 3.1 and 3.2]{DEA21}, which in turn go back to \cite{CorHamA}. For clarity we split the proof into three steps. In the first step we construct a {\em maximal} element $\mathfrak{Q}^{\mathsf{max}}$ in $\Omega_{\ice}(T_0, T_1, \vec{x}, \vec{y}, \infty,g)$, provided the latter is non-empty, and show that each element $\mathfrak{Q} \in \Omega_{\ice}(T_0, T_1, \vec{x}, \vec{y}, \infty,g)$ can be obtained from $\mathfrak{Q}^{\mathsf{max}}$ after finitely many elementary moves. In the second step we construct two Markov chains $\{X^b_n: n \geq 0\}, \{X^t_n: n \geq 0\}$ on $\Omega_{\ice}(T_0, T_1, \vec{x}\,^b, \vec{y}\,^b, \infty,g^b)$ and $\Omega_{\ice}(T_0, T_1, \vec{x}\,^t, \vec{y}\,^t, \infty,g^t)$, respectively, such that for each time $n \in \mathbb{Z}_{\geq 0}$ we have that $X^t_n(i,s) \geq X^b_n(i,s)$ for $i \in \llbracket 1, k \rrbracket$ and $s \in \llbracket T_0, T_1 \rrbracket$. Furthermore, we show that $X_n^t$ converge weakly to $\mathfrak{Q}^t$ and $X_n^b$ converge weakly to $\mathfrak{Q}^b$ as in the statement of the lemma. In the third step we use the Skorohod representation theorem to construct a coupling for $\mathfrak{Q}^t, \mathfrak{Q}^b$ as in the statement of the theorem. We now turn to the details.\\

{\bf \raggedleft Step 1.} Suppose that $T_0, T_1 \in \mathbb{Z}$, $T_0 < T_1$, $\vec{x}, \vec{y} \in \mathfrak{W}_k$, and $g: \llbracket T_0, T_1 \rrbracket \rightarrow [-\infty, \infty)$ are such that $\Omega_{\ice}(T_0, T_1, \vec{x}, \vec{y}, \infty,g) \neq \emptyset$. Our first task in this step is to construct $\mathfrak{Q}^{\mathsf{max}} = \{Q_i^{\mathsf{max}}\}_{i = 1}^k \in \Omega_{\ice}(T_0, T_1, \vec{x}, \vec{y}, \infty,g)$ such that for each $\mathfrak{Q} = \{Q_i\}_{i = 1}^k \in \Omega_{\ice}(T_0, T_1, \vec{x}, \vec{y}, \infty,g) $ we have
\begin{equation}\label{A1E1}
Q_i^{\mathsf{max}}(s) \geq Q_i(s) \mbox{ for all } i \in \llbracket 1, k \rrbracket \mbox{, and } s \in \llbracket T_0, T_1\rrbracket.
\end{equation}
In the sequel we fix $\mathfrak{Q} \in \Omega_{\ice}(T_0, T_1, \vec{x}, \vec{y}, \infty,g) $ (note that the latter set is non-empty).

We define for $s \in \llbracket 0, T_1 - T_0 \rrbracket$ and $i \in \llbracket 1, k \rrbracket$
\begin{equation}\label{A1E2}
Q_i^{\mathsf{max}}(T_0 + s) = \min(x_{i-s}, y_i), \mbox{ where } x_m = \infty \mbox{ for } m \leq 0.
\end{equation}
Note that since $\mathfrak{Q} = \{Q_i\}_{i = 1}^k \in \Omega_{\ice}(T_0, T_1, \vec{x}, \vec{y}, \infty,g)$, we have for $s \in \llbracket 0, T_1 - T_0 \rrbracket$, $i \in \llbracket 1, k \rrbracket$
$$ x_{i-s} \geq Q_{i-s}(T_0) \geq \cdots \geq Q_{i-1}(T_0 + s-1) \geq Q_i(T_0 + s) \mbox{ and } y_i \geq Q_i(T_0 +s),$$
where we have set $Q_m(a) = \infty$ if $m \leq 0$. The latter and (\ref{A1E2}) show that (\ref{A1E1}) holds. Thus, we only need to check that $\mathfrak{Q}^{\mathsf{max}} = \{Q^{\mathsf{max}}_i\}_{i = 1}^k \in \Omega_{\ice}(T_0, T_1, \vec{x}, \vec{y}, \infty,g) $.

Observe that since $x_1 \geq x_2 \geq \cdots \geq x_k$, we have 
\begin{equation}\label{A1E3}
\begin{split}
&Q_i^{\mathsf{max}}(s) \in \mathbb{Z} \mbox{ for all } s \in \llbracket T_0, T_1\rrbracket \mbox{, } i \in \llbracket 1, k \rrbracket \mbox{, and } \\
&Q_i^{\mathsf{max}}(s+1) \geq Q_i^{\mathsf{max}}(s)\mbox{ for all } s \in \llbracket T_0, T_1 -1\rrbracket \mbox{, } i \in \llbracket 1, k \rrbracket.
\end{split}
\end{equation}
In addition, we have $x_i = Q_i(T_0) \leq Q_i(T_1) = y_i$, and from (\ref{A1E1}) and (\ref{A1E2}) that $y_i = Q_i(T_1) \leq Q_i^{\mathsf{max}}(T_1) \leq y_i$, which gives 
\begin{equation}\label{A1E4}
Q^{\mathsf{max}}_i(T_0) = x_i \mbox{ and } Q^{\mathsf{max}}_i(T_1) = y_i \mbox{ for } i \in \llbracket 1, k \rrbracket.
\end{equation}
We also have from (\ref{A1E1}) and $\mathfrak{Q} \in \Omega_{\ice}(T_0, T_1, \vec{x}, \vec{y}, \infty,g) $ that for $s \in \llbracket T_0 + 1, T_1\rrbracket$
\begin{equation}\label{A1E5}
Q^{\mathsf{max}}_k(s-1) \geq Q_k(s-1) \geq g(s).
\end{equation}
We finally note that since $y_1 \geq \cdots \geq y_k$, we have from (\ref{A1E2}) for $i \in \llbracket 1, k-1 \rrbracket$ and $s \in \llbracket 1, T_1 - T_0 \rrbracket$
\begin{equation}\label{A1E6}
Q^{\mathsf{max}}_i(T_0 + s-1) = \min(x_{i- s + 1}, y_i) \geq \min(x_{i -s + 1}, y_{i+1}) = Q^{\mathsf{max}}_{i+1}(T_0 + s).
\end{equation}
Combining (\ref{A1E3}), (\ref{A1E4}), (\ref{A1E5}) and (\ref{A1E6}), we see that $\mathfrak{Q}^{\mathsf{max}} = \{Q^{\mathsf{max}}_i\}_{i = 1}^k \in \Omega_{\ice}(T_0, T_1, \vec{x}, \vec{y}, \infty,g)$.\\

In the remainder of this step we construct a sequence $\mathfrak{Q}^{m} \in \Omega_{\ice}(T_0, T_1, \vec{x}, \vec{y}, \infty,g)$ for $m = 0, \dots, A$ where $A = (T_1 - T_0 + 1) \cdot k$ such that $\mathfrak{Q}^0 = \mathfrak{Q}$, $\mathfrak{Q}^{A} = \mathfrak{Q}^{\mathsf{max}}$, and for each $m \in \llbracket 0, A-1 \rrbracket$
\begin{equation}\label{A1E7}
\mathfrak{Q}^m_i(s) \neq \mathfrak{Q}^{m+1}_i(s) \mbox{ for at most one } (i,s) \in \llbracket 1, k \rrbracket \times \llbracket T_0 , T_1 \rrbracket.
\end{equation}
The latter statement will be used in the next step to prove that certain Markov chains are irreducible. 

Let us define the bijection $h: \llbracket 1, A \rrbracket \rightarrow \llbracket 1, k \rrbracket \times \llbracket T_0, T_1 \rrbracket$ via $h(m) = (a, T_1 - r)$, where $m = a \cdot (T_1 - T_0 + 1) + r$ and $0 \leq r \leq T_1 - T_0 $. In words, as $m$ goes from $1$ to $A$, $h(m)$ traverses $\llbracket 1, k \rrbracket \times \llbracket T_0, T_1 \rrbracket$ from top to bottom and right to left, see Figure \ref{S7_1}. 

\begin{figure}[h]
	\begin{center}
		\begin{tikzpicture}[scale=0.7]
		\begin{scope}
        \def\r{0.1}
		\draw[-, gray] (0,0) grid (5,3);

        \draw (4.5, 2.5 ) node{$1$};
        \draw (3.5, 2.5 ) node{$2$};
        \draw (2.5, 2.5 ) node{$3$};
        \draw (1.5, 2.5 ) node{$4$};
        \draw (0.5, 2.5 ) node{$5$};
        
        \draw (4.5, 1.5 ) node{$6$};
        \draw (3.5, 1.5 ) node{$7$};
        \draw (2.5, 1.5 ) node{$8$};
        \draw (1.5, 1.5 ) node{$9$};
        \draw (0.5, 1.5 ) node{$10$};

        \draw (4.5, 0.5 ) node{$11$};
        \draw (3.5, 0.5 ) node{$12$};
        \draw (2.5, 0.5 ) node{$13$};
        \draw (1.5, 0.5 ) node{$14$};
        \draw (0.5, 0.5 ) node{$15$};

        \draw (-0.5, 0.5 ) node{$3$};
        \draw (-0.5, 1.5 ) node{$2$};
        \draw (-0.5, 2.5 ) node{$1$};
        \draw (0.5, -0.5 ) node{$T_0$};
        \draw (4.5, -0.5 ) node{$T_1$};

		\end{scope}

		\end{tikzpicture}
	\end{center}
 \vspace{-4mm}
	\caption{The figure depicts the bijection $h: \llbracket 1, A \rrbracket \rightarrow \llbracket 1, k \rrbracket \times \llbracket T_0, T_1 \rrbracket$ when $T_1 - T_0 = 4$, $k = 3$ and $A = 15$.}
	\label{S7_1}
\end{figure}

We set $\mathfrak{Q}^0 = \mathfrak{Q}$, and for $m \in \llbracket 1, A\rrbracket$, $(i,s) \in \llbracket 1, k \rrbracket \times \llbracket T_0, T_1 \rrbracket$ we set
\begin{equation}\label{A1E8}
Q_i^{m}(s) = Q_i^m(s) \mbox{ if } h^{-1}(i,s) > m \mbox{ and } Q_{i}^{m}(s) = Q_{i}^{\mathsf{max}}(s) \mbox{ if } h^{-1}(i,s) \leq m,
\end{equation}
and then denote $\mathfrak{Q}^m = \{Q_i^m\}_{i = 1}^k$. The last equation implies (\ref{A1E7}), and also $\mathfrak{Q}^0 = \mathfrak{Q}$, while $\mathfrak{Q}^A = \mathfrak{Q}^{\mathsf{max}}$. Consequently, we only need to check that $\mathfrak{Q}^m \in \Omega_{\ice}(T_0, T_1, \vec{x}, \vec{y}, \infty,g)$. The latter is clear when $m = 0$. Assuming that $\mathfrak{Q}^m \in \Omega_{\ice}(T_0, T_1, \vec{x}, \vec{y}, \infty,g)$, and $m \leq A-1$, we next check that $\mathfrak{Q}^{m+1} \in \Omega_{\ice}(T_0, T_1, \vec{x}, \vec{y}, \infty,g)$.

Since $\mathfrak{Q}^m \in \Omega_{\ice}(T_0, T_1, \vec{x}, \vec{y}, \infty,g)$, and $Q_i^m(s) = Q_i^{m+1}(s)$ when $h(m+1) \neq (i,s)$, we see that we only need to check the monotonicity and interlacing inequalities at $(i,s) = h(m + 1)$. I.e. we need to verify that 
\begin{equation}\label{A1E9}
\begin{split}
&Q_i^{m+1}(s) \leq Q_i^{m+1}(s+1) \mbox{ if } s \in \llbracket T_0, T_1-1\rrbracket, \hspace{2mm} Q_i^{m+1}(s-1) \leq Q_i^{m+1}(s) \mbox{ if } s \in \llbracket T_0+1, T_1\rrbracket, \\
&Q_{i+1}^{m+1}(s+1) \leq Q_i^{m+1}(s) \mbox{ if } s \in \llbracket T_0, T_1-1\rrbracket, \hspace{2mm} Q_i^{m+1}(s) \leq Q_{i-1}^{m+1}(s-1) \mbox{ if } s \in \llbracket T_0+1, T_1\rrbracket, 
\end{split}
\end{equation}
where the convention is that $Q_0^{m+1}(a) = \infty$ and $Q_{k+1}^m(a) = g(a)$. The first inequality on the first line in (\ref{A1E9}) follows from
$$Q_i^{m+1}(s) = Q_{i}^{\mathsf{max}}(s) \leq Q_{i}^{\mathsf{max}}(s+1) = Q_i^{m+1}(s+1),$$
and the second on the first line in (\ref{A1E9}) from
$$Q_i^{m+1}(s) = Q_{i}^{\mathsf{max}}(s) \geq Q_{i}^{\mathsf{max}}(s-1) \geq Q_{i}(s-1) = Q_i^{m+1}(s-1),$$
where the second inequality used (\ref{A1E1}). One similarly has that the first inequality on the second line in (\ref{A1E9}) follows from
$$Q_{i+1}^{m+1}(s+1) = Q_{i+1}(s+1) \leq Q^{\mathsf{max}}_{i+1}(s + 1) \leq Q^{\mathsf{max}}_{i}(s) =   Q_i^{m+1}(s),$$
while the second inequality on the second line in (\ref{A1E9}) from
$$Q_i^{m+1}(s) =  Q^{\mathsf{max}}_{i}(s) \leq Q^{\mathsf{max}}_{i-1}(s) = Q_{i-1}^{m+1}(s-1),$$
where as before $Q_0(a) = Q_0^{\mathsf{max}}(a) = \infty$ and $Q_{k+1}(a) = Q_{k+1}^{\mathsf{max}}(a) = g(a)$. This completes the proof of (\ref{A1E9}) and hence our construction of $\mathfrak{Q}^m$.\\

{\bf \raggedleft Step 2.} In this step we construct two Markov chains $\{X^b_n: n \geq 0\}, \{X^t_n: n \geq 0\}$ on the same probability space, taking values in $\Omega_{\ice}(T_0, T_1, \vec{x}\,^b, \vec{y}\,^b, \infty,g^b)$ and $\Omega_{\ice}(T_0, T_1, \vec{x}\,^t, \vec{y}\,^t, \infty,g^t)$, respectively. Let $\mathfrak{Q}^{b, \mathsf{max}} \in \Omega_{\ice}(T_0, T_1, \vec{x}\,^b, \vec{y}\,^b, \infty,g^b)$ and $\mathfrak{Q}^{t, \mathsf{max}} \in \Omega_{\ice}(T_0, T_1, \vec{x}\,^t, \vec{y}\,^t, \infty,g^t)$ be the maximal elements we constructed in Step 1. We then define $X_0^b = \mathfrak{Q}^{b, \mathsf{max}}$ and $X_0^t = \mathfrak{Q}^{t, \mathsf{max}}$, and note that when $n = 0$, we have from (\ref{A1E2}) that
\begin{equation}\label{A1E10}
X_n^b(i,s) \geq X_n^t(i,s) \mbox{ for } (i,s) \in \llbracket 1, k \rrbracket \times \llbracket T_0, T_1 \rrbracket.
\end{equation}
We now consider a sequence of i.i.d. uniform points $(A_n,B_n)$ in $\llbracket 1, k \rrbracket \times \llbracket T_0, T_1 \rrbracket$, as well as a sequence of i.i.d. uniform random variables $U_n$ on $(0,1)$. We consider the following update for $X_n^t$ and $X_n^b$. 
\begin{itemize}
\item If $B_n \in \{T_0, T_1\}$, then set $X_{n+1}^b = X_n^b$ and $X_{n+1}^t = X_n^t$.
\item If $B_n \in \llbracket T_0 + 1, T_1 - 1 \rrbracket$, then set 
\begin{equation}\label{A1E11}
X_{n+1}^{b/t}(i,s) = \begin{cases} X_{n}^{b/t}(i,s) &\mbox{ if } (i,s) \neq (A_n, B_n), \\
C^{b/t}_n + \lfloor U_n \cdot (D^{b/t}_n - C^{b/t}_n + 1) \rfloor, &\mbox{ if } (i,s) = (A_n,B_n),
\end{cases}
\end{equation}
where $C^{b/t}_n = \max(X_n^{b/t}(A_n + 1, B_n + 1), X_n^{b/t}(A_n, B_n - 1))$ and $D^{b/t}_n = \min(X_n^{b/t}(A_n-1, B_n - 1),(X_n^{b/t}(A_n, B_n + 1)  )$.
\end{itemize}
In words, at each step we pick a coordinate $(A_n, B_n)$ for $X_n^{b/t}$ to update uniformly at random. If that coordinate is on the side, i.e. $B_n \in \{T_0, T_1\}$, then we do not change anything. If that coordinate is in $\llbracket 1, k \rrbracket \times \llbracket T_0 + 1, T_1 +1 \rrbracket$, then we pick a uniform value in $\llbracket C_n^{b/t},D_n^{b/t} \rrbracket$ and change $X_{n+1}^{b/t}(A_n,B_n)$ to this value. The interval $\llbracket C_n^{b/t},D_n^{b/t} \rrbracket$ is precisely the interval of possible values of $X_{n+1}^{b/t}(A_n,B_n)$ that maintains monotonicity of the paths and the interlacing conditions. We finally mention that our convention is that $X_n^{b/t}(0,a) = \infty$, while $X_n^{b/t}(k+1,a) = g^{b/t}(a)$.\\

The above specifies our dynamics, and it is clear that separately $\{X^b_n: n \geq 0\}, \{X^t_n: n \geq 0\}$ are Markov in their own filtrations. We next show that (\ref{A1E10}) holds for all $n \geq 0$. Since by construction it holds when $n = 0$, it suffices to show that the update rule (\ref{A1E11}) maintains it. Assuming that (\ref{A1E10}) holds at $n$, we note that $C^{t}_n \geq C^{b}_n$, and $D^{t}_n \geq D^{b}_n$. Combining the latter with (\ref{A1E11}), we see that to show (\ref{A1E10}) holds at $n+1$ it suffices to show that for any integers $a,b,c,d$ with $b\geq a$, $d \geq c$, $c \geq a$, $d \geq b$ and $x \in (0,1)$ we have  
\begin{equation}\label{A1E12}
c + \lfloor x \cdot (d - c + 1) \rfloor \geq a + \lfloor x \cdot (b - a + 1) \rfloor.
\end{equation}
The latter can be seen from the following simple inequalities
$$a + \lfloor x \cdot (b - a + 1) \rfloor \leq a + \lfloor x \cdot (d - a + 1) \rfloor \mbox{ and } a + \lfloor x \cdot (d - a + 1) \rfloor \leq c + \lfloor x \cdot (d - c + 1) \rfloor,$$
where the second one holds as $\lfloor x (p + q) \rfloor \leq \lfloor xp \rfloor + q$ for any $p,q \in \mathbb{Z}_{\geq 0}$ and $x \in (0,1)$.\\

The last thing we show is that the Markov chains $\{X^b_n: n \geq 0\}, \{X^t_n: n \geq 0\}$ are irreducible, aperiodic and have invariant distributions given by the uniform distributions on the sets $\Omega_{\ice}(T_0, T_1, \vec{x}\,^b, \vec{y}\,^b, \infty,g^b)$ and $\Omega_{\ice}(T_0, T_1, \vec{x}\,^t, \vec{y}\,^t, \infty,g^t)$, respectively. We proceed to verify the three conditions and to ease notation we drop the superscripts $b,t$.

From (\ref{A1E11}) we have for each $\mathfrak{Q} \in \Omega_{\ice}(T_0, T_1, \vec{x}, \vec{y}, \infty,g)$ that 
$$\mathbb{P}(X_{n+1} = \mathfrak{Q} | X_{n} = \mathfrak{Q}) > 0,$$
and so the chain is aperiodic. In addition, if $\mathfrak{Q} \in \Omega_{\ice}(T_0, T_1, \vec{x}, \vec{y}, \infty,g)$, and $\mathfrak{Q}^{\mathsf{max}}$ is as in Step 1, while $\mathfrak{Q}^{m}$ satisfy (\ref{A1E7}) we have from (\ref{A1E11}) that 
$$\mathbb{P}(X_{n+1} = \mathfrak{Q}^{m+1} | X_{n} = \mathfrak{Q}^m) > 0 \mbox{ and } \mathbb{P}(X_{n+1} = \mathfrak{Q}^{m} | X_{n} = \mathfrak{Q}^{m+1}) > 0.$$
In particular, the chain is irreducible. Finally, from (\ref{A1E11}) we have for $\mathfrak{Q}, \mathfrak{Q}' \in \Omega_{\ice}(T_0, T_1, \vec{x}, \vec{y}, \infty,g)$ 
$$\mathbb{P}(X_{n+1} = \mathfrak{Q}' | X_{n} = \mathfrak{Q}) = \mathbb{P}(X_{n+1} = \mathfrak{Q} | X_{n} = \mathfrak{Q}').$$
The latter shows that the transition matrix for $X_n$ is doubly stochastic and so the uniform measure on $ \Omega_{\ice}(T_0, T_1, \vec{x}, \vec{y}, \infty,g)$ is invariant.\\

{\bf \raggedleft Step 3.} From our work in the last two paragraphs in Step 2, and \cite[Theorem 1.8.3]{Norris} we conclude that $\{X^b_n: n \geq 0\}$ and $\{X^t_n: n \geq 0\}$ weakly converge to $\mathbb{P}_{\ice,\operatorname{Geom}}^{T_0, T_1, \vec{x}\,^b, \vec{y}\,^b, \infty, g^b}$ and  $\mathbb{P}_{\ice,\operatorname{Geom}}^{T_0, T_1, \vec{x}\,^t, \vec{y}\,^t, \infty, g^t}$, respectively. In particular, $\{X^b_n: n \geq 0\}$ and $\{X^t_n: n \geq 0\}$ are tight sequences, and then so is $\{(X^b_n, X^t_n): n \geq 0\}$. By Prohorov's theorem, we conclude that $\{(X^b_n, X^t_n): n \geq 0\}$ is relatively compact, and suppose that $\{(X^b_{n_m}, X^t_{n_m}): m \geq 0\}$ is a weakly convergent subsequence. By the Skorohod representation theorem, \cite[Theorem 6.7]{Bill}, we can assume that this sequence is defined on the same probability space $(\Omega, \mathcal{F}, \mathbb{P})$ and the convergence is for each $\omega \in \Omega$. Denoting the limit by $(X^b_{\infty}, X^t_{\infty})$ we see that $X^b_{\infty}$ has law $\mathbb{P}_{\ice,\operatorname{Geom}}^{T_0, T_1, \vec{x}\,^b, \vec{y}\,^b, \infty, g^b}$, while $X^t_{\infty}$ has law $\mathbb{P}_{\ice,\operatorname{Geom}}^{T_0, T_1, \vec{x}\,^t, \vec{y}\,^t, \infty, g^t}$. From (\ref{A1E10}) we know that 
$$X^t_{n_m}(i,s) \geq X^b_{n_m}(i,s) \mbox{ for } (i,s) \in \llbracket 1, k \rrbracket \times \llbracket T_0, T_1 \rrbracket.$$
Taking the limit as $m \rightarrow \infty$, we conclude 
$$X^t_{\infty}(i,s) \geq X^b_{\infty}(i,s) \mbox{ for } (i,s) \in \llbracket 1, k \rrbracket \times \llbracket T_0, T_1 \rrbracket.$$
In particular, we see that $(\mathfrak{Q}^b, \mathfrak{Q}^t) = (X^b_{\infty}, X^t_{\infty})$ satisfy the conditions of the lemma.

%---------------------------------------------------------------------------------------------------------
%
% Appendix A.2
%
%---------------------------------------------------------------------------------------------------------
\subsection{Proof of Lemma \ref{lem:RW}} \label{AppendixA2} We start with a simple statement we will require, which is an immediate consequence of \cite[Corollary 2.9]{CorHamA}.

\begin{lemma}\label{NoTouch} Fix $a,b,x,y \in \mathbb{R}$ with $a < b$ and a continuous function $f \in C([a,b])$ such that $f(a) > x$ and $f(b) > y$. Let $B$ be a Brownian bridge from $B(a) = x$ to $B(b) = y$ with variance $1$ as in (\ref{S22E3}). Define the events $C = \{ B(t) > f(t) \mbox{ for some $t \in [a,b]$}\}$ (crossing) and $T = \{ B(t) = f(t) \mbox{ for some } t\in [a,b]\}$ (touching). Then, $\mathbb{P}(T \cap C^c) = 0.$
\end{lemma}
\begin{proof} Similarly to the proof of Lemma \ref{CHL}, we let $F: C([0,1]) \rightarrow C([a,b])$ be given by 
$$F(g)(t) = (b-a)^{1/2} \cdot g \left(\frac{t-a}{b-a} \right) + \left( \frac{b-t}{b-a} \right) \cdot x + \left( \frac{t-a}{b-a} \right) \cdot y,$$
and note that $F$ is a diffeomorphism, and $F^{-1}(B)$ is a standard Brownian motion. Defining $\tilde{C} = \{ F^{-1}(B)(t) > F^{-1}(f)(t) \mbox{ for some $t \in [0,1]$}\}$, and $\tilde{T} = \{ F^{-1}(B)(t) = F^{-1}(f)(t) \mbox{ for some } t\in [0,1]\}$, we have $\mathbb{P}(T \cap C^c)  = \mathbb{P}(\tilde{T} \cap \tilde{C}^c)$. The latter is equal to zero by \cite[Corollary 2.9]{CorHamA}.
\end{proof}

\begin{proof}[Proof of Lemma \ref{lem:RW}] The proof we present is similar to that of \cite[Lemma 4.6]{DEA21}. For clarity we split the proof into three steps. In the first step we establish the result when $k = 1$, $f = \infty$ and $g = -\infty$. In the second step we construct $N_0$ in the first statement of the lemma, and in the third step we prove the second statement of the lemma. \\

{\bf \raggedleft Step 1.} In this step we assume that $k = 1$, $f = \infty$ and $g = -\infty$. We note that since $p > 0$ and $d_n \rightarrow \infty$, we can find $N_0 \in \mathbb{N}$ such that for $n \geq N_0$ we have $X_1^n \leq Y_1^n$, $F_n = \infty$, $G_n = - \infty$. In particular, for $n \geq N_0$ we have that $\Omega_{\ice}(A_n, B_n, X_1^n, Y_1^n, \infty, -\infty)$ is non-empty.

Let us set 
\begin{equation}\label{A2E1}
a_n = d_n^{-1}A_n, \hspace{2mm} b_n = d_n^{-1} B_n, \hspace{2mm} x_1^n = \sigma^{-1} d_n^{-1/2}(X_1^n - pA_n), \hspace{2mm} y_1^n = \sigma^{-1} d_n^{-1/2}(Y_1^n - pB_n).
\end{equation}
Let $\tilde{B}$ be a standard Brownian bridge on $[0,1]$, and define random elements $B^n$, $B$ in $C([a_n, b_n])$ and $C([a,b])$, respectively, via
\begin{equation}\label{A2E2}
\begin{split}
&B^n(t) = \sqrt{b_n - a_n } \cdot \tilde{B} \left( \frac{t - a_n}{b_n - a_n} \right) + \frac{t - a_n}{b_n - a_n} \cdot y_1^n + \frac{b_n - t}{b_n - a_n} \cdot x_1^n \\
&B(t) = \sqrt{b - a } \cdot \tilde{B} \left( \frac{t - a}{b - a} \right) + \frac{t - a}{b - a} \cdot y_1 + \frac{b - t}{b - a} \cdot x_1.
\end{split}
\end{equation}
Note that $B$ has law $\mathbb{P}_{\operatorname{avoid}}^{a,b,x_1,y_1,\infty, - \infty}$, which is the same as $\mathbb{P}_{\operatorname{free}}^{a,b,x_1, y_1}$, and $B^n\vert_{[a,b]} \Rightarrow B$ as random elements in $C([a,b])$ as $n \rightarrow \infty$ (the convergence is actually almost sure).

On the other hand, by Proposition \ref{KMT} (applied to $A = \sigma \cdot (b-a)^{-1/2} \cdot (|x_1 - y_1| + 1)$, $p$ as in the present lemma, $n = B_n - A_n$, $z = Y_1^n - X_1^n$), we have for any $\epsilon > 0$ and all large $n$ that there is a probability space that supports $\tilde{B}$ and $\tilde{Q}_1^n$ with law $\mathbb{P}_{ \operatorname{Geom}}^{0,B_n - A_n, 0, Y_1^n - X_1^n}$ such that
\begin{equation*}
\mathbb{P}\left( \sup_{t \in [0, B_n - A_n]}  \left| \sqrt{B_n - A_n} \cdot \sigma \tilde{B}\left( \frac{t}{B_n - A_n}\right) + \frac{t(Y_1^n - X_1^n) }{B_n - A_n}  - \tilde{Q}_1^n(t)  \right| > (B_n - A_n)^{1/4} \right) < \epsilon.
\end{equation*}
By translating $\tilde{Q}_1^n$ and using (\ref{A2E2}) we conclude that there is a probability space that supports $B^n$ and $Q_1^n$ with law $\mathbb{P}_{\operatorname{Geom}}^{A_n,B_n , X_1^n, Y_1^n} = \mathbb{P}_{\ice, \operatorname{Geom}}^{A_n,B_n , X_1^n, Y_1^n, -\infty, \infty}$ such that
\begin{equation}\label{A2E3}
\mathbb{P}\left( \sup_{t \in [A_n, B_n]}  \left|  \sigma \cdot d_n^{1/2} \cdot B^n (t/d_n)  - Q_1^n(t) + tp  \right| > (B_n - A_n)^{1/4} \right) < \epsilon.
\end{equation}
From (\ref{A2E3}) we conclude that if $\mathcal{Q}_1^n$ is as in (\ref{S22E5}), then we can couple it with $B^n$ so that
\begin{equation}\label{A2E4}
\lim_{n \rightarrow \infty} \mathbb{P} \left(\sup_{t \in [a,b]} \left| B^n(t) - \mathcal{Q}_1^n(t) \right| > \delta \right) = 0
\end{equation}
for any $\delta > 0$. From (\ref{A2E4}), $B^n\vert_{[a,b]} \Rightarrow B$, and the convergence together lemma, see \cite[Theorem 3.1]{Bill}, we conclude $\mathcal{Q}_1^n \Rightarrow B$, which proves the second statement of the lemma.\\

{\bf \raggedleft Step 2.} In this step we find $N_0$ as in the first part of the lemma. For $u_1, u_2 \in C([a,b])$, we write $\rho(u_1, u_2) = \sup_{x \in [a,b]} |u_1(x)- u_2(x)|$.

Note that we can find $\epsilon > 0$ and $h_1, \dots, h_k: C([a,b])$, depending on $a,b,\vec{x}, \vec{y}, f,g$ with $h_i(a) = x_i$, $h_i(b) = y_i$ for $i \in \llbracket 1, k \rrbracket$ such that if $u_i \in C([a,b])$ satisfy $\rho(u_i, h_i) < \epsilon$, then 
\begin{equation}\label{A2E5}
f(x) - \epsilon > u_1(x) + \epsilon > u_1(x) - \epsilon > \cdots > u_k(x) + \epsilon > u_k(x) - \epsilon > g(x) + \epsilon 
\end{equation}
for all $x \in [a,b]$. By Lemma \ref{CHL} for $\{B_i\}_{i = 1}^k$ with law $\mathbb{P}_{\operatorname{free}}^{a,b, \vec{x}, \vec{y}}$ we have
\begin{equation}\label{A2E6}
\mathbb{P}_{\operatorname{free}}^{a,b,\vec{x}, \vec{y}} \left( \rho(B_i, h_i) < \epsilon \mbox{ for } i \in \llbracket 1, k \rrbracket \right) > 0.
\end{equation}
From Step 1 we have for all large $n$ that the law $\mathbb{P}_{\mathsf{Geom}}^{A_n, B_n, X_i^n, Y_i^n}$ is well-defined. Furthermore, from the work in that step we have that if $\tilde{\mathfrak{Q}}^n = \{\tilde{Q}^n_i\}_{i = 1}^k$ have law $\mathbb{P}_{\mathsf{Geom}}^{A_n, B_n, \vec{X}^n, \vec{Y}^n}$, then the line ensembles $\tilde{\mathcal{Q}}^n$ on $[a,b]$ defined as
\begin{equation}\label{A2E7}
\tilde{\mathcal{Q}}_i^n(s) = \sigma^{-1} d_n^{-1/2} \cdot \left( \tilde{Q}^n_i(td_n) - p td_n \right)
\end{equation}
converge weakly to $\mathbb{P}_{\operatorname{free}}^{a,b, \vec{x}, \vec{y}}$. In view of (\ref{A2E6}) we conclude that there exists $N_1 \in \mathbb{N}$ such that for $n \geq N_1$, we have that $\mathbb{P}_{\mathsf{Geom}}^{A_n, B_n, \vec{X}^n, \vec{Y}^n}$ is well-defined and 
\begin{equation}\label{A2E8}
\begin{split}
&\mathbb{P}_{\mathsf{Geom}}^{A_n, B_n, \vec{X}^n, \vec{Y}^n} ( E_n ) > 0, \mbox{ where } \\
&E_n = \{\tilde{\mathfrak{Q}}^n \in \Omega_{\operatorname{Geom}}(A_n, B_n, \vec{X}^n, \vec{Y}^n): \rho(\tilde{\mathcal{Q}}^n_i, h_i) < \epsilon \mbox{ for } i \in \llbracket 1, k \rrbracket \}.
\end{split}
\end{equation}

In the remainder of this step, we show that there exists $N_2 \in \mathbb{N}$ such that for $n \geq \max(N_1, N_2)$
\begin{equation}\label{A2E9}
\begin{split}
&E_n \subseteq \Omega_{\ice}(A_n, B_n, \vec{X}^n, \vec{Y}^n, F_n, G_n),
\end{split}
\end{equation}
where we recall from (\ref{EventInter}) that 
\begin{equation}\label{A2E10}
\begin{split}
& \Omega_{\ice}(A_n, B_n, \vec{X}^n, \vec{Y}^n, F_n, G_n) = \{\tilde{\mathfrak{Q}}^n \in \Omega_{\operatorname{Geom}}(A_n, B_n, \vec{X}^n, \vec{Y}^n):  \tilde{Q}^n_i(r-1) \geq \tilde{Q}_{i+1}(r) \\
&\mbox{ for all $r \in \llbracket A_n + 1, B_n \rrbracket$ and $i \in \llbracket 0 , k \rrbracket$} \}.
\end{split}
\end{equation}
In (\ref{A2E10}) and below we adopt the convention $\tilde{Q}^n_0(r) = F_n(r)$, $\tilde{Q}_{k+1}^n(r) = G_n(r)$, $\tilde{\mathcal{Q}}^n_{0}(t) = f_n(t)$, and $\tilde{\mathcal{Q}}^n_{k+1}(t) = g_n(t)$. If (\ref{A2E9}) is true, then from (\ref{A2E8}) $E_n \neq \emptyset$ and so $\Omega_{\ice}(A_n, B_n, \vec{X}^n, \vec{Y}^n, F_n, G_n) \neq \emptyset$ as well, establishing the first part of the lemma with $N_0 = \max(N_1, N_2)$.\\

Using (\ref{A2E5}), the continuity of $h_i$, $f$, $g$ and the uniform convergence of $f_n, g_n$ to $f$ and $g$ on $[a,b]$, we can find $N_{2,1} \in \mathbb{N}$ and $\delta > 0$ (depending on these functions and $\epsilon$) such that if $n \geq N_{2,1}$, $x,y \in [a,b]$, $|x-y| \leq \delta$, and $u_i \in C([a,b])$ satisfy $\rho(u_i, h_i) < \epsilon$, we have
\begin{equation}\label{A2E11}
u_i(x) -\epsilon \geq u_{i + 1}(y) \mbox{ for } i \in \llbracket 0, k \rrbracket \mbox{ with $u_0(x) = f_n(x)$ and $u_{k+1}(x) = g_n(x)$ for $x \in [a,b]$}.
\end{equation}
Suppose that $N_{2,2} \in \mathbb{N}$ is such that for $n \geq N_{2,2}$ we have 
\begin{equation}\label{A2E12}
\delta > d_{n}^{-1}, \hspace{2mm} \sigma\epsilon d_n^{1/2} > 2p.
\end{equation}
We then have for $n \geq \max(N_{2,1}, N_{2,2})$, $r, r - 1 \in [a d_n, b d_n]$ and $\tilde{\mathfrak{Q}}^n \in E_n$ that the following holds
\begin{equation}\label{A2E13}
\begin{split}
&\tilde{Q}^n_{i+1}(r) =\sigma d_n^{1/2} \tilde{\mathcal{Q}}^n_{i+1}(r/d_n) + r p \leq \sigma d_n^{1/2} \tilde{\mathcal{Q}}_{i}^n((r-1)/d_n) + rp - \epsilon \sigma d_n^{1/2} \\
&= \tilde{Q}^n_{i}(r-1) + p - \epsilon \sigma d_n^{1/2} \leq \tilde{Q}^n_{i}(r-1) \mbox{ for $i \in \llbracket 0, k \rrbracket$.}
\end{split}
\end{equation}
We mention that the inequality in the first line used (\ref{A2E11}) and the fact that $d_n^{-1} < \delta$ from (\ref{A2E12}), and the inequality on the second line used that $\sigma\epsilon d_n^{1/2} > p$ -- again from (\ref{A2E12}). 

Equation (\ref{A2E13}) verifies the interlacing conditions in (\ref{A2E10}) except when $r = A_n+1$ and $r = B_n$, which we need to check separately. From (\ref{EdgeLim}) and (\ref{SideLim}) we can find $N_{2} \geq \max(N_{2,1}, N_{2,2})$ large enough so that for $n \geq N_{2}$ and $\tilde{\mathfrak{Q}}^n \in \Omega_{\operatorname{Geom}}(A_n, B_n, \vec{X}^n, \vec{Y}^n)$ 
\begin{equation}\label{A2E14}
\begin{split}
&\tilde{Q}^n_i(A_n) = X_i^n \in \left[ (x_i - \epsilon/2)  \sigma d_n^{1/2} + p A_n, (x_i + \epsilon/2) \sigma d_n^{1/2} + p A_n \right], \mbox{ for } i \in \llbracket 1, k \rrbracket  \\
&\tilde{Q}^n_i(B_n) = Y_i^n \in \left[ (y_i - \epsilon/2)  \sigma d_n^{1/2} + p B_n, (y_i + \epsilon/2) \sigma d_n^{1/2} + p B_n \right], \mbox{ for } i \in \llbracket 1, k \rrbracket, \\
&F_n(A_n) \geq \sigma d_n^{1/2} (f_n(a) - \epsilon) + p (A_n + 1), \hspace{2mm} G_n(B_n) \leq \sigma d_n^{1/2} (g_n(b) + \epsilon) + p(B_n-1).
\end{split}
\end{equation}
From (\ref{A2E11}) and (\ref{A2E12}) when $n \geq N_2$ and $\tilde{\mathfrak{Q}}^n \in E_n$ we have
\begin{equation}\label{A2E15}
\begin{split}
& \tilde{Q}^n_{i+1}(A_n + 1) \leq \sigma d_n^{1/2}\left( h_{i}(a) - \epsilon \right) + p (A_n + 1) = \sigma d_n^{1/2}\left( x_i - \epsilon \right) + p (A_n + 1)  \mbox{ for } i \in \llbracket 1 , k \rrbracket, \\
& \tilde{Q}^n_{1}(A_n + 1) \leq \sigma d_n^{1/2}\left( f_n(a) - \epsilon \right) + p (A_n + 1).
\end{split}
\end{equation}
Combining (\ref{A2E14}), (\ref{A2E15}), and the fact that $\sigma \epsilon d_n^{1/2} > 2p$ from (\ref{A2E12}), we conclude for $n \geq N_{2}$
\begin{equation*}
\begin{split}
&\tilde{Q}^n_{i+1}(A_n + 1) \leq \sigma d_n^{1/2}\left( x_i - \epsilon \right) + p (A_n + 1) \leq \tilde{Q}_i^n(A_n)  - \sigma d_n^{1/2} (\epsilon/2) + p \leq  \tilde{Q}_i^n(A_n) \mbox{ for } i \in \llbracket 1, k  \rrbracket, \\
&\tilde{Q}^n_{1}(A_n + 1) \leq \sigma d_n^{1/2}\left( f_n(a) - \epsilon \right) + p (A_n + 1) \leq F_n(A_n).
\end{split}
\end{equation*}
One analogously shows for $n \geq N_2$ and $\tilde{\mathfrak{Q}}^n \in E_n$ that
\begin{equation*}
\begin{split}
&\tilde{Q}^n_{i}(B_n - 1) \geq \sigma d_n^{1/2}\left( y_{i+1} + \epsilon \right) + p (B_n-1) \geq \tilde{Q}_{i+1}^n(B_n)  - \sigma d_n^{1/2} (\epsilon/2) - p \geq  \tilde{Q}_i^n(B_n) \\
&\mbox{ for } i \in \llbracket 0, k -1 \rrbracket, \mbox{ and } \tilde{Q}_k^n(B_n-1) \geq \sigma d_n^{1/2}\left( g_n(b) + \epsilon \right) + p(B_n-1) \geq G_n(B_n).
\end{split}
\end{equation*}
The last two displayed equations verify that any $\tilde{\mathfrak{Q}}^n \in E_n$ satisfies the inequalities in (\ref{A2E10}) when $r = A_n + 1$ or $B_n$. These two equations and (\ref{A2E13}) imply that (\ref{A2E9}) holds for $n \geq N_2$.\\

{\bf \raggedleft Step 3.} In this step we establish the second part of the lemma. Define $I: C(\llbracket 1, k \rrbracket \times [a,b]) \rightarrow \mathbb{R}$ by 
\begin{equation}\label{A2E16}
I(\{h_i\}_{i = 1}^k) = {\bf 1}_{E_{\operatorname{avoid}}} \mbox{, where } E_{\operatorname{avoid}} = \{f(s) > h_1(s) > \cdots > h_k(s) > g(s) \mbox{ for } s \in [a,b]\},
\end{equation}
and $I_n: \Omega_{\operatorname{Geom}}(A_n, B_n, \vec{X}^n, \vec{Y}^n) \rightarrow \mathbb{R}$ by
\begin{equation}\label{A2E17}
I_n(\tilde{\mathfrak{Q}}^n) = {\bf 1}\{ \tilde{\mathfrak{Q}}^n \in  \Omega_{\ice}(A_n, B_n, \vec{X}^n, \vec{Y}^n, F_n, G_n) \},
\end{equation}
where we recall that $\Omega_{\ice}(A_n, B_n, \vec{X}^n, \vec{Y}^n, F_n, G_n)$ is as in (\ref{A2E10}). We claim that for any bounded continuous function $H: C(\llbracket 1, k \rrbracket \times [a,b]) \rightarrow \mathbb{R}$ we have
\begin{equation}\label{A2E18}
\lim_n \mathbb{E}_{\mathsf{Geom}}^{A_n, B_n, \vec{X}^n, \vec{Y}^n} \left[ H(\tilde{\mathcal{Q}}^n) \cdot I_n(\tilde{\mathfrak{Q}}^n) \right] = \mathbb{E}_{\operatorname{free}}^{a, b, \vec{x}, \vec{y}} \left[ H(\{B_i\}_{i = 1}^k) \cdot I(\{B_i\}_{i = 1}^k) \right].
\end{equation}
If (\ref{A2E18}) is true, then for any bounded continuous $H: C(\llbracket 1, k \rrbracket \times [a,b]) \rightarrow \mathbb{R}$ we have
\begin{equation*}
\begin{split}
&\lim_n \mathbb{E}_{\ice, \operatorname{Geom}}^{A_n,B_n, \vec{X}^n, \vec{Y}^n, F_n, G_n} \left[ H(\mathcal{Q}^n) \right] = \lim_n \frac{\mathbb{E}_{\mathsf{Geom}}^{A_n, B_n, \vec{X}^n, \vec{Y}^n} \left[ H(\tilde{\mathcal{Q}}^n) \cdot I_n(\tilde{\mathfrak{Q}}^n) \right]}{\mathbb{E}_{\mathsf{Geom}}^{A_n, B_n, \vec{X}^n, \vec{Y}^n} \left[  I_n(\tilde{\mathfrak{Q}}^n) \right]} \\
&= \frac{\mathbb{E}_{\operatorname{free}}^{a, b, \vec{x}, \vec{y}} \left[ H(\{B_i\}_{i = 1}^k) \cdot I(\{B_i\}_{i = 1}^k) \right] }{\mathbb{E}_{\operatorname{free}}^{a, b, \vec{x}, \vec{y}} \left[  I(\{B_i\}_{i = 1}^k) \right]} = \mathbb{E}_{\operatorname{avoid}}^{a, b, \vec{x}, \vec{y},f,g} \left[ H(\{B_i\}_{i = 1}^k) \right],
\end{split}
\end{equation*}
which establishes the second part of the lemma. We remark that $\mathbb{E}_{\operatorname{free}}^{a, b, \vec{x}, \vec{y}} \left[  I(\{B_i\}_{i = 1}^k) \right] > 0$ in view of (\ref{A2E5}) and (\ref{A2E6}). In the remainder we establish (\ref{A2E18}).\\

Define $E_{\operatorname{cross}} \subset C(\llbracket 1, k \rrbracket \times [a,b])$ by
\begin{equation}\label{A2E19}
E_{\operatorname{cross}} = \{\{h_i\}_{i = 1}^k: h_{i+1}(s) > h_{i}(s) \mbox{ for some } s \in [a,b] \mbox{ and } i \in \llbracket 0, k \rrbracket \},
\end{equation}
where $h_0(s) = f(s)$ and $h_{k+1}(s) = g(s)$, and note that from Lemma \ref{NoTouch} we have
\begin{equation}\label{A2E20}
\mathbb{P}_{\operatorname{free}}^{a,b, \vec{x}, \vec{y}}\left( \{B_i\}_{i = 1}^k \in E_{\operatorname{cross}} \cup E_{\operatorname{avoid}} \right) = 1.
\end{equation}
From Step 1 we know that if $\tilde{\mathfrak{Q}}^n$ have laws $\mathbb{P}_{\mathsf{Geom}}^{A_n, B_n, \vec{X}^n, \vec{Y}^n}$ and $\tilde{\mathcal{Q}}^n = \{\tilde{\mathcal{Q}}_i^n \}_{i = 1}^k$ are as in (\ref{A2E7}), then $\tilde{\mathcal{Q}}^n$ converge weakly to $\mathbb{P}_{\operatorname{free}}^{a,b, \vec{x}, \vec{y}}$. By the Skorohod representation theorem, \cite[Theorem 6.7]{Bill}, we can find a probability space $(\Omega, \mathcal{F}, \mathbb{P})$ that supports $\tilde{\mathfrak{Q}}^n$ and $\{B_i\}_{i = 1}^k$ with law $\mathbb{P}_{\operatorname{free}}^{a,b, \vec{x}, \vec{y}}$ such that for each $\omega \in \Omega$ and $i \in \llbracket 1, k \rrbracket$
\begin{equation}\label{A2E21}
\lim_{n} \sup_{t \in [a,b]} \left| \tilde{\mathcal{Q}}^n_i(s) - B_i(s) \right| = 0.
\end{equation}
From the continuity of $H$ and (\ref{A2E21}) we have for each $\omega \in \Omega$
\begin{equation}\label{A2E22}
\lim_{n} H(\tilde{\mathcal{Q}}^n) = H(\{B_i\}_{i = 1}^k).
\end{equation}
We claim that 
\begin{equation}\label{A2E23}
\begin{split}
&\lim_{n} I_n(\tilde{\mathfrak{Q}}^n) = 0 \mbox{ if } \omega \in \Omega \mbox{ is such that } \{B_i\}_{i = 1}^k \in E_{\operatorname{cross}}, \\
&\lim_{n} I_n(\tilde{\mathfrak{Q}}^n) = 1 \mbox{ if } \omega \in \Omega \mbox{ is such that } \{B_i\}_{i = 1}^k \in E_{\operatorname{avoid}}.
\end{split}
\end{equation}
If (\ref{A2E23}) is true, then from (\ref{A2E20}) we would conclude
$$\lim_{n} I_n(\tilde{\mathfrak{Q}}^n) = I(\{B_i\}_{i = 1}^k) \hspace{3mm} \mbox{ $\mathbb{P}$- a.s.},$$
which together with (\ref{A2E22}) and the bounded convergence theorem would imply (\ref{A2E18}). We have thus reduced the proof of the lemma to establishing (\ref{A2E23}).\\

Suppose first that $\omega \in \Omega$ is such that $\{B_i\}_{i = 1}^k \in E_{\operatorname{cross}}$. Denoting $B_0(s) = f(s)$ and $B_{k+1}(s) = g(s)$, we can find $i_0 \in \llbracket 0, k \rrbracket$, $t_0 \in [a,b]$ and $\epsilon > 0$ (depending on $\omega$) such that
\begin{equation*}
B_{i_0+1}(t_0) > B_{i_0}(t_0) + 2\epsilon.
\end{equation*}
Using the continuity of $B_i$ for $i \in \llbracket 0, k +1 \rrbracket$, we can find $\delta > 0$ such that for $x_0,y_0 \in [t_0 - \delta, t_0 + \delta] \cap [a,b]$
\begin{equation*}
B_{i_0+1}(x_0) > B_{i_0}(y_0) + \epsilon.
\end{equation*}
Finally, using (\ref{A2E21}) and the fact that $f_n, g_n$ converge to $f,g$ uniformly, we have for all large $n$, depending on $\omega$, and $x_0,y_0 \in [t_0 - \delta, t_0 + \delta] \cap [a,b]$ that
\begin{equation}\label{A2E24}
\tilde{\mathcal{Q}}^n_{i_0+1}(x_0) > \tilde{\mathcal{Q}}^n_{i_0}(y_0),
\end{equation}
where as earlier in Step 2 we have set $\tilde{\mathcal{Q}}^n_{0}(t) = f_n(t)$, and $\tilde{\mathcal{Q}}^n_{k+1}(t) = g_n(t)$. From equation (\ref{A2E24}) we conclude that if $n$ is large enough so that $3d_n^{-1} < \delta$ we can find $r \in \mathbb{Z}$ such that $r d_n^{-1}, (r-1)d_n^{-1} \in [t_0 - \delta, t_0 + \delta] \cap [a,b]$ and
\begin{equation*}
\begin{split}
&\tilde{Q}^n_{i_0 +1}(r) = \sigma d_n^{1/2} \tilde{\mathcal{Q}}^n_{i_0+1}(r d_n^{-1}) + pr >  \sigma d_n^{1/2} \tilde{\mathcal{Q}}^n_{i_0}((r-1) d_n^{-1})  + pr  = \tilde{Q}^n_{i_0 }(r-1) + p,
\end{split}
\end{equation*}
where we recall from Step 2 that $\tilde{Q}^n_0(r) = F_n(r)$, $\tilde{Q}_{k+1}^n(r) = G_n(r)$, $\tilde{\mathcal{Q}}^n_{0}(t) = f_n(t)$. The last inequality shows that for all large $n$ the inequality at at $i = i_0$ and $r$ in (\ref{A2E10}) is violated. Consequently, $\tilde{\mathfrak{Q}}^n \not \in \Omega_{\ice}(A_n, B_n, \vec{X}^n, \vec{Y}^n, F_n, G_n)$ for all large $n$, establishing the first line in (\ref{A2E23}).\\

Lastly, suppose that $\omega \in \Omega$ is such that $\{B_i\}_{i = 1}^k \in E_{\operatorname{avoid}}$. We can find $\epsilon > 0$, depending on $\omega$ such that if $u_i\in C([a,b])$ and $\rho(B_i, h_i) < \epsilon$, then 
\begin{equation}\label{A2E25}
f(x) - \epsilon > u_1(x) + \epsilon > u_1(x) - \epsilon > \cdots > u_k(x) + \epsilon > u_k(x) - \epsilon > g(x) + \epsilon. 
\end{equation}
From (\ref{A2E21}) we know that $\tilde{\mathfrak{Q}}^n \in E_n$ for all large $n$, where $E_n$ is as in (\ref{A2E8}) with $h_i$ replaced with $B_i$ for $i \in \llbracket 1,k \rrbracket$. From (\ref{A2E9}) we conclude that $\tilde{\mathfrak{Q}}^n \in E_n \subseteq \Omega_{\ice}(A_n, B_n, \vec{X}^n, \vec{Y}^n, F_n, G_n) $ for all large $n$, which proves the second line in (\ref{A2E23}) and hence the lemma.
\end{proof}

%---------------------------------------------------------------------------------------------------------------
% Appendix A.3
%
%----------------------------------------------------------------------------------------------------------------
\subsection{Proof of Lemmas \ref{Lem.FinEnsSatGB} and \ref{Lem.StrongGP}} \label{AppendixA3} 

\begin{proof}[Proof of Lemma \ref{Lem.FinEnsSatGB}] The fact that $\mathfrak{L}$ is interlacing follows directly from its definition. Suppose that $K = \llbracket k_1, k_2 \rrbracket \subseteq \llbracket 1, k \rrbracket$ and $a,b \in \llbracket T_0, T_1 \rrbracket$ with $a < b$. In addition, suppose that $\tilde{f}, \tilde{g}$ are two increasing paths drawn in $\{ (r,z) \in \mathbb{Z}^2 : a \leq r \leq b\}$ and $\vec{u}, \vec{v} \in \mathfrak{W}_m$ with $m = k_2-k_1+1$ altogether satisfy that $\mathbb{P}(A) > 0$ where $A$ denotes the event 
$$A =\{ \vec{u} = ({L}_{k_1}(a), \dots, {L}_{k_2}(a)), \vec{v} = ({L}_{k_1}(b), \dots, {L}_{k_2}(b)), L_{k_1-1} \llbracket a,b \rrbracket = \tilde{f}, L_{k_2+1} \llbracket a,b \rrbracket = \tilde{g} \}.$$
As usual, if $k_1 = 1$, we adopt the convention $\tilde{f} = \infty = L_0$. Then, we seek to show that for any $\vec{B} = (B_1,\dots, B_m)$, such that $B_i \in \Omega(a, b, u_i , v_i)$ for $i \in \llbracket 1,m\rrbracket$, we have
\begin{equation}\label{Eq.XD1}
\mathbb{P}\left( L_{i + k_1-1}\llbracket a,b \rrbracket = B_{i} \mbox{ for $i \in \llbracket 1, m \rrbracket$} \, \vert  A \, \right) = \mathbb{P}_{\ice, \operatorname{Geom}}^{a,b, \vec{u}, \vec{v}, \tilde{f}, \tilde{g}} \left( \cap_{i = 1}^m\{ Q_i = B_i \} \right).
\end{equation}

Since $\mathfrak{L}$ is interlacing, we conclude for any $\vec{B} \not \in \Omega_{\ice}(a,b, \vec{u},\vec{v},\tilde{f}, \tilde{g})$, that 
\begin{equation}\label{SchurEqRed1}
\mathbb{P}\left( L_{i + k_1-1}\llbracket a,b \rrbracket = B_{i} \mbox{ for $i \in \llbracket 1, k \rrbracket$} \, \vert  A \, \right) = 0.
\end{equation}
Suppose instead that $\vec{B} \in \Omega_{\ice}(a,b, \vec{x},\vec{y},f,g)$. Set $D = \llbracket 1, k+1 \rrbracket \times \llbracket T_0, T_1 \rrbracket \setminus \llbracket k_1, k_2 \rrbracket \times \llbracket a+1, b-1 \rrbracket$, and let $\mathcal{M}$ be the set of vectors $\mu = (\mu_{i,j} \in \mathbb{Z}: (i,j) \in D)$, such that $\mathbb{P}(A \cap \{\mathfrak{L}\vert_{D} = \mu \}) > 0$, where $\mathfrak{L}\vert_{D}$ is the restriction of $\mathfrak{L}$ to $D$. We then have
\begin{equation}\label{SchurEqRed2}
\mathbb{P}\left( L_{i + k_1-1}\llbracket a,b \rrbracket = B_{i} \mbox{ for $i \in \llbracket 1, k \rrbracket$} \, \vert  A \, \right) = \frac{\sum_{\mu \in \mathcal{M}} P(\mu, \vec{B})}{\sum_{\mu \in \mathcal{M}} \sum_{\vec{B}' \in  \Omega_{\ice}(a,b, \vec{u},\vec{v},\tilde{f},\tilde{g})}P(\mu, \vec{B}') }, 
\end{equation}
where 
$$ P(\mu, \vec{B}) =  \mathbb{P}\left( \{L_{i + k_1-1}\llbracket a,b \rrbracket = B_{i} \mbox{ for $i \in \llbracket 1, k \rrbracket$} \} \cap \{\mathfrak{L}_D = \mu \} \right).$$
Each summand in the numerator and denominator in (\ref{SchurEqRed2}) is equal to $|\Omega_{\ice}(T_0, T_1, \vec{x},\vec{y}, \infty, g)|^{-1}$, and so 
\begin{equation}\label{SchurEqRed3}
\mathbb{P}\left( L_{i + k_1-1}\llbracket a,b \rrbracket = B_{i} \mbox{ for $i \in \llbracket 1, k \rrbracket$} \, \vert  A \, \right) = \frac{1}{|\Omega_{\ice}(a,b, \vec{u},\vec{v},\tilde{f},\tilde{g})|}.
\end{equation}
Equations (\ref{SchurEqRed1}) and (\ref{SchurEqRed3}) imply (\ref{Eq.XD1}).
\end{proof}

\begin{proof}[Proof of Lemma \ref{Lem.StrongGP}] As $\mathfrak{L}$ is interlacing, we see that (\ref{SchurEqV2}) holds trivially if $\vec{B} \not \in \Omega_{\ice}(a,b, \vec{x},\vec{y},f, g)$, as both sides are equal to zero. If $\vec{B} \in \Omega_{\ice}(a,b, \vec{x},\vec{y},f, g)$, then (\ref{SchurEqV2}) is equivalent to showing that
\begin{equation}\label{SchurEqV2Red1}
{\bf 1}_{A} \cdot \mathbb{P}\left( L_{i + k_1-1}\llbracket a,b \rrbracket = B_{i} \mbox{ for $i \in \llbracket 1, k \rrbracket$} \, \vert  \mathcal{F}_{\mathrm{ext}} \right) = {\bf 1}_{A} \cdot \frac{1}{|\Omega_{\ice}(a,b, \vec{x},\vec{y},f,g)|}.
\end{equation}
Fix a finite $M \geq k_2+1$, and set $D = \llbracket 1, M \rrbracket \times \llbracket T_0, T_1 \rrbracket \setminus \llbracket k_1, k_2 \rrbracket \times \llbracket a+1, b-1 \rrbracket$. By the defining property of conditional expectation, and the $\pi-\lambda$ theorem, it suffices to show that for any $\mu = (\mu_{i,j}: (i,j) \in D)$, such that $\mathbb{P}(A \cap \{\mathfrak{L}\vert_{D} = \mu \}) > 0$, we have 
\begin{equation}\label{SchurEqV2Red2}
 \mathbb{P}\left(A \cap \{ L_{i + k_1-1}\llbracket a,b \rrbracket = B_{i} \mbox{ for $i \in \llbracket 1, k \rrbracket$} \} \cap \{\mathfrak{L}\vert_{D} = \mu \}  \right) =  \frac{\mathbb{P}(A \cap \{\mathfrak{L}\vert_{D} = \mu \} ) }{|\Omega_{\ice}(a,b, \vec{x},\vec{y},f,g)|}.
\end{equation}
Here, $\mathfrak{L}\vert_{D}$ is the restriction of $\mathfrak{L}$ to $D$.

Let $\mathcal{M}$ be the set of triplets $(\vec{u}, \vec{v}, \tilde{g})$, with $\vec{u}, \vec{v} \in \mathfrak{W}_{M-1}$, and $\tilde{g}: \llbracket T_0, T_1 \rrbracket \rightarrow \mathbb{Z}$ an increasing path, that satisfy
\begin{equation*}
\mathbb{P}(\{ \vec{L}(T_0) = \vec{u}, \vec{L}(T_1) = \vec{v}, L_{M}\llbracket T_0, T_1 \rrbracket = \tilde{g} \} \cap A ) > 0,
\end{equation*}
where $\vec{L}(s) = (L_1(s), \dots, L_{M-1}(s))$. As $\mathfrak{L}$ satisfies the interlacing Gibbs property, we have
\begin{equation*}
\begin{split}
&\mathbb{P}\left(A \cap \{ L_{i + k_1-1}\llbracket a,b \rrbracket = B_{i} \mbox{ for $i \in \llbracket 1, k \rrbracket$} \} \cap \{\mathfrak{L}\vert_{D} = \mu \}  \right) \\
&= \sum_{(\vec{u}, \vec{v}, \tilde{g}) \in \mathcal{M}} \frac{\mathbb{P}( \vec{L}(T_0) = \vec{u}, \vec{L}(T_1) = \vec{v}, L_{M}\llbracket T_0, T_1 \rrbracket = \tilde{g}  )}{|\Omega_{\ice}(T_0,T_1, \vec{u},\vec{v},\infty,\tilde{g})|},
\end{split}
\end{equation*}
and also 
\begin{equation*}
\begin{split}
&\mathbb{P}(A \cap \{\mathfrak{L}\vert_{D} = \mu \} ) =  \sum_{\vec{B}' \in \Omega_{\ice}(a,b, \vec{x},\vec{y},f,g)} \mathbb{P}\left(A \cap \{ L_{i + k_1-1}\llbracket a,b \rrbracket = B'_{i} \mbox{ for $i \in \llbracket 1, k \rrbracket$} \} \cap \{\mathfrak{L}\vert_{D} = \mu \}  \right) \\
&=  \sum_{\vec{B}' \in \Omega_{\ice}(a,b, \vec{x},\vec{y},f,g)} \sum_{(\vec{u}, \vec{v}, \tilde{g}) \in \mathcal{M}} \frac{\mathbb{P}( \vec{L}(T_0) = \vec{u}, \vec{L}(T_1) = \vec{v}, L_{M}\llbracket T_0, T_1 \rrbracket = \tilde{g} )}{|\Omega_{\ice}(T_0,T_1, \vec{u},\vec{v},\infty,\tilde{g})|},
\end{split}
\end{equation*}
The last two displayed equations imply (\ref{SchurEqV2Red2}).
\end{proof}

\end{appendix}

\bibliographystyle{alpha}
\bibliography{PD}

\end{document}